\documentclass[11pt]{amsart}
\numberwithin{equation}{section}
\usepackage[english]{babel}
\usepackage[T1]{fontenc}
\usepackage[latin1]{inputenc}
\usepackage{indentfirst}
\usepackage{enumitem}
\usepackage{amsmath,amssymb, amsbsy}
\usepackage{comment}
\usepackage{amsfonts}
\usepackage{hyperref}
\usepackage{cleveref}
\usepackage{esint}
\usepackage{latexsym}
\usepackage{amsthm}
\usepackage[dvips]{graphicx}
\usepackage{xcolor}
\usepackage{tikz}
\usepackage{pgfplots}
\usetikzlibrary{intersections}
\usepackage{tikz-3dplot}
\usepackage[outline]{contour}
\DeclareGraphicsExtensions{.pdf,.png,.jpg,.eps}

\usepackage{amsmath}

\usepackage{bbm}

\usepackage[font=small,labelfont=bf]{caption}

\usepackage{amsthm}

\usepackage{hyperref}

\usepackage{xcolor}
\usepackage[latin1]{inputenc}
\usepackage[active]{srcltx}
\definecolor{Arancio}{cmyk}{0,0.61,0.87,0}

\definecolor{blus}{RGB}{0,102,204}

\usepackage{doi}
\newcommand{\arxiv}[1]{arXiv:\href{https://arxiv.org/abs/#1}{#1}}
\let\arXiv\arxiv

\newcommand{\brd}[1]{\mathbb{#1}}
\newcommand{\R}{\brd{R}}

\newcommand{\N}{\brd{N}}

\newcommand{\e}{\varepsilon}

\newcommand{\eps}{\varepsilon}
\newcommand{\be}{\begin{equation}}
\newcommand{\ee}{\end{equation}}

\newtheorem{teo}{Theorem}[section]
\newtheorem{Corollary}[teo]{Corollary}
\newtheorem{Lemma}[teo]{Lemma}
\newtheorem{Theorem}[teo]{Theorem}
\newtheorem{Proposition}[teo]{Proposition}
\theoremstyle{definition}

\newtheorem{remark}[teo]{Remark}

\newtheorem{ass}[teo]{Assumption}

\newcommand{\supp}{\operatorname{supp}}

\newcommand{\D}{\nabla}
\newcommand{\dive}{\operatorname{div}}

\renewcommand{\div}{\operatorname{div}}

\renewcommand{\S}{\mathbb{S}}
\newcommand{\loc}{{\rm loc}}

\usepackage{amsmath}
\usepackage{tikz}
\usepackage{mathdots}
\usepackage{yhmath}
\usepackage{cancel}
\usepackage{color}
\usepackage{siunitx}
\usepackage{array}
\usepackage{multirow}
\usepackage{amssymb}
\usepackage{textcomp}
\usepackage{gensymb}
\usepackage{tabularx}
\usepackage{extarrows}
\usepackage{booktabs}
\usetikzlibrary{fadings}
\usetikzlibrary{patterns}
\usetikzlibrary{shadows.blur}
\usetikzlibrary{shapes}

\pgfplotsset{compat=1.18} 
\begin{document}

\subjclass[2020] {35B65, 35J70, 35J75, 35B40, 35B44, 35B45, 35B53}

\keywords{Weighted elliptic equations; Degenerate ellipticity; Monomial weights; Schauder regularity estimates; Liouville Theorems.}

\title[Regularity for elliptic equations with monomial weights]{Regularity for elliptic equations with monomial weights}

\author{Gabriele Cora, Gabriele Fioravanti, Francesco Pagliarin, Stefano Vita}

\address{Gabriele Cora, D\'epartement de Math\'ematique, Universit\'e Libre de Bruxelles, Boulevard du Triomphe 155, 1050, Brussels, Belgium}
\email{gabriele.cora@ulb.be}

\address{Gabriele Fioravanti, Dipartimento di Matematica "G. Peano", Universit\`a degli Studi di Torino, Via Carlo Alberto 10, 10124, Torino, Italy}
\email{gabriele.fioravanti@unito.it}

\address{Francesco Pagliarin, Institut Camille Jordan, Universit\'e Claude Bernard Lyon 1, 43 boulevard du 11 Novembre 1918
69622 Villeurbanne cedex, France}
\email{pagliarin@math.univ-lyon1.fr}

\address{Stefano Vita, Dipartimento di Matematica "F. Casorati", Universit\`a di Pavia, Via Ferrata 5, 27100, Pavia, Italy}
\email{stefano.vita@unipv.it}

\begin{abstract}
We study regularity properties for solutions to elliptic equations that are degenerate or singular along orthogonal hyperplanes. The degenerate ellipticity is carried out by a weight term which is the monomial product of different powers of the distance functions to each hyperplane; that is, given the space dimension $d\geq2$, the number of orthogonally crossing hyperplanes $1\leq n\leq d$ and the generic variable point $z=(x,y)\in\mathbb R^{d-n}\times\mathbb R^n$, then the weight is given by $\omega(y)=\prod_{i=1}^ny_i^{a_i}$ with $a_i>-1$, $y_i=\mathrm{dist}(z,\Sigma_i)$ and $\Sigma_i=\{y_i=0\}$. We prove $C^{0,\alpha}$ and $C^{1,\alpha}$ estimates up to the corners formed by the intersections of two or more hyperplanes, for solutions of the conormal problem with variable coefficients. This is done by a regularization-approximation procedure, a blow-up argument and Liouville theorems. Finally, we provide smoothness of solutions when the equation is isotropic and homogeneous, and we show an application to Caffarelli-Kohn-Nirenberg inequalities with monomial weights.

\end{abstract}

\maketitle


\section{Introduction}

Aim of the paper is to start investigating regularity features of solutions to elliptic equations which are degenerate on multiple orthogonally crossing hyperplanes. The degeneration is carried out by some monomial weights, already object of interest in the literature in connection to particular isoperimetric inequalities, see e.g. \cite{CabRos13a,CabRos13b,CabRosSer16,CinGlaPraRosSer22}.

The local regularity theory for weighted degenerate equations starts with the seminal work \cite{FabKenSer82}, where the De Giorgi-Nash-Moser theory is proved for equations with weights emerging from quasi-conformal mapping or belonging to the celebrated $\mathcal A_2$-Muckenhoupt class.

Then, operators that are degenerate on a single hyperplane (or on a curved hypersurface) have profound connections with the edge calculus \cite{Maz91,MazVer14}, the realization of Dirichlet-to-Neumann maps and fractional operators \cite{CafSil07,CafSti16}, also in a conformal geometric setting \cite{ChaGon11,GonSae18}.

The paper can be seen as a generalization to the case of several hyperplanes of the regularity theory in \cite{SirTerVit21a,SirTerVit21b,TerTorVit24,JeoVit24,DonJeoVit24}, see also \cite{AudFioVit24,AudFioVit25,DonJeo25a,DonJeo25b} for the parabolic counterpart. We would like to mention also \cite{DonPha21,DonPha23,DonPhaSir24,HanXie24a,HanXie24b} and many references therein.

Let $d\geq2$ and $1\leq n\leq d$ be two integers. Consider the coordinates $z=(x,y)\in \R^{d-n}\times\R^n$, and define
\[
\R^d_* =\R^d_*(n):= \bigcap_{i=1}^n\{z \in \R^d \mid y_i > 0\}\,.
\]
This set can be written as $\R^d_* =\R^{d-n}\times(0,\infty)^n$, where $(0,\infty)^n$ is the \emph{positive orthant} in $\R^n$. 
With a slight abuse of notation, we will refer to $\R^d_*$ as the $n$-orthant in $\R^d$. Note that $\R^d$ is partitioned into $2^n$ adjacent copies of $\R^d_*$, corresponding to all combinations of signs in the $y$ coordinates.

Then, in $\R^d_*$ we consider the unit ball centered at the origin $B_1^*:=B_1\cap \R^d_*$. This domain presents some corner-type singularities wherever at least two orthogonal hyperplanes $\Sigma_i:=\{y_i=0\}$ intersect.

Consider the vector $a=(a_1,\dots,a_n)\in\R^n$ and for $z\in B_1^*$ the monomial weight
\begin{equation*}
    \omega^a(y):=\prod_{i=1}^ny_i^{a_i},
\end{equation*}
which may be degenerate or singular on each $\Sigma_i$ depending on the sign of its powers $a_i$. Notice that $\omega^a$ may not fall into the $\mathcal A_2$-Muckenhoupt class. The latter happens as soon as some power $a_i\notin(-1,1)$.

We aim to establish regularity results for weak solutions of the conormal derivative problem
\begin{equation}\label{eq:degenerate:equation:B*}
\begin{cases}
-\dive(\omega^a A\D u) = \omega^a f+\dive(\omega^a F), & \text{in }B_1^* \\
\omega^a(A\D u +F)\cdot e_{y_i} = 0,  & \text{on }\partial B^*_1 \cap \Sigma_i, \text{ for } i=1,\dots,n,
\end{cases}
\end{equation}
namely, functions $u\in H^{1}(B_1^*,\omega^a(y)dz)$ satisfying
\[
\int_{B_1^*}\omega^a A\D u\cdot \D \phi\, dz = \int_{B_1^*}\omega^a (f\phi - F \cdot \D \phi)\, dz, \qquad \text{for every } \phi \in C_c^\infty(B_1).
\]

Additionally, we are interested in determining the necessary structural assumptions on the variable coefficient matrix in order to \emph{preserve} the orthogonality structure of the corner domain. In other words, we would like to provide macroscopic minimal assumptions that ensure a local description of solutions as polynomials.

\begin{ass}\label{Ass:matrix}
   Let $A$ be a matrix valued function, written in block form as
$$A= \left(\begin{array}{cc}
       P & Q  \\ 
       R & S \\ 
  \end{array}\right),$$
  where $P$ is $(d-n)\times (d- n)$-dimensional, $Q$ is $(d-n)\times n$-dimensional, $R$ is $n\times (d-n)$-dimensional, $S$ is $n\times n$-dimensional. We make the following assumptions:
\begin{enumerate}
\item [$i)$] the matrix $A$ satisfies the \emph{uniform ellipticity condition}; that is, there exist $0<\lambda\le\Lambda$ such that
\[
\|A\|_{L^\infty}\le \Lambda,\quad  A(z)\xi\cdot \xi \ge \lambda|\xi|^2,\quad \text{ for a.e. }z, \text{ for any }\xi\in\R^d;
\]

\item [$ii)$] the block $S$ satisfies an \emph{orthogonality condition} on the intersection of hyperplanes
\begin{equation*}
   (S_{i,j})_{{|}_{\Sigma_i\cap\Sigma_j}}= (S_{j,i})_{{|}_{\Sigma_i\cap\Sigma_j}} = 0, \qquad\text{for }i,j=1,\dots,n, \ i\not=j;
\end{equation*}

      \item [$iii)$] if $n\geq2$, the blocks $Q, R$ satisfy a \emph{symmetry condition} on the hyperplanes:
\[
Q_{j,i}=R_{i,j},\quad \text { on } \Sigma_i, \ \text{ for every } i = 1,\dots,n , \ j=1,\dots,d-n.
\]
  \end{enumerate}
\end{ass}

We note that it remains unclear whether the symmetry assumption in $iii)$ can be removed, although we believe that all subsequent results should still hold without it.

For a vector $a\in\R^n$, we define its positive part as $a^+ \in [0, \infty)^n$ by $(a^+)_i = (a_i)^+= \max\{a_i,0\}\,$. For any $a \in \R^n$, we define $\langle a \rangle = \sum_{i=1}^{n} a_i$. The following theorem provides the main result of this work.

\begin{Theorem} [Regularity for elliptic equations degenerating along orthogonal hyperplanes]\label{thm1}
    Let  $u$ be a weak solution to \eqref{eq:degenerate:equation:B*}, where the matrix $A$ satisfies Assumption \ref{Ass:matrix}. Then the following statements hold.
    \begin{enumerate}
        \item [$i)$] Let $p>\frac{d+\langle a^+\rangle}{2}$, $q>d+\langle a^+\rangle$ and $\alpha \in (0,2-\frac{d+\langle a^+\rangle}{p}]\cap(0,1-\frac{d+\langle a^+\rangle}{q}]$. Let us suppose that $f\in L^p(B_1^*,\omega^a)$, $F\in (L^q(B_1^*,\omega^a))^d$ and that $A$ admits a modulus of continuity $\sigma$ such that
        \[
        \|A\|_{C^{0,\sigma(\cdot)}(B_1^*)}:= \|A\|_{L^\infty(B_1^*)} + \sup_{x,y\in B_1^*}\frac{|A(x)-A(y)|}{\sigma(|x-y|)}\le L.
        \]       
        Then $u \in C^{0,\alpha}(B_r^*)$ for any $r\in(0,1)$. Moreover, there exists a constant $C>0$ depending only on $d,\lambda,\Lambda,p,q,a,\alpha,L$ such that
        \[
            \|u\|_{C^{0,\alpha}(B_{1/2}^*)}\le C \big(
            \|u\|_{L^2(B_1^*,\omega^a)}+ \|f\|_{L^p(B_1^*,\omega^a)} +\|F\|_{L^q(B_1^*,\omega^a)}
            \big).
        \]
        \item [$ii)$] Let $p>d+\langle a^+\rangle$, and $\alpha \in (0,1-\frac{d+\langle a^+\rangle}{p}]$. Let us suppose that $f\in L^p(B_1^*,\omega^a)$, $F\in C^{0,\alpha}(B_1^*)$ and $\|A\|_{C^{0,\alpha}(B_1^*)}\le L$.   
        Then $u \in C^{1,\alpha}(B_r^*)$ for any $r\in(0,1)$. Moreover, there exists a constant $C>0$ depending only on $d,\lambda,\Lambda,p,a,\alpha,L$ such that
        \begin{equation}\label{eq:1alpha}
            \|u\|_{C^{1,\alpha}(B_{1/2}^*)}\le C \big(
            \|u\|_{L^2(B_1^*,\omega^a)}+ \|f\|_{L^p(B_1^*,\omega^a)} +\|F\|_{C^{0,\alpha}(B_1^*)}
            \big).
        \end{equation}    
        Furthermore, $u$ satisfies pointwise the boundary condition
        \begin{equation}\label{eq:BC:thm1}
            (A\D u +F)\cdot e_{y_i} = 0,  \quad \text{on }\partial B^*_r \cap \Sigma_i, \, \text{ for every } i=1,\dots,n,
        \end{equation}
    \end{enumerate}
\end{Theorem}

The result is proved through several steps:

\begin{enumerate}
    \item[$i)$] Proof of Theorem \ref{thm1} when $a=0$. The proof is done by induction on the number of hyperplanes $n$ with $1\leq n\leq d$ and is based on a priori estimates, freezing of the variable coefficients along intersections of hyperplanes, and an approximation argument. See Section \ref{S:unif:ell:reg}.
    
    \item[$ii)$] Proof of uniform regularity estimates with respect to families of regularized problems. Given $\eps>0$, the weight $\omega^a$ is regularized as
    \begin{equation*}     \omega^a_\varepsilon(y):=\prod_{i=1}^n\rho_{\varepsilon}(y_i)^{a_i},\qquad \rho_{\varepsilon}(y_i)=(\varepsilon^2+y_i^2)^{1/2}.
    \end{equation*}
    Uniform estimates are proved by blow-up procedures in the spirit of \cite{Sim97}, and by a Liouville theorem (see Theorem \ref{T:Liouville}). See Section \ref{S:liouville} and Section \ref{S:regularity}.
    
    \item[$iii)$] Approximation of the limiting problem with solutions to the regularized ones. See Section \ref{S:solution}.
\end{enumerate}

Let us remark that the $C^{0,\alpha}$ estimate in Theorem \ref{thm1}, $i)$ requires the continuity of the variable coefficients and comes with an explicit exponent $\alpha$, which is arbitrarily close to $1$ in case of zero right hand sides. This is a stronger result compared to the $C^{0,\alpha}$ estimate (for some implicit $\alpha$) coming from the De Giorgi-Nash-Moser theory in the presence of only $L^\infty$ variable coefficients. The latter result is provided in Section \ref{s:degiorgi} by reframing our weighted Sobolev setting into the theory of $2$-admissible weights; see \cite{FabKenSer82,HeiKilMar06}.

Once the main result has been established, we devote Section \ref{S:CKN} to the study of higher order regularity in some simple cases. 

First, in connection with a Caffarelli--Kohn--Nirenberg inequality \cite{CafKohNir85} with monomial weights, we provide a $C^3$-regularity result under additional structural assumptions (see Theorem \ref{thm7.1} and Corollary \ref{corapplCKN}).
In \cite{Paglia25}, the third author provides, under curvature bounds, the optimal constant $C_{opt}$ of the following monomial CKN inequality
\begin{equation}\label{eqmonCKN}
\Big(\int_{\R^d_*}\omega^a \frac{|u|^p}{|z|^{bp}}\, dz \Big)^{{2}/{p}}\leq C_{opt}\int_{\R^d_*}\omega^a \frac{|\nabla u|^2}{|z|^{2q}}\,dz,\qquad \forall u \in C^\infty_c(\overline{\R^d_*} \setminus\{0\})
\end{equation}
where $\omega^a=\omega^a(y)$ is the monomial weight defined above, the parameters $a,b,p,q$ satisfy suitable assumptions and $u$ belongs to an appropriate functional space. In that context, in order to take advantage of an integrated curvature dimension condition, some integration by parts concerning solutions to equations \eqref{eq7.2} need to be justified. Since derivatives up to third order are involved, our $C^3$-regularity result is sufficient to make the arguments in \cite{Paglia25} rigorous. 

Finally, we prove $C^\infty$ regularity in the isotropic homogeneous setting; that is, the special case where $A = I_d$ and $f = F =0$ (see Theorem \ref{TheoSmooth}).

\subsection*{Notation}
We establish the notation that will be used throughout the paper.

\begin{itemize}[left=0pt]

    \item[$-$] Let $d\geq 2$ and $1\leq n\leq d$. We always write coordinates in the form $z=(x,y)\in \R^{d-n}\times\R^n \equiv\R^d$.

    \item [$-$] We denote by $\{e_{z_\ell}\}$, where $\ell=1,\dots,d$, the canonical basis of $\R^d$. To distinguish between the variables $x$ and $y$, we will often denote the basis as $\{e_{x_j}\,, e_{y_i}\}$, where $j=1,\dots,d-n$ and $i=1,\dots,n$. 
    
    \item [$-$] For $m \in \N$ we denote by $B_R^{m}(\zeta)$ the open $m$-dimensional ball of radius $R>0$ centered at $\zeta \in \R^m$. To ease the notation, we simply write $B_R^m= B^m_R(0)$ when $\zeta = 0$, and $B_R(\zeta)= B^d_R(\zeta)$, when the dimension $m=d$. In particular, $B_R = B^d_R(0)$. 

\item[$-$] We recall that with $\R^d_*$ we refer to the $n$-orthant in $\R^d$ given by $\R^d_* =\R^{d-n}\times(0,\infty)^n$. Moreover, for $R>0$, we set $B_R^*=B_R\cap \R^d_*$.  

\item[$-$] For each $i = 1, \ldots, n$, we introduce the hyperplane $\Sigma_i = \{z \in \R^d \mid y_i = 0\}$,
and define the characteristic set
\[
\Sigma_0 = \bigcup_{i=1}^n \Sigma_i\,.
\]

\item [$-$]  Throughout the paper, any positive constant whose value is not important is denoted by $c$. It may take different values at different places.

\end{itemize}

\section{Regularity for uniformly elliptic equations in domains with orthogonal corners}\label{S:unif:ell:reg}

This section is devoted to the proof of Theorem \ref{thm1} in the case of uniformly elliptic problems, i.e., when $a=0$. The argument relies on a regularization-approximation scheme, which combines the freezing of variable coefficients along the intersection of orthogonal hyperplanes with suitable a priori estimates. Local regularity is then deduced from a Liouville-type theorem in the orthant for homogeneous conormal boundary value problems.

Given a matrix $A$, a function $f$ and a vector field $F$, a weak solution to the boundary value problem
\begin{equation}\label{eq:uniell:equation:B*}
\begin{cases}
-\dive(A\D u) = f+\dive(F), & \text{ in }B_1^* \\
(A\D u +F)\cdot e_{y_i} = 0,  & \text{ on }\partial B_1^* \cap \Sigma_i, \text{ for } i=1,\dots,n\,,
\end{cases}
\end{equation}
is a function $u\in H^{1}(B_1^*)$ satisfying
\[
\int_{B_1^*}A\D u\cdot \D \phi\, dz = \int_{B_1^*} (f\phi - F \cdot \D \phi)\, dz, \qquad \text{for every } \phi \in C_c^\infty(B_1).
\]
The main result of this section is the following. 
\begin{Theorem} [Regularity in domains with corners]\label{L:schuader:unifell}
    Let  $u$ be a weak solution to \eqref{eq:uniell:equation:B*}, where the matrix $A$ satisfies the Assumption \ref{Ass:matrix}. Then:
    \begin{enumerate}
        \item [$i)$] Let $p>d/2$, $q>d$ and $\alpha \in (0,2-d/p]\cap(0,1-d/q]$. Suppose that $f\in L^p(B_1^*)$, $F\in (L^q(B_1^*))^d$, and that there exists a modulus of continuity $\sigma$ such that
        \[
        \|A\|_{C^{0,\sigma(\cdot)}(B_1^*)}:= \|A\|_{L^\infty(B_1^*)} + \sup_{x,y\in B_1^*}\frac{|A(x)-A(y)|}{\sigma(|x-y|)}\le L.
        \]       
       Then $u \in C^{0,\alpha}(B_r^*)$ for any $r\in(0,1)$, and there exists a constant $c>0$ depending on $d,\lambda,\Lambda,p,q,\alpha,L$ such that
        \begin{equation}\label{eq:unifell:0alpha}
            \|u\|_{C^{0,\alpha}(B_{1/2}^*)}\le c \big(
            \|u\|_{L^2(B_1^*)}+ \|f\|_{L^p(B_1^*)} +\|F\|_{L^q(B_1^*)}
            \big).
        \end{equation}
        \item [$ii)$] Let $p>d$ and $\alpha \in (0,1-d/p]$. Suppose that $f\in L^p(B_1^*)$, $F\in C^{0,\alpha}(B_1^*)$ and that $A \in C^{0,\alpha}(B_1^*)$ is such that $\|A\|_{C^{0,\alpha}(B_1^*)}\le L$.   
        Then $u \in C^{1,\alpha}(B_r^*)$ for any $r\in(0,1)$, and there exists a constant $c>0$ depending on $d,\lambda,\Lambda,p,\alpha,L$ such that
        \begin{equation}\label{eq:unifell:1alpha}
            \|u\|_{C^{1,\alpha}(B_{1/2}^*)}\le c \big(
            \|u\|_{L^2(B_1^*)}+ \|f\|_{L^p(B_1^*)} +\|F\|_{C^{0,\alpha}(B_1^*)}
            \big).
        \end{equation}     
    \end{enumerate}
\end{Theorem}

\begin{remark}\label{rem:even:reflection}
Given a solution to \eqref{eq:uniell:equation:B*} in $B_1^*$, one can perform an even reflection across any $\Sigma_i$, obtaining a solution of a new equation in the full ball $B_1$. However, this procedure may introduce jumps in the coefficients and in the field terms across the hyperplanes.

Let $y = (y', y_n)\in \R^{n-1}\times \R$, and define the symmetrization of $B_1^*$ across the hyperplane $\Sigma_n$ as
\begin{equation}\label{eq:reflectset}
 B_1^{\dagger} := B_1 \cap  \bigcap_{i=1}^{n-1} \{z \in \R^d \mid y_i>0\}\,.
\end{equation}
We define the reflected data as
\[
\begin{aligned}
&A^\dagger(x,y',y_n) = \begin{cases}
    A(x,y',y_n), & \text{ in }  B_1^\dagger \cap \{y_n>0\}\\
    JA(x,y',-y_n)J  & \text{ in }   B_1^\dagger \cap \{y_n<0\}
\end{cases}\\
&f^\dagger(x,y',y_n) = \begin{cases}
    f(x,y',y_n), &\!\!\qquad \text{ in }  B_1^\dagger \cap \{y_n>0\}\\
    f(x,y',-y_n)  & \!\!\qquad\text{ in }  B_1^\dagger \cap \{y_n<0\}
\end{cases}\\
&F^\dagger(x,y',y_n)  = \begin{cases}
    F(x,y',y_n), & \!\!\quad\text{ in }  B_1^\dagger \cap \{y_n>0\}\\
    JF(x,y',-y_n)  & \!\!\quad\text{ in }   B_1^\dagger \cap \{y_n<0\}
    \end{cases}
\end{aligned}
\]
where 
$$
J:=\left(\begin{array}{cc}
       I_{d-1} & 0  \\ 
       0 & -1\\ 
  \end{array}\right)\,.
$$
The even reflection $u^\dagger$ of $u$ across the hyperplane $\Sigma_n$, given by 
\[
u^\dagger(x, y', y_n) = u(x, y', |y_n|)\,,
\]
is a weak solution to the problem
\begin{equation*}
\begin{cases}
-\dive( A^\dagger\D u^\dagger) = f^\dagger+\dive( F^\dagger), & \text{in } B_1^\dagger \\
( A^\dagger\D u^\dagger + F^\dagger)\cdot e_{y_i} = 0,  & \text{on }\partial B_1^\dagger \cap \Sigma_i, \text{ for } i=1,\dots,n-1\,.
\end{cases}
\end{equation*}
If $ A\in C^{0,\alpha}(B_1^*)$, then $A^\dagger \in C^{0,\alpha}(B_1^\dagger)$ if and only if 
\[
A_{j,d}=A_{d,j}=0 \quad \text{ on } \Sigma_n \quad \text{ for all }j=1,\dots,d-1\,,
\]
that is, the last column and row of $A$ vanish on the hyperplane $\Sigma_n$, except possibly for $A_{d,d}$. Similarly, if $F\in C^{0,\alpha}(B_1^*)$ then $F^\dagger \in C^{0,\alpha}(B_1^\dagger)$ if and only if $F\cdot e_{y_n}=0$ on $\Sigma_n$. For related considerations in a similar context, see \cite[Remark 4.4]{AudFioVit24}.
\end{remark}

The first step in the proof of Theorem \ref{L:schuader:unifell} is the following Liouville-type result. The argument follows the same scheme as in Section \ref{S:liouville}; indeed, the statement can be recovered as a special case of Theorem \ref{T:Liouville} by taking $a =0$. For this reason, we omit the proof.
\begin{Theorem} [Liouville Theorem in the orthant]\label{L:Liouville:UE} 
Let $A$ be a constant, uniformly elliptic matrix of the form
\begin{equation*}
    A:= \left(\begin{array}{cc}
       P & Q  \\ 
       R & S\\ 
  \end{array}\right),
\end{equation*}
where $P$ is $(d-n)\times (d- n)$-dimensional, $Q$ is $(d-n)\times n$-dimensional, $R$ is $n\times (d- n)$-dimensional, and $S$ is $n\times n$-dimensional and diagonal. 
Let $u$ be an entire weak solution to 
\begin{equation*}
\begin{cases}
-\dive(A\D u) = 0, & \text{in }\R^d_* \\
A\D u\cdot e_{y_i} = 0,  & \text{on } \partial\R^d_* \cap \Sigma_i, \text{ for } i=1,\dots,n\,,
\end{cases}
\end{equation*}
satisfying the growth condition
\[
|u(z)|\le c(1+|z|^\gamma)
\]
for some constant $c>0$ and $\gamma \in  [0,2)$.  Then $u$ is an affine function.
Moreover, if $\gamma\in [0,1)$, then $u$ is constant.
\end{Theorem}
\begin{remark}
    If the block $S$ is not diagonal, Theorem \ref{L:Liouville:UE} may not hold. Consider for instance the simple case $d=n=2$ (so that, in the notation above, $A = S$), and the boundary value problem 
  \begin{equation}\label{eq:contex}
        \begin{cases}
            -\mathrm{div}(A_\theta\nabla u)=0 &\mathrm{in \ }\R^2_*\\
            A_\theta\nabla u\cdot e_i=0 &\mathrm{on \ }\partial \R^2_* \cap \Sigma_i, \ \mathrm{for \ } i=1,2,
        \end{cases}
   \end{equation}
  where the constant, uniformly elliptic matrix $A_\theta$ is given by
  \[
        A_\theta=\frac{1}{\sin\theta}\left(\begin{array}{cc}
       1 & -\cos\theta  \\ 
       -\cos\theta & 1\\ 
  \end{array}\right)\,, \qquad \ \theta\in(0,\pi)\,.
  \]
The function
    \[
u_\theta(y_1, y_2) = \left( y_1^2 + y_2^2 + 2y_1y_2 \cos\theta \right)^{\frac{\pi}{2\theta}} 
 \cos\left( \frac{\pi}{\theta} \arctan\left( \frac{y_2 \sin\theta}{y_1 + y_2 \cos\theta} \right) \right)
\]
is an entire solution to \eqref{eq:contex}, and satisfies $|u_\theta(z)| \leq |z|^{\pi/\theta}$. However, $u_\theta$ is not affine for any $\theta \in (\pi/2, \pi)$. 
\end{remark}

The second key ingredient in the proof of Theorem \ref{L:schuader:unifell} consists of suitable a priori estimates, stated in the following lemma. The proof is based on a contradiction argument combined with Theorem \ref{L:Liouville:UE}. Since the argument follows the same strategy as in the proofs of Theorems \ref{thmC0stable} and \ref{thmc1stab} -- which are substantially more involved and will be presented in later sections -- we omit the details here.

\begin{Lemma}[A priori estimates in domains with corners]\label{L:schuader:unifell:apriori}
     Let  $u$ be a weak solution to \eqref{eq:uniell:equation:B*}, where the matrix $A$ satisfies the Assumption \ref{Ass:matrix}. Then
    \begin{enumerate}
        \item [$i)$] if the assumptions of Theorem \ref{L:schuader:unifell}, $i)$ hold, and $u\in C^{0,\alpha}(B_1^*)$, then estimate \eqref{eq:unifell:0alpha} holds;
        \item [$ii)$] if the assumptions of Theorem \ref{L:schuader:unifell}, $ii)$ hold, and  $u\in C^{1,\alpha}(B_1^*)$, then estimate \eqref{eq:unifell:1alpha} holds.
    \end{enumerate}
\end{Lemma}

We are now ready to prove Theorem \ref{L:schuader:unifell}.

\begin{proof}[Proof of Theorem \ref{L:schuader:unifell}]
We present only the proof of $ii)$, since the argument for $i)$ is analogous. 
We set $d = k + n$, where $k \geq 0$ is a fixed integer, and proceed by induction on $n$, the number of intersecting hyperplanes.
The base case $n=1$ follows from the classical Schauder boundary regularity theory, since in this case the domain $B_1^*$ is a half-ball. 
Assume now that the statement holds for $n-1$ hyperplanes. We shall prove it for the case of $n$ hyperplanes.

\emph{Step 1. Construction of the approximating solutions}. Fix $\delta\in(0,1/4)$ and define
\[ 
A_\delta(x,y):=A(x,y)+\eta\Big(\frac{|y|}{\delta}\Big)(A(x,0)-A(x,y)), 
\]
where
\[
\eta\in C^\infty([0,\infty)),\quad \eta=1 \text{ in } [0,1), \quad \eta=0 \text{ in } (2,\infty), \quad 0\le\eta\le1.
\]
Then, $A_\delta$ satisfies Assumption \ref{Ass:matrix} with the same ellipticity constants $\lambda, \Lambda$ as $A$. Notice that when $|y| \leq \delta$ then $A_\delta(x, y) = A(x, 0)$. Since $(x, 0) \in \cap_{i=1}^n \Sigma_i$, Assumption \ref{Ass:matrix} ensures that $A_\delta$ satisfies in $B_1\cap\{|y|\le\delta\}$ the following two conditions: $R_\delta=Q_\delta^\top$ and $S_\delta$ is diagonal. Moreover, there exists $c>0$, independent of $\delta$, such that 
\begin{equation}\label{eq:approxreg}
\|A_\delta\|_{C^{0,\alpha}(B_1^*)}\le c \|A\|_{C^{0,\alpha}(B_1^*)}.
\end{equation}
Indeed, for $z=(x,y),\,z'=(x',y') \in B_1^*$ with $|y|\le 2\delta$, a direct computation gives
    \begin{align*}
        |A_\delta(z)-A_\delta(z')| &\le |A(z)-A(z')| 
        + \Big|\eta\Big(\frac{|y|}{\delta}\Big)\Big(A(x,0)-A(z)\Big) - \eta\Big(\frac{|y'|}{\delta}\Big)\Big(A(x',0)-A(z')\Big) \Big|\\
        &\le [A]_{C^{0,\alpha}(B_1^*)}|z-z'|^\alpha 
        +\Big| \eta\Big(\frac{|y'|}{\delta}\Big)\Big|\Big(|
        A(x,0) - A(x', 0)| + |A(z)-A(z')|\Big)\\
        &+\Big| \eta\Big(\frac{|y|}{\delta}\Big)- \eta\Big(\frac{|y'|}{\delta}\Big)\Big|\Big|
        A(x,0)-A(z)
        \Big|\\
        &\le 3[A]_{C^{0,\alpha}(B_1^*)}|z-z'|^\alpha + [\eta]_{C^{0,\alpha}(\R)}\frac{|y-y'|^\alpha}{\delta^\alpha} [A]_{C^{0,\alpha}(B_1^*)}|y|^\alpha \le  c[A]_{C^{0,\alpha}(B_1^*)}|z-z'|^\alpha\,,
    \end{align*}    
and \eqref{eq:approxreg} follows. 

Next, let $\rho_\delta=\rho_\delta(x)$ be a standard mollifier supported in $\{|x|\le \delta\}$, acting only on the 
$x$-variable. Define 
\[
\bar A_\delta(x,y):= A_\delta(x,y) \ast \rho_\delta(x)\quad \text { and } \quad \bar F_\delta(x,y):=F(x,y) \ast \rho_\delta(x).
\] 
Also in this case one can verify that $\bar A_\delta$ satisfies Assumption \ref{Ass:matrix} with the same ellipticity constants $\lambda, \Lambda$ as $A$, and $\|\bar A_\delta\|_{C^{0,\alpha}(B_{3/4}^*)}\le c \|A\|_{C^{0,\alpha}(B_1^*)}$. Similarly, the field $\bar F_\delta$ satisfies $\|\bar F_\delta\|_{C^{0,\alpha}(B_{3/4}^*)}\le c \|F\|_{C^{0,\alpha}(B_1^*)}$. 

Let now $u$ be a weak solution to \eqref{eq:uniell:equation:B*}, and denote by $u_\delta$ the unique weak solution to
\[
\begin{cases}
-\dive(\bar A_\delta\D u_\delta) = f+\dive(\bar F_\delta), & \text{in }B_{3/4}^* \\
(\bar A_\delta\D u_\delta +\bar F_\delta)\cdot e_{y_i} = 0,  & \text{on }\partial B^*_{3/4} \cap \Sigma_i,  \text{ for } i=1,\dots,n, \\
u_\delta = u, & \text{on }\partial B^*_{3/4} \cap \R^{k+n}_*, 
\end{cases}
\]
It follows from standard arguments that 
\begin{equation}\label{eq:u:delta:unifell:regularity}
    u_\delta \to u \ \text{ in } H^1(B^*_{1/2}),\quad \text{ and }\quad \|u_\delta\|_{H^1(B^*_{1/2})} \le c(\|u\|_{L^2(B^*_1)}+\|f\|_{L^2(B^*_1)}+\|F\|_{L^2(B^*_1)}),
\end{equation}
for some constant $c>0$ which does not depend on $\delta$.

\emph{Step 2. Regularity of the approximating solutions.} We claim that $u_\delta \in C^{1,\alpha}(B_{1/2}^*)$.  By the inductive assumption, we already know that $u_\delta \in C_\loc^{1,\alpha}(B_{3/4}^*\setminus \{y=0\})$. Hence, it suffices to show that 
\[
u_\delta \in C^{1,\alpha}(B_{1/2}^* \cap \{|y| \leq \bar r\})\quad  \text{ for some }\bar r>0.
\]

First we recall that, by construction, in $B_{3/4}^* \cap \{|y| \leq \delta\}$ the matrix $\bar A_\delta$ satisfies Assumption \ref{Ass:matrix} and it depends only on the $x$-variables (and in fact $\bar A_\delta \in C^{\infty}(B_{3/4}^*\cap \{|y| \leq \delta\})$).  Moreover, it can be written in block form as
$$\bar A_\delta = \left(\begin{array}{cc}
       \bar P_\delta & \bar Q_\delta  \\ 
       \bar Q^\top_\delta  & \bar S_\delta \\ 
  \end{array}\right),$$
where $\bar P_\delta$ is $k \times k$-dimensional, $\bar Q_\delta$ is $k\times n$-dimensional, and $\bar S_\delta$ is $n\times n$-dimensional and diagonal (noting that the structural conditions follow from the fact that $A_\delta$ satisfies them in $B_1\cap \{|y|\le \delta\}$).
Since in the following we restrict ourselves to $B_{3/4}^*\cap \{|y| \leq \delta\}$, to simplify the notation we simply write $\bar A_\delta = \bar A_\delta(x)$.
  
Next, we introduce an appropriate change of variables that removes the last column of the off-diagonal block $\bar Q_\delta$ when evaluated on $\Sigma_n$. Define the vector field 
\[
b_\delta(x):=\bar Q_\delta(x) \bar S_\delta^{-1}(x)e_{y_n}\,,
\]
and consider the smooth map
\[
{\Phi_\delta}:B^*_{3/4}\cap \{|y| \leq \delta\} \to \R^d\,, \qquad \Phi_\delta(x,y):=(x+ y_nb_\delta(x) ,y)\,.
\]
Its Jacobian matrix is 
\[
J_{\Phi_\delta}(x,y) = \left(
\begin{array}{cc}
    I_{k}+y_n J_{b^\delta}(x) & \mathcal{B}_\delta(x)  \\
   0  & I_n
\end{array}\right), \quad 
 \text{ where }  \quad  
\mathcal{B}_\delta(x):= (\,0 \ |\  0\ | \ldots \ | \ 0 \ |\ b_\delta(x)\,)\,.
\]
Since $\det J_{{\Phi_\delta}} = \det(I_{k}+y_n J_{b^\delta})$, it follows that $1/2\le |\det J_{{\Phi_\delta}}|\le 3/2$ whenever $|y_n|$ is sufficiently small. Thus, for some $0 < r \leq \delta$, the map ${\Phi_\delta} \in C^{\infty}(B_{3/4}^*\cap\{|y|\le r\})$ is a diffeomorphism.

For $0 < \rho \leq r$, set
\[
\mathcal C^*_\rho := B^*_{1/2}\cap \{|y|\le \rho\} \subset B_{3/4}^*\cap\{|y|\le r\}
\]
and choose $\rho >0$ such that 
\[
\Phi_\delta(\mathcal C^*_\rho) \subset B^*_{3/4} \cap \{|y| \leq r\}.
\]
On $\mathcal C^*_\rho$, define
\[
\hat{A}_\delta:=|\det J_{{\Phi_\delta}}|J_{{\Phi_\delta}}^{-1} \bar A_\delta \circ \Phi_\delta (J_{{\Phi_\delta}}^{-1})^T,\qquad \hat f_\delta:=|\det J_{{\Phi_\delta}}|f\circ {\Phi_\delta},\qquad \hat F_\delta := |\det J_{{\Phi_\delta}}|J_{{\Phi_\delta}}^{-1}  \bar F_\delta\circ {\Phi_\delta}.
\]
Since $u_\delta$ weakly solves 
\[
\begin{cases}
-\dive(\bar A_\delta\D u_\delta) = f+\dive(\bar F_\delta), & \text{in }B_{3/4}^*\cap\{|y|\le r\} \\
(\bar A_\delta\D u_\delta +\bar F_\delta)\cdot e_{y_i} = 0,  & \text{on }\partial (B_{3/4}^*\cap\{|y|\le r\}) \cap \Sigma_i,  \text{ for } i=1,\dots,n\,,
\end{cases}
\]
a direct computation shows that $\hat u_\delta := u_\delta \circ {\Phi_\delta} $ weakly satisfies
\[
\begin{cases}
-\dive(\hat A_\delta\D \hat u_\delta) = \hat f_\delta+\dive(\hat F_\delta), & \text{in }\mathcal{C}_{\rho}^* \\
(\hat A_\delta\D \hat u_\delta +\hat F_\delta)\cdot e_{y_i} = 0,  & \text{on }\partial \mathcal{C}^*_{\rho} \cap \Sigma_i,  \text{ for } i=1,\dots,n.
\end{cases}
\]
Moreover, using that
\[
J_{\Phi_\delta}^{-1} =
\begin{pmatrix}
(I_k + y_n J_{b^\delta}(x))^{-1} & - (I_k + y_n J_{b^\delta}(x))^{-1}\mathcal{B}_\delta(x)\\
0 & I_d
\end{pmatrix}\,,
\]
one readily verifies that, in $\mathcal{C}^*_\rho$, the matrix $\hat A_\delta$ can be expressed  in block form as
\[
\hat A_\delta = \begin{pmatrix}
\hat P_\delta  & \hat Q_\delta\\
\hat Q_\delta^\top & \hat S_\delta
\end{pmatrix}\,,
\]
where $\hat S_\delta$ is diagonal and the last column of $\hat Q_\delta$ is identically zero when $y_n = 0$. 

Further, define
\[
g_\delta(x, y):= (\bar S_\delta^{-1})_{n,n}(x)(\hat F_\delta(x, y_1, \ldots, y_{n-1},0)\cdot e_{y_n})y_n, \quad \text{ and } \quad \hat G^\delta := \hat F_\delta-\hat A^\delta\D g_\delta\,.
\]
Then $v_\delta:=\hat{u}_\delta + g_\delta$ is a weak solution to 
\[
\begin{cases}
-\dive(\hat A_\delta\D v_\delta) = \hat f_\delta+\dive\hat G_\delta, & \text{in }\mathcal{C}_{\rho}^* \\
(\hat A_\delta\D v_\delta +\hat G_\delta)\cdot e_{y_i} = 0,  & \text{on }\partial \mathcal{C}^*_{\rho} \cap \Sigma_i,  \text{ for } i=1,\dots,n,
\end{cases}
\]
with the additional property $\hat G_\delta\cdot e_{y_n}=0$ on $\Sigma_n$. As a consequence, thanks to Remark \ref{rem:even:reflection}, we can reflect $v_\delta$ evenly in $y_n$, obtaining a solution of an equation in $B^\dagger_{1/2}\cap\{|y|\leq \rho\}$ (see \eqref{eq:reflectset}). By construction, the regularity of both $\hat{A}_\delta$ and $\hat{G}_\delta$ is preserved under this reflection. Thus, by the inductive step we have $v_\delta \in C^{1,\alpha}(B_{1/2}^\dagger\cap \{|y|\leq \rho\})$ and then, composing back with the diffeomorphism $\Phi_\delta$, it follows that $u_\delta \in C^{1,\alpha}(B^*_{1/2} \cap \{|y| \leq \bar r\})$, for some $\bar r=\bar r(\delta)>0$. This completes the proof of the claim.

\emph{Step 3. Conclusion.} Combining the regularity of $u_\delta$ established in \emph{Step 2} with Lemma \ref{L:schuader:unifell:apriori} and \eqref{eq:u:delta:unifell:regularity}, the desired result follows by a standard compactness argument using the Arzel\'a-Ascoli Theorem.
\end{proof}

\section{Functional framework}\label{S:ff}

Let $a  \in (-1, \infty)^n$ and $\eps \in [0,1]^n$. For each $i=1, \ldots, n$, define 
\[
\rho_{\eps_i}(y_i) = (\eps_i^2 + y_i^2)^{1/2}\,.
\]
We introduce the weight function
\begin{equation}\label{eq:regweight}
\omega^a_\eps(y) = \prod_{i=1}^n \rho_{\eps_i}^{a_i}(y_i)\,,
\end{equation}
which satisfies $\omega_\eps^a \in L^1_\loc(\R^d)$. To simplify the notation, we write $\omega^a = \omega^a_0$ whenever $\varepsilon = 0$. 

The aim of this section is to introduce the Sobolev spaces naturally associated with the weight $\omega^a_\eps$. We provide both the definition via density and the definition via weak derivatives, and we show that these two definitions coincide. Finally, we establish several inequalities that hold in these spaces and that are stable with respect to $\eps$. Specifically, we prove a trace inequality on the boundary of balls, a weighted Sobolev inequality, and a weighted Hardy inequality.

\subsection{Weighted Lebesgue and Sobolev spaces}

Let $\Omega \subseteq \R^d$ be a bounded Lipschitz domain. Given $p \geq 1$, we define 
\[
L^{p,a, \eps}(\Omega) := L^p(\Omega, \omega^a_\eps(y) dz)
\]
as the space of $p$-integrable functions with respect to the measure $\omega^a_\eps(y) dz$. If $F$ is a vector field, we write $F \in (L^{p,a, \eps}(\Omega))^d$.

\begin{remark}
Since $\omega_\eps^a \in L^1_\loc(\R^d)$, the null sets of the measure $\omega^a_\eps(y) dz$ coincide with those of the standard Lebesgue measure on $\R^d$. As a result, when we say that a property holds almost everywhere, we will not specify the reference measure, as it is understood to be the Lebesgue measure in all cases.
\end{remark}

Let $C^\infty(\overline\Omega) = \{ u_{\mid \overline\Omega} \mid u \in C^\infty(\R^d)\}\,$ and consider the weighted norm on $C^\infty(\overline\Omega)$ given by 
\[
\| u\|_{H^{1,a, \eps}(\Omega)}^2 = \int_{\Omega}\omega^a_\eps ( u^2  + |\nabla u|^2) \,dz\,.
\]
The corresponding weighted Sobolev space $H^{1,a,\eps}(\Omega)$ is then defined as
\[
H^{1,a, \eps}(\Omega) = \text{the completion of $C^\infty(\overline{\Omega})$ with respect to $\| \cdot\|_{H^{1,a, \eps}(\Omega)}$} \,.
\]
If $\Omega$ is unbounded, we define the corresponding local space as
\[
H^{1,a,\eps}_\loc(\Omega) = \{u: \Omega \to \R \mid u \in H^{1,a, \eps}(B_R\cap \Omega) \text{ for all }R>0\}\,.
\]
In the case $\Omega = B_R$, we also introduce the subspace of functions that are even with respect to each $y_i$-variable as
\[
H^{1,a, \eps}_{\rm e}(B_R) = \{ u \in H^{1,a, \eps}(B_R) \mid u(x, y) = u(x, |y_1|, \ldots, |y_n|) \}\,. 
\]
In all the above definitions, we adopt the following notational simplifications:
\begin{itemize}[left=0pt]
\item[$-$] if $\eps=0$, we omit the superscript and write, for instance, $H^{1,a}(\Omega) = H^{1,a,0}(\Omega)$. 
\item[$-$] if $a = 0$, we write $H^1(\Omega) = H^{1,0,\eps}(\Omega)$ for any $\eps \in [0,1]^n$. In this case, the space coincides with the standard Sobolev space $H^1(\Omega)$.
\end{itemize}

Let us recall that $B_R^*=B_R\cap \R^d_*$. The next lemma establishes the precise relationship between
\[
\begin{aligned}
H^{1,a,\eps}(B_R^*) \qquad \text{ and }\qquad  H^{1,a,\eps}(B_R)\,.
\end{aligned}
\]
\begin{Lemma}\label{L:Hevenid}
Let $a \in (-1, \infty)^n$ and $\eps \in [0,1]^n$. Let $u \in H^{1,a,\eps}(B_R^*)$, and define $\bar u : B_R \to \R$ by
\[
\bar u(x, y) = u(x, |y_1|, \ldots, |y_n|)\,.
\]
Then $\bar u \in H^{1,a,\eps}_{\rm e}(B_R)$ and $\| \bar u \|_{H^{1,a, \eps}(B_R)} = 2^n \| u\|_{H^{1,a, \eps}(B_R^*)}$. Moreover, the map $u \mapsto \bar u$ defines a bijection between $H^{1,a,\eps}(B_R^*)$ and $H^{1,a,\eps}_{\rm e}(B_R)$. Therefore,
\[
H^{1,a,\eps}(B_R^*) \simeq H^{1,a,\eps}_{\rm e}(B_R) \subset H^{1,a,\eps}(B_R)\,.
\]
\end{Lemma}
\begin{proof}
Let $u \in H^{1,a,\eps}(B_R^*)$, and let $\varphi_h\in C^\infty(\overline{B^*_R})$ be a sequence such that $ \varphi_h \to u$ in $H^{1,a,\eps}(B_R^*)$. Define
\[
\tilde \varphi_h(x, y) = \varphi(x, |y_1|, \ldots, |y_n|)\,. 
\]
Via convolution with a radially symmetric mollifier and a standard diagonal argument, starting from $\tilde \varphi_h$ we construct a new sequence $\bar \varphi_h \in C^\infty(\overline{B_R})$ such that each $\bar \varphi_h$ is even in each $y_i$ and satisfies $\| \bar u -\bar \varphi_h\|_{H^{1,a, \eps}(B_R)} \to 0$. Hence, $\bar u \in H^{1,a,\eps}_{\rm e}(B_R)$. The remaining claims are straightforward: we limit ourselves to notice that the inverse of the map $u \mapsto \bar u$ is simply the restriction $\bar u \mapsto \bar u_{|B_R^*}$.
\end{proof}

\subsection{The weighted \texorpdfstring{$H=W$}{H=W} theorem}

Let $\Omega \subseteq \R^d$ be a bounded Lipschitz domain. Define 
\[
\tilde W(\Omega):= \{u\in W^{1,1}_{\loc}(\Omega) \mid \|u\|_{H^{1,a, \eps}(\Omega)}<\infty\}\,,
\]
and introduce the Hilbert space
\begin{equation}\label{defW1a}
W^{1,a, \eps}(\Omega):=\text{the completion of }\tilde{W}(\Omega)\text{ with respect to the norm }\|\cdot\|_{H^{1,a, \eps}(\Omega)}.
\end{equation}
Observe that if $a_i \geq 1$, then $|y_i|^{-a_i} \not \in L^1_\loc(\R^d)$. For this reason, we say that the factor $\rho_{\eps_i}^{a_i}$ is a \emph{superdegenerate} component of the weight $\omega^a_\eps$ whenever both $a_i \geq 1$ and $\eps_i= 0$. 
We denote by  
\[
J_{\rm sd}:=\{i\in \{1,2,\dots,n\} \mid a_i\geq 1, \ \eps_i = 0\}
\]
the set of indices corresponding to the superdegenerate components of the weight $\omega^a_\eps$. Accordingly, the superdegenerate part of the characteristic set $\Sigma_0$ is defined as 
\[
\Sigma_{\rm sd} = \bigcup_{i \in J_{\rm sd}} \Sigma_i\,.
\]
We then have the following characterization.
\begin{Lemma}\label{lemH=W}
Let $a \in (-1, \infty)^n$ and $\eps \in [0,1]^n$. Then 
\[
W^{1,a, \eps}(\Omega)=\{u\in W^{1,1}_{\loc}(\Omega\setminus \Sigma_{\rm sd})\mid \|u\|_{H^{1,a, \eps}(\Omega)}<\infty\}\,.
\]
Moreover, for every $u \in W^{1,a, \eps}(\Omega)$ there exists a sequence $u_h \in W^{1,a, \eps}(\Omega)$ such that 
\[
{\rm supp}(u_h) \subset \Omega \setminus \Sigma_{\rm sd}\quad \text{ and }\quad  \|u-u_h\|_{H^{1,a, \eps}(\Omega)}\to 0 \quad \text{as } h\to \infty.
\]
\end{Lemma}
\begin{proof}
We consider the case $\eps=0$, the general case follows by the same argument. Let $u_k \in  \tilde{W}(\Omega)$ be a Cauchy sequence with respect to the $H^{1,a}(\Omega)$-norm. Since weighted $L^2$-spaces are complete, 
there exist $u\in L^{2,a}(\Omega)$ and a field $V\in (L^{2,a}(\Omega))^d$ such that 
\begin{equation}\label{eqhw2.1}
u_k \rightarrow u \quad \text{in }  L^{2,a}(\Omega)\,, \quad \text{ and }\quad 
\nabla u_k \rightarrow V \quad \text{in } (L^{2,a}(\Omega))^d\,.
\end{equation}
Since $\omega^{-a} \in L^1_\loc(\Omega \setminus \Sigma_{\rm sd})$, H\"older's inequality gives $u\in L^1_{\loc}(\Omega \setminus \Sigma_{\rm sd})$ and $V\in (L^1_{\loc}(\Omega \setminus \Sigma_{\rm sd}))^d$. Fix $\phi\in C^\infty_c(\Omega \setminus \Sigma_{\rm sd})$. By \eqref{eqhw2.1} and H\"older's inequality
\[
\int_{\Omega} (u-u_k)\nabla \phi \,dz\rightarrow 0\, \quad \text{and} \quad\int_{\Omega}(V-\nabla u_k)\phi \,dz\rightarrow 0\,.
\]
Since $u_k\in W^{1,1}_{\loc}(\Omega)$, it follows that
\[
\int_{\Omega}(u \nabla \phi +V \phi)\, dz=\lim_{k\to \infty}\int_{\Omega} (u_k \nabla \phi+\phi\nabla u_k)\, dz=0\,,
\]
showing that $V = \nabla u$ in  $\Omega \setminus \Sigma_{\rm sd}$. Thus $u\in W^{1,1}_{\loc}(\Omega\setminus \Sigma_{\rm sd})$, and we conclude that
\[
W^{1,a}(\Omega)\subseteq \{u\in W^{1,1}_{\loc}(\Omega\setminus \Sigma_{\rm sd})\mid \|u\|_{H^{1,a}(\Omega)}<\infty\}\,.\]

Before proving the reverse inclusion, we construct a family of smooth cutoff functions adapted to the superdegenerate set. Let $\eta\in C^\infty(\R)$ satisfy $0 \leq \eta \leq 1$, with $\eta(t) = 1$ for $t>2$ and $\eta(t) = 0$ for $t<1$. 
For each integer $h >0$, define 
\begin{equation}\label{eq:degcapfunction}
\eta_h:= 1 - \prod_{i \in J_{sd}}\Big(1- \eta \Big(\frac{-\log|y_i|}{h}\Big)\Big)\,.
\end{equation}
Then $\eta_h \in C^\infty(\R^d)$ depends only on the superdegenerate variables $y_i$, $i\in J_{sd}$, and is such that 
\[
\begin{aligned}
&\eta_h = 1\quad \text{ if } |y_i| < e^{-2h} \text{ for some } i \in J_{\rm sd} \,,\\
&\eta_h = 0 \quad \text{ if } |y_i| > e^{-h} \text{ for all } i \in J_{\rm sd}\,.
\end{aligned}
\]
Clearly, $\eta_h\to 0$ a.e. as $h \to \infty$, and hence $\eta_h \to 0$ in $L^{2,a}(\Omega)$. To estimate the gradient, note that
\[
|\nabla \eta_h| \leq \sum_{i \in J_{sd}} \Big|\eta'\Big(-\frac{\log|y_i|}{h}\Big) \frac{y_i}{h y_i^2}\Big| \leq \frac{c}{h} \sum_{i \in J_{sd}}|y_i|^{-1}\chi_{\{e^{-2h}\leq |y_i|\leq e^{-h}\}}\,, 
\]
where the constant $c$ depends only on $\|\eta'\|_{L^\infty(\R)}$. Since $a_i \geq 1$ for $i \in J_{sd}$, it follows that
\[
\int_{\Omega}\omega^a|\nabla \eta_h|^2 \, dz \leq \frac{c}{h^2}\sum_{i \in J_{sd}}\int_{\Omega}\omega^{a - 2e_i}\chi_{\{e^{-2h}\leq |y_i|\leq e^{-h}\}}\,dz \leq \frac{c}{h^2}\sum_{i \in J_{sd}}\int_{e^{-2h}}^{e^{-h}}t^{a_i-2} dt \to 0\,, \quad \text{ as }\quad h \to \infty\,.
\]
Hence, $\|\eta_h\|_{H^{1,a}(\Omega)} \to 0 $ as $h \to \infty$. 

We now prove the reverse inclusion. Let $u\in  W^{1,1}_{\loc}( \Omega\setminus\Sigma_{\rm sd})$ with  $\|u\|_{H^{1,a}(\Omega)}<\infty$. Without loss of generality, we may assume that $u\in L^\infty(\Omega)$, since the truncations
\[
u_k=
\begin{cases}
u &\text{if }|u| \leq k\,\\
\pm k &\text{if }\pm u>k\,,
\end{cases}
\]
satisfy $u_k\rightarrow u$ in $H^{1,a}(\Omega)$ and $u_k\in L^\infty(\Omega)\cap  W^{1,1}_{\loc}(\Omega\setminus\Sigma_{\rm sd})$. Fix $h>0$, and define
\[
u_h:=u(1-\eta_h)\,,
\]
where $\eta_h$ is as in \eqref{eq:degcapfunction}.
Since $\supp(u_h)\subset \Omega\setminus\Sigma_{\rm sd}$, we have $u_h, |\nabla u_h|\in L^1(\Omega)$. Moreover, for any $\phi\in C^\infty_c(\Omega)$ it holds $(1-\eta_h)\phi\in C^\infty_c( \Omega\setminus\Sigma_{\rm sd})$, thus 
\[
\int_{\Omega} (u_h\nabla \phi+\nabla u_h \phi)\, dz=\int_{\Omega}( u \nabla [(1-\eta_h) \phi]+ (1-\eta_h)\phi \nabla u)\, dz=0\,,
\]
using that $u\in W^{1,1}_{\loc}(\Omega\setminus\Sigma_{\rm sd})$. Hence, $u_h\in W^{1,1}_{\loc}(\Omega)$. To conclude, we estimate
\[
\begin{aligned}
\|u-u_h\|_{H^{1,a}(\Omega)}^2\leq &\ 2\int_{\Omega}\omega^a \eta_h^2(u^2+|\nabla u|^2)\,dz+2\int_{\Omega}\omega^a  u^2  |\nabla \eta_h|^2 \,dz \\
\leq &\ o(1) + 2\|u\|_{L^\infty}\|\eta_h\|_{H^{1,a}(\Omega)}\rightarrow 0\,, \quad \text{ as }\quad h \to \infty\,,
\end{aligned}
\]
by dominated convergence and using that $\|\eta_h\|_{H^{1,a}(\Omega)} \to 0 $. This completes the proof. 
\end{proof}

We are now ready to prove the main theorem of this section.              
\begin{Lemma}\label{thmH=W}
Let $a \in (-1, \infty)^n$ and $\eps \in [0,1]^n$. Then
\[
H^{1,a, \eps}(\Omega)=W^{1,a, \eps}(\Omega)\,.
\]
\end{Lemma}
\begin{proof}
Fix $u\in H^{1,a, \eps}(\Omega)$. By definition, there exists a sequence $u_k\in C^\infty(\overline{\Omega}) \subset W^{1,1}_\loc(\Omega)$ such that $\|u_k-u\|_{H^{1,a,\eps}(\Omega)}\to 0$. Thus, \eqref{defW1a} trivially implies $H^{1,a, \eps}(\Omega)\subseteq W^{1,a, \eps}(\Omega)$.

To prove the reverse inclusion, we begin by decomposing the weight $\omega_\eps^a$ into a superdegenerate part and a regular parts as 
\[
\omega^{a}_\eps = \tilde \omega^a_\eps\prod_{i\in J_{sd}}|y|^{a_i}\,.
\]
Indeed, for every $i \not \in J_{\rm sd}$ either $\eps_i \neq 0$ or $a_i \in (-1, 1)$, so the weight $\tilde\omega^a_\eps$ belongs to the Muckenhoupt $\mathcal A_2$ class. 

Now, by Lemma \ref{lemH=W}, any $u\in W^{1,a, \eps}(\Omega)$ can be approximated by functions $u_h\in W^{1,a, \eps}(\Omega)$ with $\supp(u_h)\subset \Omega \setminus \Sigma_{\rm sd}$. On $\supp(u_h)$, the weights $\omega^a_\eps$ and $\tilde \omega^a_\eps$ are equivalent. Therefore, for each $h$, there exists a sequence $\varphi_{k}^h\in C^\infty_c(\overline{\Omega})$ such that $\|u_h-\varphi^h_k\|_{H^{1}(\Omega, \tilde{\omega}^a_\eps(z)dz)}\rightarrow 0$ as $k\rightarrow \infty$ (see for instance \cite[Theorem 2.5]{Kil94}). Since $\omega^a_\eps \leq c(a, \Omega)\tilde{\omega}^a_\eps$ in $\Omega$, it follows that
\[
\|u_h-\varphi^h_k\|_{H^{1,a, \eps}(\Omega)}\leq \|u_h-\varphi^h_k\|_{H^{1}(\Omega, \tilde{\omega}^a_\eps(z)dz)}\rightarrow 0\,.
\]
Then, a standard diagonal argument gives a sequence in $C^\infty(\overline {\Omega})$ converging to $u$ in $H^{1,a, \eps}(\Omega)$, showing that $u \in H^{1,a, \eps}(\Omega)$. This completes the proof. 
\end{proof}

\subsection{Trace theorem and zero trace Sobolev spaces}

This subsection is devoted to the proof of a trace inequality and to the characterization of weighted spaces of functions with zero trace on $\partial B_R$.

\begin{Lemma}\label{L:traceineq}
Let $a \in (-1, \infty)^n$ and $R>0$. Then there exists a constant $c = c(a, d)>0$ such that, for every $\eps \in [0,1]^n$, $0<r\leq R$, and for each $u\in C^\infty(\overline{B_R})$, it holds
\begin{equation*}
    c\int_{\partial B_r}\omega^a_\eps u^2 d\mathcal{H}^{d-1}\leq r\int_{B_R}\omega^a_\eps|\nabla u|^2 dz+r^{-1}\int_{B_R}\omega^a_\eps  u^2  dz\,.
\end{equation*}
\end{Lemma}

\begin{proof}

Fix $u \in C^\infty(\overline B_r)$. In fact, without loss of generality we can assume $u$ to be the restriction of a function, still denoted by $u$, in $C^\infty_c(\R^d)$. Take $0 < \rho < r \leq R$ and $\sigma \in \S^{d-1}$. Using spherical coordinates we compute
\begin{equation}\label{eq:trace1}
\begin{aligned}
\int_{\partial B_r}\omega^a_\eps  u^2  d\mathcal{H}^{d-1}  &= \int_{\S^{d-1}}r^{d-1}\omega^a_\eps(r \sigma)  u^2 (r\sigma) d\mathcal{H}^{d-1}(\sigma) \\
&\leq \int_{\S^{d-1}}\Big|r^{d-1}\omega^a_\eps(r \sigma)  u^2 (r\sigma) - \rho^{d-1}\omega^a_\eps(\rho \sigma)  u^2 (\rho\sigma)\Big| d\mathcal{H}^{d-1}(\sigma) \\
&+ \int_{\S^{d-1}}\rho^{d-1}\omega^a_\eps(\rho \sigma)  u^2 (\rho\sigma) d\mathcal{H}^{d-1}(\sigma)\,.
\end{aligned}
\end{equation}
To estimate the first integral in the right hand side of \eqref{eq:trace1}, we use the fundamental theorem of calculus to get
\[
\int_{\S^{d-1}}\Big|r^{d-1}\omega^a_\eps(r \sigma)  u^2 (r\sigma) - \rho^{d-1}\omega^a_\eps(\rho \sigma)  u^2 (\rho\sigma)\Big| \leq I_1 + I_2 + I_3\,,
\]
where
\[
\begin{aligned}
&I_1 = (d-1) \int_{\S^{d-1}}\int^r_\rho \tau^{d-2}\omega^a_\eps(\tau \sigma)  u^2 (\tau\sigma)d\tau d\mathcal{H}^{d-1}(\sigma)\,\\
&I_2 = \int_{\S^{d-1}}\int^r_\rho \tau^{d-1}\Big|\frac{d}{d\tau}\omega^a_\eps(\tau \sigma)\Big|  u^2 (\tau\sigma)d\tau d\mathcal{H}^{d-1}(\sigma)\,\\
&I_3 = 2\int_{\S^{d-1}}\int^r_\rho \tau^{d-1}\omega^a_\eps(\tau \sigma)|u|(\tau\sigma)|\nabla u|(\tau \sigma)d\tau d\mathcal{H}^{d-1}(\sigma)\,.
\end{aligned}
\]
Since $\rho \leq \tau \leq r$, we immediately have
\[
I_1 \leq (d-1)\rho^{-1}\int_{\S^{d-1}}\int^r_\rho \tau^{d-1}\omega^a_\eps(\tau \sigma)  u^2 (\tau\sigma)d\tau d\mathcal{H}^{d-1}(\sigma) = (d-1)\rho^{-1} \int_{B_r\setminus B_\rho}\omega^a_\eps u^2  dz\,,
\]
and, using the H\"older's and Young's inequalities, 
\[
\begin{aligned}
I_3 &\leq 2 \int_{\S^{d-1}}\Big( \rho^{-1}\int^r_\rho \tau^{d-1}\omega^a_\eps(\tau \sigma) u^2 (\tau\sigma)d\tau\Big)^{1/2}\Big(\rho \int^r_\rho \tau^{d-1}\omega^a_\eps(\tau \sigma)|\nabla u|^2(\tau \sigma)d\tau\Big) ^{1/2}d\mathcal{H}^{d-1}(\sigma)\\
&\leq \rho^{-1}\int_{B_r \setminus B_\rho} \omega^a_\eps u^2 dz + \rho \int_{B_r \setminus B_\rho}\omega^a_\eps|\nabla u|^2dz\,.
\end{aligned}
\]
To estimate $I_2$, let us first introduce, for $i = 1, \ldots, n$, the projections $\pi_i: \R^d \to \R$ such that $\pi_i z = y_i$. Given $z = r\sigma$, it is easy to see that $y_i = r\pi_i \sigma$. Using this notations, we see that
\[
\frac{d}{d\tau}\omega^a_\eps(\tau \sigma) = \frac{d}{d\tau}\prod_{j=1}^n (\eps_j^2 + \tau^2|\pi_j \sigma|^2)^{a_j/2} = \omega^a_\eps(\tau\sigma)\sum_{i=j}^n a_j \frac{\tau|\pi_j \sigma|^2}{\eps_j^2 + \tau^2|\pi_j \sigma|^2}\,.
\]
Since
\[
\frac{\tau|\pi_j \sigma|^2}{\eps_j^2 + \tau^2|\pi_j \sigma|^2} = \tau^{-1}\frac{\tau^2|\pi_j \sigma|^2}{\eps_j^2 + \tau^2|\pi_j \sigma|^2}\leq \tau^{-1}\,,
\]
we readily see that
\[
\Big|\frac{d}{d\tau}\omega^a_\eps(\tau \sigma)\Big| \leq c \tau^{-1}\omega^a_\eps(\tau\sigma)\,,
\]
where $c = \sum_{j=1}^n|a_j|$. Therefore
\[
I_2 \leq c\int_{\S^{d-1}}\int^r_\rho \tau^{d-2}\omega^a_\eps(\tau \sigma)  u^2 (\tau\sigma)d\tau d \mathcal{H}^{d-1}(\sigma) \leq c\rho^{-1}\int_{B_r \setminus B_\rho} \omega^a_\eps u^2 dz \,.
\]
Thanks to previous estimates, \eqref{eq:trace1} becomes 
\[
\begin{aligned}
\int_{\partial B_r}\omega^a_\eps  u^2  d\mathcal{H}^{d-1}  &\leq c\rho^{-1}\int_{B_r \setminus B_\rho} \omega^a_\eps u^2 dz
+ \rho \int_{B_r \setminus B_\rho}\omega^a_\eps|\nabla u|^2dz \\&+ \int_{\S^{d-1}}\rho^{d-1}\omega^a_\eps(\rho \sigma)  u^2 (\rho\sigma) d\mathcal{H}^{d-1}(\sigma)\,,
\end{aligned}
\]
hence, it immediately follows
\[
\begin{aligned}
\rho\int_{\partial B_r}\omega^a_\eps  u^2  d\mathcal{H}^{d-1}  &\leq c\int_{B_R} \omega^a_\eps u^2 dz + \rho^2 \int_{B_R}\omega^a_\eps|\nabla u|^2dz +r\int_{\S^{d-1}}\rho^{d-1}\omega^a_\eps(\rho \sigma)  u^2 (\rho\sigma) d\mathcal{H}^{d-1}(\sigma)\,.
\end{aligned}
\]
To compete the proof, simply integrate previous inequality over $\rho \in (0, r)$.
\end{proof}

Define $L^{2,a, \eps}(\partial B_R):=L^2(\partial B_R, \omega^a_\eps d\mathcal{H}^{d-1})$, and introduce the Sobolev space 
\[
H^{1,a, \eps}_0(B_R) = \text{the completion of $C^\infty_c(B_R)$ with respect to $\| \cdot\|_{H^{1,a, \eps}(B_R)}$}\,.
\]
\begin{Lemma}\label{thmtrace}
Let $a \in (-1, \infty)^n$, $\eps \in [0,1]^n$ and $R>0$. There exists a unique bounded linear operator
   \[
        T\,: \, H^{1,a, \eps}(B_R)\rightarrow L^{2,a, \eps}(\partial B_R)
   \]
    such that 
    \[
    T u= u_{\vert \partial B_R}\qquad \text{for }u\in C^\infty(\overline{B_R})\,.
   \]
Moreover, the following characterization holds
\begin{equation}\label{charH0}
    H^{1,a,\eps}_0(B_R)=\{u\in H^{1,a, \eps}(B_R) \mid T u=0\}\,.
\end{equation}
\end{Lemma}
\begin{proof}
Boundedness and uniqueness of $T$ follow from Lemma \ref{L:traceineq}. It remains to prove \eqref{charH0}. Since one inclusion is immediate ($\subseteq$), we focus on the reverse one ($\supseteq$). Let $u\in H^{1,a, \eps}(B_R)$ with $Tu=0$. We extend $u$ to $B_{2R}$ by setting 
\[
\bar{u}:=
\begin{cases}
    u(z) \quad &\text{if }\,z\in B_R\\
    0\quad &\text{if }\,z\in B_{2R}\setminus B_{R}\,.
\end{cases}
\]
We claim that $\bar{u}\in H^{1,a, \eps}(B_{2R})$.

Let $\phi\in C^\infty_c(B_{2R}\setminus \Sigma_{\rm sd})$, and let $\varphi_k \in C^\infty(\overline{B_R})$ be a sequence approximating $u$ in $H^{1,a, \eps}(B_R)$. Then
\[
    \int_{B_{2R}}\bar u\nabla \phi \, dz=\int_{B_R}u \nabla \phi \, dz=\lim_{k\rightarrow \infty}\int_{B_R}\varphi_k \nabla \phi \,dz.
\]
Since $\varphi_k$ is smooth, integration by parts gives 
\[
   \int_{B_R}\varphi_k \nabla \phi\, dz=-\int_{B_R}\phi\nabla \varphi_k\, dz+\int_{\partial B_R} \varphi_k \phi \,\nu \,d\mathcal{H}^{d-1}\,,
\]
where $\nu$ denotes the outward unit normal to $\partial B_R$.
By the assumption $Tu=0$ we have 
\[
\int_{\partial B_R} \omega^a_\eps \varphi_k^2 \, d \mathcal{H}^{d-1}\to 0\,,
\]
so the boundary term vanishes in the limit. Thus
\[
\lim_{k\rightarrow \infty}\int_{B_R}\varphi_k \nabla \phi \,dz=-\lim_{k\rightarrow \infty}\int_{B_R}\phi\nabla \varphi_k\, dz=-\int_{B_R}\phi\nabla u \,dz\,.
\]
This shows that $\bar u\in W^{1,1}_{\loc}(B_{2R}\setminus \Sigma_{\rm sd})$, with weak gradient
\[
   \nabla \bar u(z):=
\begin{cases}
 \nabla u(z) \quad &\text{if }\,z\in B_R\\
 0\quad &\text{if }\, z\in B_{2R}\setminus B_R\,.
\end{cases}
\]
Since $\|\bar u\|_{H^{1,a, \eps}(B_{2R})}=\| u\|_{H^{1,a, \eps}(B_{R})}<\infty$, Lemma \ref{lemH=W} and Lemma \ref{thmH=W} imply that $\bar u\in H^{1,a, \eps}(B_{2R})$, as claimed.

Fix now $\lambda \in (1,2)$ and define $u_\lambda: B_R \to \R$ as 
\[
u_\lambda (z) =\bar u(\lambda z)\,.
\]
By construction, $u_\lambda(z) = 0$ for $z \in B_R \setminus B_{R_\lambda}$, where $R_\lambda =R/\lambda$. Moreover, using the change of variables $z = \lambda^{-1}\xi$ and since $\omega_\eps^a(\lambda^{-1}\xi )\leq 2\omega_\eps^a(\xi)$ for $\lambda\in (1,2)$, we see that $u_\lambda \in H^{1,a, \eps}(B_R)$. We now show that $u_\lambda \in H_0^{1, a, \eps}(B_R)$. Let $\varphi_k \in C^\infty_c(\R^d)$ be such that $\varphi_k \to u_\lambda $ in $H^{1,a,\varepsilon}(B_R)$. Consider a cutoff function $\eta\in C^\infty_c(B_R)$ such that $\eta \equiv 1$ in $B_{R_\lambda}$. Then $\eta \varphi_k\in C^\infty_c(B_R)$ and, since $u_\lambda = 0$ in $B_R \setminus B_{R_\lambda}$, for a.e. $z \in B_R$ it holds
\[
\begin{aligned}
&|u_\lambda - \eta\varphi_k| \leq  |u_\lambda - \varphi_k|\,,\\
&|\nabla u_k - \nabla(\eta \varphi_k)| \leq |\nabla u_k - \nabla \varphi_k| + \|\nabla \eta\|_{L^\infty(B_R)}|u_\lambda - \varphi_k|\,.
\end{aligned}
\]
Thus, there exists a constant $c>0$ such that 
\[
\|u_\lambda - \eta \varphi_k\|_{H^{1,a, \eps}(B_{R})} \leq c\|u_\lambda - \varphi_k\|_{H^{1,a, \eps}(B_{R})} \to 0\,.
\]
This implies that $u_\lambda \in H^{1,a, \eps}_0(B_{R})$ for each $\lambda\in (1,2)$. Since $\|u - u_\lambda\|_{H^{1,a, \eps}(B_{R})} \to 0$ as $\lambda \to 1^+$, and $H^{1,a,\eps}_0(B_R)$ is complete, it follows that $u\in H^{1,a, \eps}_0(B_R)$, completing the proof.
\end{proof}

Define $\partial ^+B_R^* = \R^d_*\cap \partial B_R^*$, and consider the weighted $L^2$ space
\[
L^{2,a, \eps}(\partial^+ B_R^*):=L^2(\partial^+ B_R^*, \omega^a_\eps d\mathcal{H}^{d-1})\,,
\]
as well as the weighted Sobolev space
\[
\hat H^{1,a, \eps}(B_R^*) = \text{the completion of $C^\infty_c(\overline{B_R^*}\setminus\partial^+ B_R^*)$ with respect to $\| \cdot\|_{H^{1,a, \eps}(B_R^*)}$ }\,.
\]
Thanks to Lemma \ref{L:Hevenid}, Lemma \ref{L:traceineq} and Lemma \ref{thmtrace}, we obtain the following result.
\begin{Lemma}\label{thmeqtrace*}
Let $a \in (-1, \infty)^n$, $\eps \in [0,1]^n$ and $R>0$. Then the following hold:
\begin{itemize}
\item[$i)$] there exists a constant $c = c(a, d)>0$ such that for every $0<r\leq R$ and for each $u\in C^\infty(\overline{B^*_R})$, it holds
\[
    c\int_{\partial^+ B_r^*}\omega^a_\eps u^2 d\mathcal{H}^{d-1}\leq r\int_{B_R^*}\omega^a_\eps|\nabla u|^2 dz+r^{-1}\int_{B_R^*}\omega^a_\eps  u^2  dz\,;
\]
\item[$ii)$] there exists a unique bounded linear operator
    \[
        T\,: \, H^{1,a, \eps}(B_R^*)\to L^{2,a, \eps}(\partial^+ B_R^*)
    \]
    such that 
    $$
    T u= u_{\vert \partial^+ B_R^*}\qquad \text{for }u\in C^\infty(\overline{B^*_R})\,;
    $$
\item[$iii)$] the following characterization holds
\[
    \hat H^{1,a,\eps}(B_R^*)=\{u\in H^{1,a, \eps}(B_R^*) \mid T u=0\}\,.
\]
\end{itemize}
\end{Lemma}

\subsection{Sobolev inequality}

Set $\langle a^+\rangle=\sum_{i=1}^n (a_i)^+$. When $  d + \langle a^+\rangle -2 >0 $, we define the critical exponent
\begin{equation}\label{eq:sobexp}
2^*_a = \frac{2(d + \langle a^+\rangle)}{d + \langle a^+\rangle - 2}.
\end{equation}
The aim of this subsection is to prove the following result. 

\begin{Lemma}[$\eps$-stable Sobolev embeddings]\label{L:Sobolev}
Let $a \in (-1, \infty)^n$ and $R>0$. Then there exists a constant $c= c(a, R)>0$ such that for every $\eps \in [0,1]^n$, the following hold:
\begin{itemize}
\item[$i)$] if $d + \langle a^+\rangle-2 >0$, then for all $2 \leq q \leq 2^*_a$ and $u \in H^{1,a,\eps}_0(B_R)$, we have
\[
c\, \Big( \int_{B_R}\omega^a_\e | u |^{q} \,dz \Big)^{\frac{1}{q}}\leq \Big(\int_{B_R}\omega^a_\e |\nabla u|^2 \,dz\Big)^\frac{1}{2}\,.
\]
\item[$ii)$] if $d + \langle a^+\rangle -2 = 0$, the above inequality holds for $q\in(2,\infty)$, with the constant $c>0$ depending also on $q$.
\end{itemize}
\end{Lemma}

We begin with two preliminary results. The first one gives an $\eps$-stable $L^1$-Caffarelli-Kohn-Nirenberg type inequality in the case $ d = n = 1 $, for weights of the form $\rho_\eps^a$. 

\begin{Lemma}\label{L:Sob11}
Let $a >-1$. Then there exists a constant $c = c(a)  >0$ such that for every $\eps \in [0,1]$, for every $q \geq 1$ with $q a \leq1+a$, and for every $u \in C^{0,1}_c(\R)$, it holds:
\[
c\, \Big( \int_{\R}\rho^a_\eps(y) |u |^{q} \,dy \Big)^{\frac{1}{q}}\leq \int_{\R}\rho^{\frac{1+a}{q}}_\eps(y) |u'| \,dy\,.
\]
\end{Lemma}
\begin{proof}
We first show that for every $u \in C^{0,1}_c(\R)$ and for every $\delta >0$, it holds 
\begin{equation}\label{eq:linf11}
\||\cdot|^\delta |u|\|_{L^\infty(\R)} \leq 2 \int_\R |y|^\delta |u'|\,dy\,.
\end{equation}
To see this, fix $y \in \R$ and use the fundamental theorem of calculus to get  
\[
|y|^\delta|u| = \int_{-\infty}^y \frac{d}{ds}(|s|^\delta |u|)\,ds \leq \delta \int_\R |s|^{\delta-1}|u|\,ds + \int_\R |s|^{\delta}|u'|\,ds\,.
\]
Integrating by parts we find
\[
\delta \int_\R |s|^{\delta-1}|u|\,ds =  \int_\R \frac{d}{ds}(|s|^{\delta-1}s)|u|\,ds = - \int_\R |s|^{\delta-1}s|u|^{-1} u u' \,ds \leq \int_\R |s|^{\delta}|u'|\,ds\,,
\]
Combining the two estimates, we get \eqref{eq:linf11}.

Next, define 
\[
G_{a, \eps}(y) = \int_{0}^y \rho_\eps^a(s)\, ds\,.
\]
We distinguish to cases. Suppose first that $a >0$. Since $\rho_\eps^a$ is increasing on $(0, \infty)$, we have 
\[
|G_{a, \eps}(y)| \leq \int_0^{|y|}\rho_\eps^a(s)\, ds \leq \rho_\eps^a(y) |y|\,. 
\]
Thus, integration by parts gives
\[
\int_\R \rho_\eps^a |u|^q\, dy = \int_\R \frac{d}{dy}G_{a,\eps} |u|^{q}\,dy \leq q \int_\R |G_{a,\eps}| |u|^{q-1}|u'| \,dy \leq q \int_\R\rho_\eps^a (|y|^{a+\beta} |u|)^{q-1}|y|^\beta|u'| \,dy\,,
\]
where we set $\beta = (1+a-qa)/q$. Applying \eqref{eq:linf11} with $\delta = a+ \beta >0$, we obtain
\[
\int_\R \rho_\eps^a |u|^{q}\,dy \leq q 2^{q-1}\Big(\int_\R |y|^{a+\beta} |u'| dy\Big)^{q-1}\int_\R |y|^\beta\rho_\eps^a|u'| \,dy \,.
\]
The assumption $qa \leq 1+a$ ensures $\beta \geq 0$. Thus, using that $|y| \leq \rho_\eps$, $a>0$, and $a + \beta = \frac{1+a}{q}$, we get 
\[
\int_\R \rho_\eps^a |u|^{q}\,dy \leq q 2^{q-1}\Big(\int_\R \rho_\eps^{\frac{1+a}{q}} |u'| \,dy\Big)^{q}\,.
\]
Since $(q 2^{q-1})^{1/q}$ is uniformly bounded for $q \geq 2$, this case is complete. 

Now let $-1 < a \leq 0$. In this case we use that $\rho_\eps^a(y) \leq |y|^a$, obtaining
\[
|G_{a,\eps}(y)| \leq  \int_{0}^{|y|} s^a ds  = \frac{|y|^{a+1}}{a+1}\,. 
\]
It follows that
\[
\int_{\R} \rho_\eps^a |u|^q dy \leq q\int_{\R} |G_{a, \eps}| |u|^{q-1}|u'| dy \leq \frac{q}{a+1}\int_{\R} |y|^{a+1} |u|^{q-1}|u'| dy \,.
\] 
Applying \eqref{eq:linf11} with $\delta = \frac{a+1}{q}>0$ we get 
\[
\int_{\R} \rho_\eps^a |u|^q dy \leq \frac{q}{a+1} \int_{\R} (|y|^{\frac{a+1}{q}} |u|)^{q-1}|y|^{\frac{a+1}{q}}|u'| dy \leq \frac{q2^{q-1}}{a+1}\Big( \int_{\R} |y|^{\frac{a+1}{q}}|u'| dy\Big)^q\,,
\]
and to conclude we use that $|y| \leq \rho_\eps$ and $a+1>0$. This completes the proof.
\end{proof}

Next preparatory lemma is a $L^1$ version of Lemma \ref{L:Sobolev}

\begin{Lemma}[$\eps$-stable $L^1$ Sobolev embeddings]\label{L:SobL1}
Let $a \in (-1, \infty)^n$ and $R>0$. Then there exists a constant $c= c(d, a, R)>0$ such that for every $\eps \in [0,1]^n$, for every $q \geq 1$ with $q(d + \langle a^+ \rangle -1) \leq d + \langle a^+\rangle$, and for every $u \in C^{0,1}_c(B_R)$,  it holds
\begin{equation}\label{eq:Sobolev}
c\, \Big( \int_{B_R}\omega^a_\e |u |^{q} dz \Big)^{\frac{1}{q}}\leq \int_{B_R}\omega^a_\e |\nabla u| dz\,.
\end{equation}
\end{Lemma}
\begin{proof}
It suffices to prove the result for $d = n$. When $n< d$, the inequality follows form the case $d = n$ by extending the parameters to $\tilde a = (0,\ldots, 0, a) \in (-1, \infty)^d$ and $\tilde \eps= (0,\ldots, 0, \eps) \in [0,1]^d$.

Inequality \eqref{eq:Sobolev} is a consequence of the following, more general, inequality. 
Given $a \in (-1, \infty)^d$, there exists $c_a >0$ such that for every $\beta \in \R^d$ such that $\beta_i \geq 1$ with $\beta_ia_i \leq 1 + a_i$ and all $u \in C^{0,1}_c(\R^d)$, it holds 
\begin{equation}\label{eq:L1diseqdugn}
\int_{\R^{d}}\omega_\eps^a|u|H_\beta\Big( \frac{|u|}{\|u\|_{L^{1,a,\eps}(\R^{d})}}\Big) \,dy\leq \int_{\R^{d}}\omega^a_\eps\Big(\sum_{i=1}^{d}\rho_{\varepsilon_i}(y_i)^{\frac{1 + a_i}{\beta_i} - a_i}\Big)|\nabla u| \,dy\,,
\end{equation}
where 
\begin{equation*}
H_\beta(s):=c_a s^{1-\frac{1}{\gamma_\beta}}\,,
\end{equation*}
and the exponent $\gamma_\beta$ satisfies
\[
\begin{aligned}
&\gamma_\beta = 1  & &\text{ if }\beta_i =1 \text{ for some }i=1, \ldots, d\,,\\
&\gamma'_\beta = \sum_{i = 1 }^d \beta_i'  & &\text{ if }\beta_i \neq 1 \text{ for all }i=1, \ldots, d\,.
\end{aligned}
\]
Assume that \eqref{eq:L1diseqdugn} holds, and let $u \in C_c^{0,1}(B_R)$. Then by \cite[Lemma 3.9]{CorFioVit25}, we have
\[
c_a\Big(\int_{B_R}\omega_\eps^a|u|^{\gamma_\beta} \,dy\Big)^{\frac{1}{\gamma_\beta}}\leq \int_{B_R}\omega^a_\eps\Big(\sum_{i=1}^{d}\rho_{\varepsilon_i}(y_i)^{\frac{1 + a_i}{\beta_i}-a_i}\Big)|\nabla u| \,dy\,.
\]
Using that $\rho_{\eps_i} \leq \sqrt{1+R^2}$ for $y\in B_R$, and noting that $0 \leq \frac{1+a_i}{\beta_i}-a_i\leq 1$, we obtain 
\[
\sum_{i=1}^{d}\rho_{\varepsilon_i}(y_i)^{\frac{1 + a_i}{\beta_i}- a_i} \leq d\sqrt{1 + R^2}\,.
\]
Moreover, the assumptions on $\beta_i$ imply that $\beta'_i \geq 1 + a^+_i$, which in turn gives $\gamma_\beta' \geq d + \langle a^+ \rangle$. Therefore, 
\[
\gamma_\beta \geq 1 \quad \text{ and } \quad \gamma_\beta(d + \langle a^+ \rangle -1) \leq d + \langle a^+\rangle\,.
\]
Hence, to conclude it suffices to relabel $\gamma_\beta$ as $q$. 

Let us now prove \eqref{eq:L1diseqdugn}. We proceed by induction. The base case $d=1$ is a straightforward consequence of Lemma \ref{L:Sob11} and \cite[Lemma 3.9]{CorFioVit25}, noting that $\omega^a_\eps = \rho^a_\eps$ in this case. Let now fix $a \in (-1, \infty)^d$, $\eps \in[0,1]^d$ and $\beta \in \R^d$ such that $\beta_i \geq 1$ with $\beta_ia_i \leq 1 + a_i$. Moreover, let us denote $\hat a = (a_1, \ldots, a_{d-1})$, $\hat \eps = (\eps_1, \ldots, \eps_{d-1})$ and $\hat \beta = (\beta_1, \ldots, \beta_{d-1})$\,. By inductive assumption, for all $v \in C^{0,1}_c(\R^{d-1})$, it holds 
\begin{equation}\label{eq:L1diseqdugnind}
\int_{\R^{d-1}}\omega_{\hat\eps}^{\hat a}|v|H_{\hat\beta}\Big( \frac{|v|}{\|v\|_{L^{1,{\hat a},{\hat \eps}}(\R^{d-1})}}\Big) \,dy\leq \int_{\R^{d-1}}\omega^{\hat a}_{\hat \eps}\Big(\sum_{i=1}^{d-1}\rho_{\varepsilon_i}(y_i)^{\frac{1 + a_i}{\beta_i} - a_i}\Big)|\nabla v| \,dy\,.
\end{equation}
Moreover, by Lemma \ref{L:Sob11} and \cite[Lemma 3.9]{CorFioVit25}, for every $w \in C^{0,1}_c(\R)$  we have
\begin{equation}\label{eq:FGSob}
\int_{\R}\rho_{\eps_d}^{a_d}|w|F\Big( \frac{|w|}{\|w\|_{L^{1,a_d,\eps_d}(\R)}}\Big) \,dy_d\leq \int_{\R}\rho_{\eps_d}^{\frac{1 + a_d}{\beta_d}}|\nabla w| \,dy_d\,,
\end{equation}
where $F(s) = c_d s^{1-\frac{1}{\beta_d}}$ and the constant $c_d$ only depends on $a_d$. 

Employing \eqref{eq:L1diseqdugnind}, \eqref{eq:FGSob} and adapting the proof of \cite[Proposition 3.3]{CouGriLev03} (see also \cite[Lemma 3.10]{CorFioVit25} and \cite{Grig85}), we deduce that, for every $u \in C^{0,1}_c(\R^d)$, the following $G$-Sobolev inequality holds
\begin{equation*}
\int_{\R^{d}}\omega_\eps^a|u|G\Big( \frac{|u|}{\|u\|_{L^{1,a,\eps}(\R^{d})}}\Big) \,dy\leq \int_{\R^{d}}\omega^a_\eps\Big(\sum_{i=1}^{d}\rho_{\varepsilon_i}(y_i)^{\frac{1 + a_i}{\beta_i} - a_i}\Big)|\nabla u| \,dy\,,
\end{equation*}
with
\[
G(r):=\inf_{st = r}[H_{\hat \beta}(s) + F(t)]
\]
After some standard computations, we see that
\[
G(r) \geq \min\{c_{\hat a}, c_d\} r^{\frac{1}{\gamma_{\hat\beta}' + \beta_d'}}\,,
\]
and \eqref{eq:L1diseqdugn} follows, completing the proof.
\end{proof}

\begin{proof}[Proof of Lemma \ref{L:Sobolev}]
Let $q \geq 2$, $(d + \langle a^+ \rangle-2)q \leq 2(d + \langle a^+ \rangle)$ be fixed. Let us denote
\[
r = \frac{2q}{q+2}\,,
\]
and notice that $ 1 \leq r$, $(d + \langle a^+ \rangle-1)r \leq d + \langle a^+ \rangle$.

We fix $ u \in C^\infty_c (B_R)$ and define $v = |u|^{\frac{q}{r}}$. Since $q >r$, then $v \in C^1_c(B_R)$ and we can compute
\[
\int_{B_R}\omega^a_\eps |u|^q dz = \int_{B_R}\omega^a_\eps |v|^r dz \leq c \Big( \int_{B_R}\omega_\eps^a|\nabla v| dz \Big)^{r},
\]
where the last inequality holds true thanks to Lemma \ref{L:SobL1}

Next we use that $|\nabla v| = \frac{q}{r}|u|^{\frac{q-r}{r}}|\nabla u|$ and Holder inequality to obtain 
\[
c \Big(\int_{B_R}\omega^a_\eps |u|^q dz\Big)^{\frac{1}{r}} \leq \Big( \int_{B_R}\omega^a_\eps|\nabla u|^2 dz \Big)^{\frac{1}{2}} \Big( \int_{B_R}\omega^a_\eps|u|^{\frac{2(q-r)}{r}} dz \Big)^{\frac{1}{2}},
\]
and since 
\[
\frac{2(q-r)}{r} = q \quad \text{ and }\quad \frac{1}{r} - \frac{1}{2} = \frac{1}{q}\,,
\]
we readily conclude. 
\end{proof}

\subsection{Hardy and Poincar\'e-type inequalities}
\begin{Lemma}[$\eps$-stable Hardy inequality]\label{L:Hardy}

Let $a \in (-1, \infty)^n$, $\eps \in [0,1]^n$, and $R>0$. For each $i=1, \ldots, n$, there exists a constant $c = c(a_i)>0$ such that, for all $u \in C^\infty_c(\R^d)$, it holds
\begin{equation}\label{Hardyineq}
c \int_{B_R}\omega^a_\eps  u^2  \,dz \leq \frac{1}{R} \int_{\partial B_R}\omega^{a+2e_i}_\eps u^2 \,d\mathcal{H}^{d-1}  + \int_{B_R}\omega^{a+2e_i}_\eps \Big( \frac{\partial u}{\partial y_i}\Big)^2 \,dz\,.
\end{equation}
\end{Lemma}
\begin{proof}
Fix $u \in C^\infty_c(\R^d)$. Using the identity
\[
\frac{\partial^2 \omega_\eps^{a + 2e_i}}{\partial y_i^2} = (a_i + 2)(a_i +1)\omega_{\eps}^a - (a_i+2)a_i\eps_i^2 \omega_\eps^{a-2e_i}
\]
and integrating by parts, we obtain 
\begin{equation}\label{eq:hardypass}
(a_i+1)\int_{B_R}\omega^{a}_\eps  u^2  \,dz= \frac{1}{R}\int_{\partial B_R}\omega^{a}_\eps |y_i|^{2}  u^2  \,d\mathcal{H}^{d-1} - 2 \int_{B_R}\omega^{a}_\eps y_i u \frac{\partial u}{\partial y_i}\,dz + \eps_i^2a_i\int_{B_R}\omega_{\eps}^{a - 2e_i} u^2 \,dz\,.
\end{equation}
By H\"older's and Young's inequalities, for every $\delta >0$ we have
\[
\Big|2 \int_{B_R}\omega^{a}_\eps y_i u \frac{\partial u}{\partial y_i}\,dz\Big| \leq  \delta\int_{B_R}\omega^{a}_\eps  u^2  \,dz
+ \frac{1}{\delta} \int_{B_R}\omega^{a+ 2e_i}_\eps \Big(\frac{\partial u}{\partial y_i}\Big)^2\,dz\,.
\]
Substituting this into \eqref{eq:hardypass} gives
\[
(a_i+1-\delta)\int_{B_R}\omega^{a}_\eps  u^2  \,dz\leq \frac{1}{R}\int_{\partial B_R}\omega^{a + 2e_i}_\eps   u^2  \,d\mathcal{H}^{d-1}  +\frac{1}{\delta} \int_{B_R}\omega^{a+ 2e_i}_\eps \Big(\frac{\partial u}{\partial y_i}\Big)^2\,dz +  \eps_i^2a_i\int_{B_R}\omega_{\eps}^{a - 2e_i}|u^2|\,dz\,.
\]
We now distinguish two cases. If $a_i >0$, we take $\delta = 1/2$ and use that $\eps_i^2 \leq \rho_{\eps_i}^{2}$ to get
\[
\frac{1}{4}\int_{B_R}\omega^{a}_\eps  u^2  \,dz\leq \frac{1}{2R}\int_{\partial B_R}\omega^{a + 2e_i}_\eps   u^2  \,d\mathcal{H}^{d-1} + \int_{B_R}\omega^{a+ 2e_i}_\eps \Big(\frac{\partial u}{\partial y_i}\Big)^2\,dz\,.
\]
If $-1 < a_i < 0$, we drop the last term on the right-hand side of \eqref{eq:hardypass} and put $\delta = (a_i+1)/2$ to get 
\[
\Big(\frac{a_i+1}{2}\Big)^2\int_{B_R}\omega^{a}_\eps  u^2 \,dz\leq \frac{a_i+1}{2R}\int_{\partial B_R}\omega^{a + 2e_i}_\eps   u^2  \,d\mathcal{H}^{d-1}  +\int_{B_R}\omega^{a+ 2e_i}_\eps \Big(\frac{\partial u}{\partial y_i}\Big)^2\,dz\,.
\]
In both cases, inequality \eqref{Hardyineq} follows with a constant $c = c(a_i) > 0$.
\end{proof}

\begin{remark}\label{Hardyrmrk}
Let $a \in (-1, \infty)^{n-1}\times (-\infty, -1)$ and $\eps \in [0,1]^n$ with $\eps_n = 0$. Then inequality \eqref{Hardyineq} remains valid (and can even be strengthened) for smooth functions $u$ vanishing in a neighbourhood of $\Sigma_n$. In particular,
\[
\frac{(a_n + 1)^2}{4} \int_{B_R}\omega^a_\eps  u^2  dz \leq  \int_{B_R}\omega^{a+2e_n}_\eps \Big( \frac{\partial u}{\partial y_n}\Big)^2 dz\,, \qquad \text{for all }u \in C^\infty_c(\R^d \setminus \Sigma_n)\,.
\]
Indeed, when $\eps_n = 0$ and $a_n < -1$, identity \eqref{eq:hardypass} gives
  \[
  |a_n+1|\int_{B_R}\omega^{a}_\eps  u^2  \,dz\leq 2 \int_{B_R}\omega^{a}_\eps y_n u \frac{\partial u}{\partial y_n}\,dz \,,
  \]
and the desired inequality follows directly by H\"older's inequality.
\end{remark}

As a consequence of Lemma \ref{L:Hardy} (see also Lemma \ref{L:Sobolev} with $q=2$), we obtain an $\varepsilon$-stable Poincar\'e inequality.  
Since its proof is straightforward, we omit it.
\begin{Lemma}\label{lemRpoinc}
Let $a \in (-1, \infty)^n$, $\eps \in [0,1]^n$ and $R>0$. There exists a constant $c = c(a)>0$ such that for all $u \in C^\infty_c(B_R)$ it holds
\[
\frac{c}{\sqrt{1 + R^2}} \int_{B_R}\omega^a_\eps  u^2  dz \leq  \int_{B_R}\omega^{a}_\eps |\nabla u|^2 dz\,.
\]
\end{Lemma}

\subsection{De Giorgi-Nash-Moser theory for \texorpdfstring{$2$}{2}-admissible weights}\label{s:degiorgi}

In this section, we show that our weight $\omega^a$ is $2$-admissible \cite{HeiKilMar06}. This leads to the validity of the De Giorgi-Nash-Moser theory for the associated degenerate equation, by applying \cite{FabKenSer82,HeiKilMar06}.
A nonnegative, locally integrable weight $\pi$ is said to be $2$-admissible if it satisfies four properties: the doubling condition of the associated measure $d\mu=\pi dz$, the well-definitness of the weak gradient, a Sobolev inequality and the Poincar\'e-Wirtinger inequality.
Let us remark that our weight $\omega^a(y)$ satisfies the four properties above in the orthant $\R^d_*$. The Sobolev inequality for $\varepsilon=0$ is contained in Lemma \ref{L:Sobolev}.
The weak gradient of functions in $H^{1,a}(B_R^*)$ is indeed well defined on the full-measure set $B_R^* \setminus \Sigma_{\rm sd}$, see Lemma \ref{lemH=W} and Lemma \ref{thmH=W}. Moreover, the other two conditions are provided below.

\begin{Lemma}[Doubling property]
    Let $a \in (-1, \infty)^n$. Then the measure $d\mu(z)=\omega^a(y)dz$ is doubling; that is, there exists a constant $C>0$ depending only on $a,n,d$ such that for any $z_0\in\R^d_*$ and any $r>0$
\begin{equation*}
    \int_{B_{2r}(z_0)\cap\R^d_*}\omega^a(y)dz\leq C\int_{B_{r}(z_0)\cap\R^d_*}\omega^a(y)dz.
\end{equation*}
\end{Lemma}
\begin{proof}
    Let $z_0=(x_0,y_0)$, with $y_0=(y_0^1,\dots,y_0^n)$.
    Observe that 
    \[
    B_{2r}(z_0)\subset B^{d-n}_{2r}(x_0)\times_{i=1}^nB^1_{2r}(y_0^i) \qquad \text{ and } \qquad B^{d-n}_{r/\sqrt{n+1}}(x_0)\times_{i=1}^nB^1_{r/\sqrt{n+1}}(y_0^i)\subset B_r(z_0)\,.
    \]
On the other hand, the measure of a product set decomposes as 
\[
\mu\Big(B^{d-n}_{\rho}(x_0)\times_{i=1}^nB^1_{\rho}(y_0^i)\cap\R^d_*\Big) = |B^{d-n}_\rho(x_0)|_{d-n}\prod_{i=1}^n\mu_i(B^1_\rho(y_0^i)\cap(0,\infty))\,,
\]
where $|\cdot|_{d-n}$ is the $(d-n)$-dimensional Lebesgue measure and
\[
\mu_i(B^1_\rho(y_0^i)\cap(0,\infty))=\int_{\max\{0,y^i_0-\rho/2\}}^{y^i_0+\rho/2}y_i^{a_i}dy\,.
\]
To conclude, one can now easily check the doubling condition for each $\mu_i$.
\end{proof}

The following result holds true only in the orthant, since whenever $a_i\geq1$ for some $i=1,\dots,n$, elements in the relevant weighted Sobolev space may jump across $\Sigma_i$, thus providing a counterexample to the Poincar\'e-Wirtinger inequality (see \cite[Example 1.4]{SirTerVit21a}).

\begin{Lemma}[Poincar\'e-Wirtinger inequality]
    Let $a \in (-1, \infty)^n$ and $R>0$. Then there exists a positive constant $c>0$ depending only on $d,n,a$ such that for any $u\in H^{1,a}(B_R^*)$
    \begin{equation*}
    c\int_{B^*_R}\omega^a|u-\langle u\rangle_R^a|^2 \,dz\leq R^2\int_{B^*_R}\omega^a|\nabla u|^2 \,dz
\end{equation*}
where
\begin{equation*}
    \langle u\rangle_R^a:=\frac{1}{\mu_a(B^*_R)}\int_{B^*_R}\omega^a u \,dz,\qquad \mathrm{and}\qquad  \mu_a(B^*_R)=\int_{B^*_R}\omega^a \,dz.
\end{equation*}
\end{Lemma}
\begin{proof}
  The proof is classical and can be easily adapted from \cite[Section 5.8.1]{Eva10}.
\end{proof}

As a consequence of the $2$-admissibility of the weight $\omega^a(y)$ one can adapt the theory developed in \cite{FabKenSer82} and \cite[Sections 3 and 6]{HeiKilMar06} to obtain the following regularity result.

\begin{Proposition}
    Let $a \in (-1, \infty)^n$, $p>\frac{d+\langle a^+\rangle}{2}$, $q>d+\langle a^+\rangle$. Let $f\in L^{p,a}(B_1^*)$, $F\in (L^{q,a}(B_1^*))^d$ and $A$ satisfies \[
\|A\|_{L^\infty(B_1)}\le \Lambda,\quad  A(z)\xi\cdot \xi \ge \lambda|\xi|^2,\quad \text{ for a.e. }z\in B_1, \text{ for any }\xi\in\R^d,
\]
for some constants $0<\lambda\leq\Lambda$. Let $u$ be a weak solution to \eqref{eq:degenerate:equation:B*} in the sense of \eqref{eq5.1}. Then, there exists a constant $\alpha^*\in(0,1)$ depending on $d,\lambda,\Lambda,a$ such that for every $\alpha\in(0,\alpha^*]\cap(0,2-\frac{d+\langle a^+\rangle}{p}]\cap(0,1-\frac{d+\langle a^+\rangle}{q}]$ there exists a constant $C>0$ depending only on $d,\lambda,\Lambda,a,p,q,\alpha$ such that
\[
            \|u\|_{C^{0,\alpha}(B_{1/2}^*)}\le C \big(
            \|u\|_{L^{2,a}(B_1^*)}+ \|f\|_{L^{p,a}(B_1^*)} +\|F\|_{L^{q,a}(B_1^*)}
            \big).
\]
\end{Proposition}

\section{Weak solutions and approximation results}\label{S:solution}

Let us introduce the weighted problems considered in this work, together with their notion of weak solution.
\begin{itemize}
\item[$-$] For $R>0$, we say that $u$ is a weak solution to
\begin{align}\label{eq5.1}
\begin{cases}
    -\div(\omega^a A\nabla u)=\omega^af+\div(\omega^a F) & \text{in }B^*_R\,,\\
    \omega^a(A\D u +F)\cdot e_{y_i} = 0,  & \text{on } \partial B^*_R \cap \Sigma_i, \text{ for } i=1,\dots,n,
    \end{cases}
\end{align}
if and only if $u \in H^{1,a}(B_R^*)$ satisfies 
\[
\int_{B^*_R}\omega^a A\nabla u \cdot \nabla \phi \,dz = \int_{B^*_R}\omega^a f \phi \,dz - \int_{B^*_R}\omega^a F \cdot \nabla \phi \,dz\,, \qquad \text{ for all }\phi \in C^\infty_c(B_R)\,.
\]
\item[$-$] We say that $u$ is a weak entire solution to 
\[
\begin{cases}
    -\div(\omega^a A\nabla u)=\omega^af+\div(\omega^a F) & \text{in }\R^d_*\,,\\
    \omega^a(A\D u +F)\cdot e_{y_i} = 0,  & \text{on } \partial \R^d_* \cap \Sigma_i, \text{ for } i=1,\dots,n,
    \end{cases}\]
if and only if $u \in H^{1,a}_\loc(\R^d_*)$ is a weak solution to \eqref{eq5.1} for every $R>0$. 

\item[$-$]  For $R>0$ we say that $u$ is a weak solution to
\begin{equation}\label{eq:ballprob}
    -\div(\omega^a A\nabla u)=\omega^af+\div(\omega^a F) \qquad  \text{in }B_R\,,
\end{equation}
if and only if $u \in H^{1,a}(B_R)$ satisfies
\[
\int_{B_R}\omega^a A\nabla u \cdot \nabla \phi \,dz = \int_{B_R}\omega^a f \phi \,dz - \int_{B_R}\omega^a F \cdot \nabla \phi \,dz\,, \qquad  \text{ for all }\phi \in C^\infty_c(B_R)\,.
\]
\item[$-$] We say that $u$ is a weak entire solution to 
\[
    -\div(\omega^a A\nabla u)=\omega^af+\div(\omega^a F) \qquad  \text{in }\R^d\,,
\]
if and only if $u \in H^{1,a}_\loc(\R^d)$ is a weak solution to \eqref{eq:ballprob} for every $R>0$. 
\end{itemize}

Goal of the remainder part of this section is to show that solutions to \eqref{eq5.1} and \eqref{eq:ballprob} can be obtained as limits of solutions to suitable regularized problems. To this end, let $\eps_k$ be a sequence such that 
\begin{equation}\label{eqepsdef}
\eps_k \in [0,1]^n\,, \qquad \text{ and }\qquad |\eps_k| \to 0 \quad \text{as } k \to \infty\,.
\end{equation}
We also allow sequences for which $(\varepsilon_k)_j=0$ for all $k$, for some fixed index $j$. The associated regularized weight is defined by
\begin{align}\label{eqomkdef}
\omega^a_k:= \omega_{\eps_k}^a = \prod_{i=1}^n((\varepsilon_k)_i^2+y_i^2)^\frac{a_i}{2}\,.
\end{align}
We begin with the case of the whole ball $B_R$. Let $u_k$ be a weak solution to 
\begin{align}\label{eq5.2}
    -\div(\omega^a_k A\nabla u_k)=\omega^a_kf_k+\div(\omega^a_kF_k) \qquad \text{in }B_R\,.
\end{align}
Our first result shows that any bounded sequence of solutions to \eqref{eq5.2} converges to a solution to \eqref{eq:ballprob}.
\begin{Lemma}\label{lemA.1}
Let $a\in (-1,\infty)^n$, $p,q\geq 2$, and let $\eps_k$ and $\omega_k$ be as in \eqref{eqepsdef} and \eqref{eqomkdef}, respectively. Let $A$ satisfy Assumption \eqref{Ass:matrix}, $i)$, and suppose $f_k \in L^{p, a, \eps_k}(B_R)$, $F_k \in (L^{q, a, \eps_k}(B_R))^d$ with
    \[
    \|f_k\|_{L^{p, a, \eps_k}(B_R)}+\|F_k\|_{L^{q,a, \eps_k}(B_R)}\leq c.
    \] 
    Let $u_k\in H^{1,a,\eps_k}(B_R)$ be a sequence of weak solutions to \eqref{eq5.2} satisfying 
    \[
    \|u_k\|_{H^{1, a, \eps_k}(B_R)}\leq c\,.
    \]
    Assume further that there exist $f\in L^{p,a}(B_R)$ and $F\in (L^{q,a}(B_R))^d$ such that
    \[
    f_k\rightarrow f \quad\text{in }L^p_{\loc}(B_R\setminus \Sigma_0)\,, \qquad F_k\rightarrow F \quad\text{in }(L^q_{\loc}(B_R\setminus \Sigma_0))^d\,.
    \]
    Then, up to a subsequence, 
    \[
    u_k\rightarrow u \quad\text{in }H^1_{\loc}(B_R\setminus \Sigma_0)\,,
    \]
    where $u\in H^{1,a}(B_R)$ is a weak solution to \eqref{eq:ballprob}.
Moreover, if $u_k\in H^{1,a}_0(B_R)$, then $u\in H^{1,a}_0(B_R)$.
\end{Lemma}
\begin{proof}
We refer to \cite[Lemma 2.12]{SirTerVit21a} and \cite[Lemma 4.2]{AudFioVit24} for a proof in a similar context.
After a suitable rescaling, we may assume without loss of generality that $R=1$.
For every compact set $K \subset B_1\setminus \Sigma_0$, it holds $\|u_k\|_{H^1(K)}\leq \|u_k\|_{H^{1,a,\eps_k}(K)}\leq c$. Thus, there exists $u \in H^1_\loc(B_1 \setminus \Sigma_0)$ such that $u_k\rightharpoonup u$ in $H^1_{\loc}(B_1\setminus \Sigma_0)$ and $u_k\rightarrow u$ in $L^2_{\loc}(B_1\setminus \Sigma_0)$. Let $\phi \in C^\infty_c(B_1\setminus \Sigma_0)$ and test \eqref{eq5.2} with $\phi^2(u_k-u)$. After standard computations, one finds that $\nabla u_k \rightarrow \nabla u$ in $L^2_{\loc}(B_1\setminus \Sigma_0)$. Hence, up to a subsequence, $u_k\rightarrow u$ and $\nabla u_k\rightarrow \nabla u$ almost everywhere in $B_1$. Fatou's lemma then gives
\begin{equation}\label{eq5.3}
    \|u\|_{H^{1,a}(B_1)}\leq \|u_k\|_{H^{1,a,\eps_k}(B_1)}\leq  c.
\end{equation}
On the other hand, since $\omega^{a^+}\leq c\,\omega_k^a$ on $B_1$, we have
\[
     \|u_k\|_{H^{1, a^+}(B_1)}\leq  c\|u_k\|_{H^{1, a, \eps_k}(B_1)}\leq  c.
\]
Thus, weak convergence implies that $u\in H^{1, a^+}(B_1)$. Applying Lemma \ref{thmH=W} together with Lemma \ref{lemH=W}, we infer that $u\in W^{1,1}_{\loc}(B_1\setminus \Sigma_{\rm sd})$. Then \eqref{eq5.3} allows us to conclude that $u\in H^{1,a}(B_1)$, again employing Lemma \ref{thmH=W} and Lemma \ref{lemH=W}. Finally, Vitali's theorem tells us that $u$ is weak solution to \eqref{eq:ballprob}. 

Suppose now that $u_k\in H^1_0(B_1,\omega^a_k dz)$. By the previous argument, this implies $u\in H^{1, a^+}_0(B_1)$. Hence, $u\in H^{1, a^+}_0(B_1)\cap H^{1,a}(B_1)$. By Lemma \ref{thmtrace}, there exist two trace operators 
\begin{align*}
    T\,:\,H^{1,a}(B_1)&\rightarrow L^{2,a}(\partial B_1)\\
    T^+:\,H^{1,a^+}(B_1)&\rightarrow L^{2,a^+}(\partial B_1)
\end{align*}
and moreover, $T^+ u=0$ $d\mathcal{H}^{d-1}$-a.e. in $\partial B_1$. 

Next, let $\{u_j\}\subset C^\infty(\overline{B_1})$ be such that $\|u_j -u\|_{H^{1,a}(B_1)}\rightarrow 0$. Since $\omega^{a^+}\leq \omega^a$ in $B_1$, we also have $\|u_j -u\|_{H^{1,a^+}(B_1)}\rightarrow 0$. Therefore $\|T^+u_j -T^+u\|_{L^{2,a^+}(\partial B_1)}\rightarrow 0$ as well as $\|Tu_j -Tu\|_{L^{2,a}(\partial B_1)}\rightarrow 0$. Thus, up to a subsequence, $Tu_j \rightarrow T u$ and $T^+u_j \rightarrow T^+ u$ $d\mathcal{H}^{d-1}$ a.e. in $\partial B_1$. Since each $u_j$ is smooth, we have $T^+ u_j=Tu_j=u_{j}{\vert}_{\partial B_1}$, thus $Tu=0$ $d\mathcal{H}^{d-1}$ a.e. in $\partial B_1$. By Lemma \ref{thmtrace} it follows that $u\in H^{1,a}_0(B_1)$, which completes the proof of the lemma.
\end{proof}

We are now ready to prove the main approximation lemma. 
\begin{Lemma}\label{lemA.2}
Let $a\in (-1,\infty)^n$ and $p,q \geq 2$. Suppose that $A$ satisfies Assumption \eqref{Ass:matrix}, $i)$, and that $f \in L^{p, a}(B_R)$, $F \in (L^{q, a}(B_R))^d$, and let $u$ be solution to \eqref{eq:ballprob}.

Then, for every $0<r<r'\leq R$ and for every sequence $\eps_k$ as in \eqref{eqepsdef}, there exist sequences $\tilde f_k \in L^{2, a, \eps_k}(B_{r'})$, $\tilde F_k \in (L^{2,a,\eps_k}(B_{r'}))^d$, together with solutions $u_k\in H^{1,a,\eps_k}_0(B_{r'})$ to \eqref{eq5.2} in $B_{r'}$, such that
\begin{equation}\label{eq6.5}
\begin{aligned}
&\|\tilde f_k\|_{L^{2, a, \eps_k}(B_{r'})}+\|\tilde F_k\|_{L^{2,a, \eps_k}(B_{r'})} \leq c\Big( \|u\|_{H^{1, a}(B_{r'})}+\|f\|_{L^{p,a}(B_{r'})}+\|F\|_{L^{q,a}(B_{r'})} \Big)\,,\\
&\|u_k\|_{H^{1, a, \eps_k}(B_{r'})}\leq c\Big(\|u\|_{H^{1, a}(B_{r'})}+\|f\|_{L^{p,a}(B_{r'})}+\|F\|_{L^{q,a}(B_{r'})}\Big),
\end{aligned}
\end{equation}
where $c>0$ is a constant depending only on $d,a,\Lambda, \lambda, r',r$.

Moreover, up to a subsequence, 
\begin{equation}
\begin{aligned}\label{eq6.890}
&u_k\rightarrow u \quad\text{in } H^1_{\loc}(B_r\setminus \Sigma_0)\,,\\
&\tilde f_k\rightarrow f\quad \text{in }L^p_{\loc}(B_r\setminus \Sigma_0), \quad & &\|\tilde f_k\|_{L^{p, a, \eps_k}(B_r)}\leq c\|f\|_{L^{p, a}(B_r)}\,,\\
&\tilde F_k\rightarrow F\quad \text{in }(L^q_{\loc}(B_r\setminus \Sigma_0))^d, & & \|\tilde F_k\|_{L^{q, a, \eps_k}(B_r)}\leq c\|F\|_{L^{q, a}(B_r)}\,.
\end{aligned}
\end{equation}
If $F\in C^{0,\alpha}(B_R)$, one may choose $\tilde F_k$ so that $\tilde F_k\in C^{0,\alpha}(B_r)$, $\tilde F_k\vert_{B_r}=F$ and
\begin{equation}\label{eq6.5b}
   \|u_k\|_{H^{1,a, \eps_k}(B_{r'})}\leq c(\|u\|_{H^{1,a}(B_{r'})}+\|f\|_{L^{p,a}(B_{r'})}+\|F\|_{C^{0,\alpha}(B_{r'})}). 
\end{equation}
An analogous result holds if $f\in C^{0,\alpha}(B_R)$.

In addition, if $f=F=0$ and $r = r'/2$, then 
\begin{equation}\label{eq6.5c}
    \int_{B_{r'}}\omega^a_k|\nabla u_k|^2\,dz\leq c\Big(\int_{B_{r'}}\omega^a|\nabla u|^2 \, dz+(r')^{-2}\int_{B_{r'}}\omega^a u^2  \, dz \Big)\,,
\end{equation}
where $c>0$ is independent of $r'$. 
\end{Lemma}
\begin{proof}
We refer to \cite[Lemma 4.3]{AudFioVit24} and \cite[Lemma 4.7]{CorFioVit25} for a proof in a similar context.
Let $\xi\in C^\infty_c(B_{r'})$ be such that 
\begin{equation}\label{eq:cutoffpropapprox}
\xi\equiv 1 \text{ in } B_r\,, \qquad 0 \leq \xi \leq 1\,, \qquad  |\nabla \xi|\leq c|r'-r|^{-1}\,.
\end{equation}
Define $\tilde u:=\xi u\in H^{1,a}_0(B_{r'})$. Then $\tilde u$ is the unique solution to
\begin{equation}\label{eq6.6}
\begin{cases}
    -{\rm div}(\omega^a A \nabla \tilde u)=\omega^a (\tilde f+ \tilde g)+{\rm div}(\omega ^a (\tilde F+\tilde G)) & \text{ in }B_{r'}\\
    u= 0 & \text{ on }\partial B_{r'} 
    \end{cases}
\end{equation}
where 
\[
\tilde f:=f \xi\,, \qquad \tilde{g}:=-F\cdot\nabla \xi -A\nabla u\cdot \nabla \xi\,,\qquad \tilde F:=F \xi\,, \qquad \tilde G:=-u A\nabla \xi\,.
\]
Now fix any sequence $\eps_k$ as in \eqref{eqepsdef}, and let $\omega^a_k$ be as in \eqref{eqomkdef}. Define 
\[
    f_k:=\left(\frac{\omega^{a^+}}{\omega^{a^+}_k}\right)^\frac{1}{p} \tilde f\,,\qquad g_k:= \left(\frac{\omega^{a^+}}{\omega^{a^+}_k}\right) ^\frac{1}{2} \tilde g\,,\qquad F_k:= \left(\frac{\omega^{a^+}}{\omega^{a^+}_k}\right) ^\frac{1}{q} \tilde F\,, \qquad G_k:=\left(\frac{\omega^{a^+}}{\omega^{a^+}_k}\right) ^\frac{1}{2} \tilde G\,,
\]
and let $u_k\in H^{1,a, \eps_k}_0(B_{r'})$ denote the unique solution to 
\begin{align}\label{eq6.8}
    \begin{cases}
        -{\rm div}(\omega^a_kA\nabla u_k)=\omega_k^a\tilde f_k+{\rm div}(\omega_k^a\tilde F_k) & \text{ in }B_{r'}\\
        u_k= 0 & \text{ on }\partial B_{r'} \,.
    \end{cases}
\end{align}
where
\[
\tilde{f}_k:=f_k+g_k\,,\qquad \tilde F_k:=F_k+G_k\,.
\]
Next we check that $u_k$, $\tilde f_k$ and $\tilde F_k$ satisfy the assumptions of Lemma \ref{lemA.1} with $p = q = 2$. First of all, by H\"older inequality, \eqref{eq:cutoffpropapprox} and the elementary inequality $\omega^{a^+}\omega_k^{a-a^+} \leq \omega^a$, we obtain
\begin{align}\label{eq6.11}
\begin{split}
    &\|f_k\|_{L^{2,a,\eps_k}(B_{r'})}\leq c (d,a,r')\|f_k\|_{L^{p,a,\eps_k}(B_{r'})}\leq c (d,a,r')\|f\|_{L^{p,a}(B_{r'})}\,,\\
    &\|g_k\|_{L^{2,a,\eps_k}(B_{r'})} \leq \|\tilde g\|_{L^{2,a}(B_{r'})}\leq c(d,a, r')\|F\|_{L^{q,a}(B_{r'})}+c(\Lambda)|r'-r|^{-1}\|\nabla u\|_{L^{2,a}(B_{r'})},\\
    &\|F_k\|_{L^{2,a,\eps_k}(B_{r'})}\leq c (d,a, r')\|F_k\|_{L^{q,a,\eps_k}(B_{r'})}\leq c (d,a, r')\|F\|_{L^{q,a}(B_{r'})}\,,\\
    &\|G_k\|_{L^{2,a,\eps_k}(B_{r'})}\leq \|\tilde G\|_{L^{2,a}(B_{r'})}\leq c(\Lambda)|r'-r|^{-1}\|u\|_{L^{2,a}(B_{r'})}\,.
\end{split}
\end{align}
Thus, $\tilde f_k \in L^{2,a,\eps_k}(B_r')$ and $\tilde F_k \in (L^{2,a,\eps_k}(B_r'))^d$, and the first inequality in \eqref{eq6.5} holds. Moreover, it is immediate to see that 
 \[
    \tilde f_k\rightarrow \tilde f + \tilde g \quad\text{in }L^2_{\loc}(B_{r'}\setminus \Sigma_0)\,, \qquad \tilde F_k\rightarrow \tilde F + \tilde G \quad\text{in }(L^2_{\loc}(B_{r'}\setminus \Sigma_0))^d\,.
 \]

Now, test \eqref{eq6.8} with $u_k\in H^{1,a, \eps_k}_0(B_{r'})$. By Assumption \eqref{Ass:matrix} and Young's inequality we get
\[
    \int_{B_{r'}}\omega_k^a|\nabla u_k|^2 dz\leq \frac{1}{\lambda^2\delta}\|\tilde f_k\|^2_{L^{2,a,\eps_k}(B_r')} +\frac{1}{\lambda^2}\|\tilde F_k\|^2_{L^{2,a,\eps_k}(B_r')}+\frac{\delta}{4}\int_{B_{r'}}\omega^a_k|u_k|^2 dz+\frac{1}{4}\int_{B_{r'}}\omega^a_k|\nabla u_k|^2  dz\,,
\]
for all $\delta>0$. Choosing $\delta = C_P^2$, where $C_P$ is the constant of the weighted Poincar\'e inequality (see Lemma \ref{lemRpoinc}), gives
\begin{align}\label{eq6.10}
    \int_{B_{r'}}\omega^a_k|\nabla u_k|^2  dz&\leq c(a, \lambda)\left((r')^2\|\tilde f_k\|^2_{L^{2,a,\eps_k}(B_r')} +\|\tilde F_k\|^2_{L^{2,a,\eps_k}(B_r')}\right)\,,
\end{align}
Using again Lemma \ref{lemRpoinc}, together with \eqref{eq6.11} and \eqref{eq6.10}, we obtain the second estimate in \eqref{eq6.5}.

We can apply Lemma \ref{lemA.1} and deduce that there exists $\overline u\in H^{1,a}_0(B_{r'})$ solution to \eqref{eq6.6} in $B_{r'}$ such that $u_k \rightarrow \overline u$ in $H^1_{\loc}(B_{r'}\setminus \Sigma_0)$. By uniqueness of solutions to the Dirichlet problem, $\overline{u}=\tilde u$. But by definition, $\tilde u=u$ in $B_r$, hence $u_k \rightarrow u$ in $H^1_{\loc}(B_{r}\setminus \Sigma_0)$, as needed. To establish the final convergence statements and estimates in \eqref{eq6.890}, it is sufficient to remark that $\tilde f_k=f_k$ and $\tilde F_k=F_k$ in $B_r$.

If $F\in C^{0,\alpha}(B_R)$ (or $f\in C^{0,\alpha}(B_R)$) the same proof, with minor modifications, gives \eqref{eq6.5b}, simply take $F_k = \tilde F$ (or $f_k = \tilde f$). 
To conclude the proof we note that, if $f = F= 0$ then \eqref{eq6.5c} follows directly from \eqref{eq6.10} and \eqref{eq6.11}. 
\end{proof}

Next, we present the counterparts of Lemma \ref{lemA.1} and \ref{lemA.2} on $B_R^*$. Let $\eps_k$ and $\omega_k$ be defined as in \eqref{eqepsdef} and \eqref{eqomkdef}, respectively, and consider the equation
\begin{align}\label{eq5.2*}
\begin{cases}
    -\div(\omega^a_k A\nabla u_k)=\omega^a_kf_k+\div(\omega^a_kF_k)\quad &\text{in }B_R^*\\
   \omega^a_k(A\D u_k +F_k)\cdot e_{y_i} = 0,  & \text{on }\partial B^*_R \cap \Sigma_i, \text{ for } i=1,\dots,n,
\end{cases}
\end{align} 
Then the following approximation results hold. Their proofs are identical to those of Lemmas \ref{lemA.1} and \ref{lemA.2}, respectively, and are therefore omitted. Indeed, it suffices to replace $B_R$ with $B_R^*$ throughout the arguments and to invoke Lemma \ref{thmeqtrace*} and Lemma \ref{L:Hevenid} when needed. 
\begin{Lemma}
Let $a\in (-1,\infty)^n$, $p,q\geq 2$, and let $\eps_k$ and $\omega_k$ be as in \eqref{eqepsdef} and \eqref{eqomkdef}, respectively. Let $A$ satisfy Assumption \eqref{Ass:matrix}, $i)$, and suppose $f_k \in L^{p, a, \eps_k}(B_R^*)$, $F_k \in (L^{q, a, \eps_k}(B_R^*))^d$ with
    \[
    \|f_k\|_{L^{p, a, \eps_k}(B^*_R)}+\|F_k\|_{L^{q,a, \eps_k}(B^*_R)}\leq c.
    \] 
    Let $u_k\in H^{1,a,\eps_k}(B^*_R)$ be a sequence of weak solutions to \eqref{eq5.2*} satisfying 
    \[
    \|u_k\|_{H^{1, a, \eps_k}(B_R^*)}\leq c\,.
    \]
    Assume further that there exist $f\in L^{p,a}(B^*_R)$ and $F\in (L^{q,a}(B^*_R))^d$ such that
    \[
    f_k\rightarrow f \quad\text{in }L^p_{\loc}(B^*_R)\,, \qquad F_k\rightarrow F \quad\text{in }(L^q_{\loc}(B^*_R))^d\,.
    \]
    Then, up to a subsequence, 
    \[
    u_k\rightarrow u \quad\text{in }H^1_{\loc}(B^*_R)\,,
    \]
    where $u\in H^{1,a}(B^*_R)$ is a weak solution to \eqref{eq5.1}.
Moreover, if $u_k\in \hat H^{1,a}(B^*_R)$, then $u\in \hat H^{1,a}(B^*_R)$.

\end{Lemma}
\begin{Lemma}\label{lemA.2*}
Let $a\in (-1,\infty)^n$ and $p,q \geq 2$. Suppose that $A$ satisfies Assumption \eqref{Ass:matrix}, $i)$, and that $f \in L^{p, a}(B^*_R)$, $F \in (L^{q, a}(B^*_R))^d$, and let $u$ be solution to \eqref{eq:ballprob}.

Then, for every $0<r<r'\leq R$ and for every sequence $\eps_k$ as in \eqref{eqepsdef}, there exist sequences $\tilde f_k \in L^{2, a, \eps_k}(B^*_{r'})$, $(\tilde F_k \in L^{2,a,\eps_k}(B^*_{r'}))^d$, together with solutions $u_k\in \hat H^{1,a,\eps_k}(B^*_{r'})$ to \eqref{eq5.2} in $B^*_{r'}$, such that
\begin{equation*}
\begin{aligned}
&\|\tilde f_k\|_{L^{2, a, \eps_k}(B^*_{r'})}+\|\tilde F_k\|_{L^{2,a, \eps_k}(B^*_{r'})} \leq c\Big( \|u\|_{H^{1, a}(B^*_{r'})}+\|f\|_{L^{p,a}(B^*_{r'})}+\|F\|_{L^{q,a}(B^*_{r'})} \Big)\,,\\
&\|u_k\|_{H^{1, a, \eps_k}(B^*_{r'})}\leq c\Big(\|u\|_{H^{1, a}(B^*_{r'})}+\|f\|_{L^{p,a}(B^*_{r'})}+\|F\|_{L^{q,a}(B^*_{r'})}\Big),
\end{aligned}
\end{equation*}
where $c>0$ is a constant depending only on $d,a,\Lambda, \lambda, r',r$.

Moreover, up to a subsequence, 
\begin{equation*}
\begin{aligned}
&u_k\rightarrow u \quad\text{in } H^1_{\loc}(B^*_r)\,,\\
&\tilde f_k\rightarrow f\quad \text{in }L^p_{\loc}(B^*_r), \quad & &\|\tilde f_k\|_{L^{p, a, \eps_k}(B^*_r)}\leq \|f\|_{L^{p, a}(B^*_r)}\,,\\
&\tilde F_k\rightarrow F\quad \text{in }L^q_{\loc}(B^*_r), & & \|\tilde F_k\|_{L^{q, a, \eps_k}(B^*_r)}\leq \|F\|_{L^{q, a}(B^*_r)}\,.
\end{aligned}
\end{equation*}
If $F\in C^{0,\alpha}(B^*_R)$, one may choose $\tilde F_k$ so that $\tilde F_k\in C^{0,\alpha}(B^*_r)$, $\tilde F_k\vert_{B^*_r}=F$ and
\begin{equation*}
   \|u_k\|_{H^{1,a, \eps_k}(B^*_{r'})}\leq c(d,a,\Lambda, \lambda, r',r)(\|u\|_{H^{1,a}(B^*_{r'})}+\|f\|_{L^{p,a}(B^*_{r'})}+\|F\|_{C^{0,\alpha}(B^*_{r'})}).  
\end{equation*}
An analogous result holds if $f\in C^{0,\alpha}(B^*_R)$.

In addition, if $f=F=0$ and $r = r'/2$, then 
\begin{equation*}
    \int_{B^*_{r'}}\omega^a_k|\nabla u_k|^2\,dz\leq c\Big(\int_{B^*_{r'}}\omega^a|\nabla u|^2 \, dz+(r')^{-2}\int_{B^*_{r'}}\omega^a u^2  \, dz \Big)\,,
\end{equation*}
where $c>0$ is independent of $r'$.
\end{Lemma}

\section{Liouville Theorem for weighted equations in the orthant}\label{S:liouville}

The goal of this section is to prove the following theorem.

\begin{teo}[Liouville Theorem in the orthant]\label{T:Liouville}
Let $a \in (-1, \infty)^n$, $\eps \in [0,1]^n$, and let $\omega^a_\eps$ be as in \eqref{eq:regweight}.
Consider a constant uniformly elliptic matrix of the form
\[
    A:= \left(\begin{array}{cc}
       P & Q  \\ 
       R & S\\ 
  \end{array}\right),
\]
where $P$ is $(d-n)\times (d- n)$-dimensional, $Q$ is $(d-n)\times n$-dimensional, $R$ is $n\times (d- n)$-dimensional, and $S$ is $n\times n$-dimensional and diagonal.

Let $u$ be a weak entire solution to
\begin{equation}\label{eq:entire:liouville}
\begin{cases}
-\dive(\omega_\e^a A\D u) = 0, & \text{in }\R^d_* \\
\omega_\e^a A\D u\cdot e_{y_i} = 0,  & \text{on } \partial\R^d_* \cap \Sigma_i\,, \ \text{ for } \ i=1,\dots,n,
\end{cases}
\end{equation}
satisfying the growth condition
\begin{equation}\label{eq:growth}
    |u(z)|\le C(1+|z|^\gamma) \qquad  \text{for a.e. $z \in \R^d$}\,,
\end{equation}
for some constants $C>0$ and $\gamma\in [0,2)$. Then $u$ is an affine function.
Moreover, if $\gamma\in [0,1)$, then $u$ is constant.
\end{teo}

We begin with a few auxiliary results. At first, let us state a weighted Caccioppoli inequality, whose proof is standard.
\begin{Lemma}\label{lemwecacc}
  Let $a \in (-1, \infty)^n$, and $\varepsilon\in [0,1]^n$. Let $R>0$, and let $u_\varepsilon$ be a weak solution to 
\[
\begin{cases}
-\dive(\omega_\e^a A\D u_\varepsilon) = \omega_\varepsilon^af_\varepsilon+\div(\omega_\varepsilon^a F_\varepsilon), & \text{in }B_R^* \\
\omega_\e^a (A\D u_\varepsilon+F_\varepsilon)\cdot e_{y_i} = 0,  & \text{on } \partial B_R^* \cap \Sigma_i\,, \ \text{ for } \ i=1,\dots,n\,.
\end{cases}
\]
Then there exists a constant $c >0$, depending only on $d,\lambda,\Lambda, R$ such that for every $r \in (0, R)$ it holds
\begin{equation}\label{eqwecacc}
      \int_{B_r^*}\omega^a_\varepsilon |\nabla u_\varepsilon|^2  dz \leq C\Big(\frac{1}{(R-r)^2}\int_{B_R^*}\omega_\varepsilon^a u_\varepsilon^2  dz +\int_{B_R^*}\omega_\varepsilon^a f_\varepsilon^2 \,dz +\int_{B^*_R}\omega_\varepsilon^a|F_\varepsilon|^2dz \Big)\,.
\end{equation}
Moreover, an analogous estimate holds for weak solutions to 
\[
-\dive(\omega_\e^a A\D u_\varepsilon) = \omega_\varepsilon^af_\varepsilon+\div(\omega_\varepsilon^a F_\varepsilon) \qquad  \text{in }B_R\,.
\]
\end{Lemma}
\begin{remark}\label{R:cacsupsing}
Let us remark that Lemma \ref{lemwecacc} remains valid even if some of the factors of $\omega^a_\eps$ are in the form $\rho^{a_i}_{\eps_i}$ with $a_i \leq -1$, provided that $\varepsilon_i>0$.
\end{remark}

 The next result concerns the variables with respect to which the weight is translation invariant, namely the $x$-variables. Its proof, which is standard and therefore omitted, relies on the method of difference quotients and proceeds via an iterative application of the Caccioppoli inequality \eqref{eqwecacc} (see \cite[Corollary 4.2, Lemma 4.3]{TerTorVit24}).
\begin{Lemma}\label{L:x:derivatives}
    Let $u$ be an entire solution to \eqref{eq:entire:liouville}. Then, for every $j=1,\dots,d-n$ the function $\partial_{x_j}u$ is itself an entire solution to the same problem. Moreover, if $u$ satisfies \eqref{eq:growth} for some $\gamma\ge0$, then $u$ is polynomial in the $x$-variable with degree at most $\lfloor \gamma \rfloor$.
\end{Lemma}

The final preparatory lemma addresses the degenerate
$y$-variables. It shows that the second-order weighted derivative of a solution satisfies the same equation as the solution itself (see \cite[Lemma 5.2]{AudFioVit24} for a proof in a similar context).
\begin{Lemma}\label{L:y:derivate:liouville}
Let $a \in (-1, \infty)^n$ and $\eps \in [0,1]^n$. Suppose $u$ is an entire solution to 
\begin{equation}\label{eq:entire:liouville:y}
-\dive(\omega_\e^a \D u) = 0 \qquad \text{in }\R^d\,.
\end{equation}
Then, for each $i = 1, \ldots, n$, the function 
\[
w:=\rho_{\eps_i}^{-a_i}\partial_{y_i} (\rho_{\eps_i}^{a_i} \partial_{y_i} u)
\]
is also an entire solution to \eqref{eq:entire:liouville:y}. Moreover, for every $R>0$ there exists a constant $c>0$, depending only on $d$ and $a$, such that 
\begin{equation}\label{eq:double:derivate:liouvlle}
    \int_{B_R} \omega_\e^a  w^2 dz\le \frac{c}{R^4} \int_{B_{8R}}\omega_\e^a   u^2  dz\,.
\end{equation}
\end{Lemma}

\begin{proof}
Fix $R>0$ and choose a sequence $\tau_k \in [0,1]^n$ such that 
\[
\tau_k \to 0 \text{ as }k \to \infty\,, \qquad \qquad (\tau_k)_i >0 \text{ if and only if } \eps_i = 0\,.
\]
Let $\tilde \omega_\eps^a$ denote the product of the factors in $\omega^a_\eps$ corresponding to indices $i$ with $\eps_i >0$. 
Applying Lemma \ref{lemA.2} with $\eps_k = \tau_k$, $r' = 4R$, $r = 2R$, $A = \tilde\omega^a_\eps  I_d$, and $f= F = 0$, we regularize the degenerate monomials $y_i^{a_i}$ in the weight $\omega^a_\eps$. We obtain a family of solutions $u_{k}$ to    
    \[
-\dive(\omega_{k}^a \D u_k) = 0, \quad \text{in } B_{{2R}},
\]
where 
\[
\omega_{k}^a=\omega^a_{\eps + \tau_k} =  \prod_{i=1}^n \rho_{(\eps + \tau_k)_i}(y_i)^{a_i}\,.
\]
Moreover, $u_k\to u$ in $H^{1}_\loc(B_{2R}\setminus \Sigma_0)$, and by \eqref{eq6.5c}, we have
\begin{equation}\label{eq:y:energy:liouville}
    \int_{B_{2R}}\omega_{k}^a 
|\D u_k|^2 dz
\le c\Big(
\int_{B_{4R}} \omega_\e^a
|\D u|^2 \,dz+ R^{-2} \int_{B_{4R}} \omega_\e^a
  u^2  dz\Big),
\end{equation}
with $c>0$ independent of $R$.

Without loss of generality, assume $i = n$, and set $\rho_k = \rho_{\eps_n + (\tau_k)_n}$. Since weak solutions to the regularized problem are smooth in $B_{4R}$ (the equation is uniformly elliptic with smooth coefficients), the function
\[
v_k:=\rho_{k}^{a_n} \partial_{y_n} u_k
\]
is a weak solution to 
\[
-\div(\omega_{k}^a \rho_{k}^{-2a_n} \D v_k) = 0, \quad \text{ in } B_{2R}\,.
\]
The Caccioppoli inequality \eqref{eqwecacc} gives
\begin{equation}\label{eq:y:caccioppoli:1}
    \int_{B_{3R/2}}\omega_{k}^a \rho_{k}^{-2a_n} |\D v_k|^2 \,dz \le \frac{c}{R^2}  \int_{B_{2R}}\omega_{k}^a \rho_{k}^{-2a_n} v_k^2 \, dz\le  \frac{c}{R^2}  \int_{B_{2R}} \omega_{k}^a  |\D u_k|^2.
\end{equation}
Repeating the argument and exploiting the smoothness of the coefficients, we deduce that
\[
w_k:= \rho_{k}^{-a_n} \partial_{y_n} v_k = \rho_{k}^{-a_n}\partial_{y_n} (\rho_{k}^{a_n} \partial_{y_n} u_k)
\]
is a smooth solution to the same equation as $u_k$, namely 
\[
-\dive(\omega_{k}^a \D w_k) = 0, \quad \text{in } B_{2R}.
\]
Using \eqref{eq:y:energy:liouville}, \eqref{eq:y:caccioppoli:1} and the Caccioppoli inequality \eqref{eqwecacc}, we obtain
\[
\begin{aligned}
\int_{B_{3R/2}}\omega_{k}^a  w_k^2 \,dz & \leq \int_{B_{3R/2 }}\omega_{k}^a \rho_k^{-2a_n} |\nabla v_k|^2 \,dz\le
\frac{c}{R^{2}}
\int_{B_{2R}} \omega_{k}^a  |\D u_k|^2 dz\\
&\le c\Big(\frac{1}{R^{2}}
 \int_{B_{4R}} \omega_\e^a
|\D u|^2 dz+ \frac{1}{R^{4}} \int_{B_{4R}} \omega_\e^a
  u^2 dz\Big)
\le
\frac{c}{R^{4}} \int_{B_{8R}}   \omega_{\e}^a u^2\,.
\end{aligned}
\]
The above estimate, together with \eqref{eqwecacc}, also gives $\|w_k\|_{H^{1, a ,\eps+ \tau_k}(B_{R})}\le C_R$. Hence, 
by Lemma \ref{lemA.1} we deduce that $w_k \to \tilde w$ in $H^{1}_\loc(B_R \setminus \Sigma_0)$, where $\tilde w$ solves
\[
-\div(\omega_\e^a \D \tilde  w) = 0 ,\quad \text{ in }B_{R},
\]
and, by Fatou's lemma, 
\[
\int_{B_R}\omega_\e^a  |\tilde w|^2\, dz \le \liminf_{k\to\infty} 
\int_{B_R}\omega_{k}^a  w_k^2\,dz
\le 
\frac{c}{R^{4}} \int_{B_{8R}}   \omega_{\e}^a  u^2 \,dz\,.
\]
Finally, since $w_k \to\rho_{\eps_n}^{-a_n}\partial_{y_n} (\rho_{\eps_n}^{a_n} \partial_{y_n} u)$ a.e. in $B_R$, then $\tilde w = \rho_{\eps_n}^{-a_n}\partial_{y_n} (\rho_{\eps_n}^{a_n} \partial_{y_n} u)$. The claim follows.
\end{proof}

\begin{proof}[Proof of Theorem \ref{T:Liouville}]
Let $u$ be an entire solution to \eqref{eq:entire:liouville}. Since $\gamma \in [0,2)$ in \eqref{eq:growth}, Lemma \ref{L:x:derivatives} implies that $u$ is linear in the unweighted $x$-variables. Therefore, $u$ can be expressed as
\[
u(x,y)= u_0(y)+\sum_{j=1}^{d-n} u_j(y)x_j\,,
\]
where the coefficients $u_j=u_j(y)$ depend only on $y$. Moreover, each $u_j$ satisfies the same growth condition \eqref{eq:growth} with exponent $\gamma$. Indeed, we have
\[
|u_0(y)|=|u(0,y)|\le C(1+|y|^\gamma),
\]
and, for every $j=1,\dots,d-n$, 
\[
|u_j(y)| = |u(e_{x_j},y)-u_0(y)|\le |u(e_{x_j},y)|+|u_0(y)| \le C (1+|y|^\gamma).
\]

By Lemma \ref{L:x:derivatives}, each component $u_j = \partial_{x_j} u$, $j=1, \ldots, d-n$, is itself an entire solution to \eqref{eq:entire:liouville}. Since $u_j$ depends only on $y$, it satisfies
\[
\begin{cases}
-\dive(\omega_\e^a S\D u_j) = 0, & \text{in }\R^n_* \\
\omega_\e^a S\D u_j\cdot e_{y_i} = 0,  & \text{on } \partial\R^n_* \cap \Sigma_i\,, \ \text{ for } \ i=1,\dots,n\,.
\end{cases}
\]
Using Remark \ref{rem:even:reflection}, we can extend $u_j$ evenly across each $y_i$-coordinate hyperplane. Denoting this extension again by $u_j$, we then have
\[
-\div(\omega_\e^a S \D u_j) = 0, \qquad \text{in }\R^{n}.
\]
Finally, performing the change of variables $y'=S^{1/2}y$, where $S^{1/2}$ denotes the square root of $S$, we reduce the problem to the case $S=I_d$.

Let us write $ y = (y', y_n) \in \mathbb{R}^{n-1} \times \mathbb{R} $. By Lemma \ref{L:y:derivate:liouville}, the function  
\[
w_1 := \rho_{\eps_n}^{-a_n} \partial_{y_n}(\rho_{\eps_n}^{a_n} \partial_{y_n} u_j),
\]
is an entire solution to the same equation as $ u_j $, and satisfies the decay estimate \eqref{eq:double:derivate:liouvlle}. We can then apply Lemma \ref{L:y:derivate:liouville} iteratively by defining
\[
\begin{cases}
w_0 := u_j, \\
w_k := \rho_{\eps_n}^{-a_n} \partial_{y_n}(\rho_{\eps_n}^{a_n} \partial_{y_n} w_{k-1}), \quad \text{for } k \geq 1,
\end{cases}
\]
so that each $ w_k $ is a weak solution to the same equation as $ u_j $. Moreover, for every $R>0$ we have
\[
\int_{B_R} \omega_\varepsilon^a w_k^2 dz\leq \frac{c}{R^4} \int_{B_{8R}} \omega_\varepsilon^a w_{k-1}^2 dz \leq \ldots \leq \frac{c}{R^{4k}} \int_{B_{8^kR}} \omega_\varepsilon^a u_j^2 dz\,.
\]
Let $R\geq 1$. Since $u_j$ satisfies the growth condition \eqref{eq:growth} and $\omega^a_\eps \leq c(\eps) R^{\langle a^+\rangle}$ in $B_R$, we obtain 
\[
\int_{B_R} \omega_\varepsilon^a w_k^2  dz\leq c R^{2\gamma + d + \langle a^+ \rangle - 4k}.
\]
Choosing $ k $ sufficiently large and letting $ R \to \infty $, it follows that $ w_k \equiv 0 $. Hence
\begin{equation*}
\rho_{\eps_n}^{-a_n} \partial_{y_n}(\rho_{\eps_n}^{a_n} \partial_{y_n} w_{k-1}) = 0\,.
\end{equation*}
We now distinguish two cases. If $\eps_n = 0$, this is an ordinary differential equation in the variable $y_n$, which can be solved explicitly. Recalling that $w_{k-1} \in H^{1,a, \eps}_\loc(\R^n)$ by Lemma \ref{L:x:derivatives}, and that $w_{k-1}$ is even in $y_n$ by construction, we conclude that the only admissible solution is the one that does not depend on $y_n$. Hence,
\[
w_{k-1}(y) = g_{k-1}(y')\,,
\]
for some function $g_{k-1}:\R^{n-1} \to \R$. By iteratively solving the corresponding equations for $ w_{k-2}, \ldots, w_0 $ and imposing the evenness condition at each step, we get
\[
u_j(y) = \sum_{\ell=0}^{k-1} g_{\ell}(y')\, |y_n|^{2\ell}\,,
\]
where each $g_\ell$ depends only on $y'$.
It then follows from \eqref{eq:growth} that $ g_\ell = 0 $ for all $ \ell \geq 1 $, and hence $ u_j(y) = g_0(y') $, meaning that $ u_j $ does not depend on $ y_n $.

If $\eps_n >0$, arguing as \cite[Theorem 1.2]{AudFioVit24}, we obtain an expansion of the form
\[
u_j(y):= \sum_{\ell=0}^{k-1} 
g_{\ell}(y') \psi_{\ell}(y_n)\,,
\]
where the functions $ \psi_\ell $ are defined recursively by
\[
\begin{cases}\displaystyle
\displaystyle \psi_0(y_n) := 1\,,\\
   \displaystyle \psi_\ell(y_n):= \int_0^{y_n} \rho_{\e_n}^{-a_n}(s)
    \Big(\int_0^s \rho_{\e_n}^{a_n}(t) \psi_{\ell-1}(t) dt\Big)
    ds, & \text{for } \ell\ge 1.
\end{cases}
\]
Using that $y_n^{-1}\rho_{\eps_n} \to 1$ as $y_n\to \infty$, and applying de l'H\v{o}pital's rule twice, we find that $\psi_\ell$ is asymptotic to $|y_n|^{2\ell}$ as $|y_n| \to \infty$. Hence, by \eqref{eq:growth}, also in this case we conclude that $ u_j $ cannot contain any nontrivial term in $ y_n $. Repeating this argument for each index $i$, we conclude that $u_j$ is constant.

As a consequence, each $ u_j $ is constant, and therefore the original function $ u $ takes the form
\[
u(x, y) = u_0(y) + \beta \cdot x\,,
\]
for some $ \beta\in \mathbb{R}^{d-n}$.
Since the affine function $\beta\cdot x + \delta \cdot y$, with $\delta = -S^{-1}R\beta \in \R^n$, is itself an entire solution to \eqref{eq:entire:liouville}, by setting
\[
\bar u (y) := u(x, y) - \beta\cdot x - \delta\cdot y = u_0(y) - \delta \cdot y\,,
\]
we obtain another solution to \eqref{eq:entire:liouville} that satisfies the same growth condition \eqref{eq:growth}, and depends only on $y$. Applying to $\bar u$ the same argument previously used for $u_j$, we conclude that $\bar u$ is constant. Therefore, 
\[
u(x, y) = \alpha + \beta\cdot x + \delta \cdot y\,, 
\]
for some constants $\alpha \in \R$, $\beta \in \R^{d-n}$, $\delta \in \R^n$, that is, $u$ is an affine function. Finally, if $\gamma \in [0,1)$, the growth condition \eqref{eq:growth} forces $\beta = \delta = 0$. 
\end{proof}

\section{Regularity estimates}\label{S:regularity}

In this section, we first establish $\varepsilon$-stable estimates for weak solutions to the problem 
\begin{align}\label{eqeps}
\begin{cases}
-\div( \omega_\varepsilon^a A\nabla u_\varepsilon)=\omega^a_\varepsilon f_\varepsilon+\div(\omega_\varepsilon^aF_\varepsilon )\qquad &\text{in }B_1^*\\
\omega^a_\varepsilon(A\D u_\varepsilon +F_\varepsilon)\cdot e_{y_i} = 0,  & \text{on }\partial B^*_1 \cap \Sigma_i, \text{ for } i=1,\dots,n.
\end{cases}
\end{align}
We then prove the main theorem by combining these estimates with the approximation results established in Section \ref{S:solution}.

\subsection{Stable \texorpdfstring{$L^\infty$}{L} bounds}

We begin by proving $\varepsilon$-stable $L^2 \to L^\infty$ estimates.

\begin{Lemma}\label{thmdegiorgi}
Let $a \in (-1, \infty)^n$ and $\varepsilon\in [0,1]^n$, and let $u_\varepsilon$ be a family of weak solutions to \eqref{eqeps}. Assume that $p>\tfrac{d+\langle a^+ \rangle}{2}$ and $q>d+\langle a^+ \rangle$.
Suppose that $A$ satisfies Assumption \eqref{Ass:matrix}, $i)$. Then,
there exists a constant $c>0$, depending only on $d,p,q,\lambda,\Lambda,a$, such that
\[
\|u_\varepsilon\|_{L^\infty(B^*_{1/2})}\leq c\big(\|u_\varepsilon\|_{L^{2,a,\eps}(B^*_1)}+ \|f_\varepsilon\|_{L^{p, a, \eps}(B^*_1)}+\|F_\varepsilon\|_{L^{q,a,\eps}(B^*_1)}\big)\,.
\]
\end{Lemma}
\begin{proof}
The argument follows the classical De Giorgi method, and we therefore omit some of the details. We emphasize that the $\varepsilon$-stability of the estimate follows directly from the $\varepsilon$-stable Sobolev inequality (see Lemma \ref{L:Sobolev}). 

We first establish Caccioppoli-type inequalities for the truncated functions 
\[
v_\varepsilon=:(u_\varepsilon-b)_+ \quad \text{ and }\quad w_\varepsilon:=(u_\varepsilon-b)_-\,,
\]
where $b\in \R$. Since $u_\varepsilon\in H^{1, a, \eps}(B^*_1)$, it follows that $v_\varepsilon,w_\varepsilon\in H^{1, a, \eps}(B^*_1)$ . Moreover, we have
\[
\nabla v_\eps = 
\begin{cases}
\nabla u_\eps & v_\eps >0\,,\\
0 & v_\eps = 0\,,
\end{cases}
\quad \text{ and } \quad 
\nabla w_\eps = 
\begin{cases}
-\nabla u_\eps & w_\eps >0\,,\\
0 & w_\eps = 0\,.
\end{cases}
\]
Let $0<r<\rho\leq 1$ be fixed, and take a cutoff function $\eta\in C^\infty_c(B_\rho)$ such that 
\[
\eta\equiv1  \text{ in }B_r\,,\qquad  |\nabla \eta|\leq 2(\rho-r)^{-1}.
\]
Testing equation \eqref{eqeps} with $\eta^2 v_\varepsilon$ and $\eta^2 w_\varepsilon$, and performing standard computations, we obtain the inequalities
\begin{equation}\label{eqcactrunc}
c\int_{B^*_\rho}\omega_\varepsilon^a|\nabla(\eta v_\varepsilon)|^2  dz\leq \int_{B^*_\rho} \omega^a_\varepsilon  v_\varepsilon^2 |\nabla \eta|^2 dz + \int_{B^*_\rho }\omega^a_\varepsilon |v_\varepsilon||f_\varepsilon| \eta^2  dz + \int_{B^*_\rho \cap \{v_\varepsilon>0\}}\omega^a_\varepsilon|F_\varepsilon|^2  dz\,,
\end{equation}
and 
\[
c\int_{B^*_\rho}\omega_\varepsilon^a|\nabla(\eta w_\varepsilon)|^2  dz \leq \int_{B^*_\rho} \omega^a_\varepsilon w_\varepsilon^2 |\nabla \eta|^2 dz + \int_{B^*_\rho} \omega^a_\varepsilon|w_\varepsilon||f_\varepsilon| \eta^2  dz + \int_{B^*_\rho \cap \{w_\varepsilon>0\}}\omega^a_\varepsilon|F_\varepsilon|^2  dz\,.
\]
We now prove the so-called \emph{no spike lemma}. Assume that 
\begin{equation}\label{eq:spikehp}
\|f_\varepsilon\|_{L^{p, a, \eps}(B^*_1)}+\|F_\varepsilon\|_{L^{q, a, \eps}(B^*_1)}\leq 1.
\end{equation}
We claim that there exists $\delta\in (0,1)$ such that
\[
\begin{aligned}
&(i) & &\text{  if }\int_{B^*_1}\omega_\varepsilon^a(u_\varepsilon)_+^2  dz\leq \delta\,, \qquad \text{then } u_\varepsilon\leq 1 \text{ almost everywhere in }B^*_{1/2}\,,\\
&(ii) & &\text{  if }\int_{B^*_1}\omega_\varepsilon^a(u_\varepsilon)_-^2  dz\leq \delta\,, \qquad\text{then } u_\varepsilon\geq -1 \text{ almost everywhere in }B^*_{1/2}\,.
\end{aligned}
\]
We prove only $(i)$, since the argument for $(ii)$ is analogous. For each integer $k \geq 0$, set
\[
b_k=1-2^{-k}, \qquad r_k=2^{-1}(2^{-k}+1)\,.
\]
Then $b_0=0$, $b_k \nearrow b_\infty=1$, $r_0=1$ and $r_k\searrow r_\infty=\frac{1}{2}$, with $r_k-r_{k+1}=2^{-(k+2)}$.
Define the nested domains $D_k=B^*_{r_k}$, which satisfy
\[
D_{k+1}\subset B^*_{\rho_k}\subset D_k \qquad \text{where }\ \rho_k:=\frac{r_k+r_{k+1}}{2}\,.
\]
Next, introduce the truncated functions and their corresponding energies
\[
v_{\varepsilon,k}:=(u_\varepsilon-b_k)_+ \quad\text{and} \quad E_k=\int_{D_k}\omega^a_\varepsilon v_{\varepsilon,k}^2 dz\,.
\]
Since $v_{\varepsilon,k+1}\leq v_{\varepsilon, k}$, we have 
\[
E_{k+1}\leq E_k\leq \dots\leq E_0\leq \delta\,,
\]
where $\delta\in (0,1)$ is a constant to be chosen later. 
Finally, for each $k\geq 0$, let $\eta_k\in C^\infty_c(B_{\rho_k})$ be such that
\[
0\leq \eta_k\leq 1\,, \quad \eta_k\equiv 1 \text{ in } D_{k+1}\,,\quad  \text{ and } \quad |\nabla \eta_k|\leq 2(\rho_k-r_{k+1})^{-1}\leq c2^{k+1}\,. 
\]
Applying the Caccioppoli-type inequality \eqref{eqcactrunc} with $v_\varepsilon=v_{\varepsilon,k+1}$, $\eta=\eta_k$, $r=r_{k+1}$, $\rho=\rho_k$ we obtain
\[
    \int_{B^*_{\rho_k}}\omega_\varepsilon^a|\nabla (\eta_k v_{\varepsilon,k+1})|^2  dz \leq c(I_k^1+I_k^2+I_k^3)\,,
\]
where
\[
I^1_k = \int_{B^*_{\rho_k}}\omega_\varepsilon^av^2_{\varepsilon,k+1}|\nabla \eta_k|^2 dz, \qquad I^2_k= \int_{B^*_{\rho_k}}\omega_\varepsilon^av_{\varepsilon, k+1}|f_\varepsilon |\eta_k^2 dz\,,\qquad 
I^3_k = \int_{B^*_{\rho_k}\cap\{v_{\varepsilon,k+1}>0\}}\omega_\varepsilon^a |F_\varepsilon|^2 dz\,.
\]
Since $v_{\varepsilon,k+1}\leq v_{\varepsilon,k}$ and $B_{\rho_k}^*\subset D_k$, it follows immediately that 
\[
I^1_k\leq c 2^{2(k+1)}E_k\,.
\]
Next, we fix $\tau>2$ by setting
\[
\tau = 
\begin{cases}
2^*_a &\text{if }d+\langle a^+ \rangle> 2\,,\\
\text{any }\tau>\max\left\{ \frac{2q}{q-2},\frac{2p}{p-1}\right\} &\text{if }d+\langle a^+ \rangle=2\,,
\end{cases}
\] 
where $2^*_a$ is as in \eqref{eq:sobexp}. 
By H\"older's inequality, Lemma \ref{L:Sobolev}, \eqref{eq:spikehp}, and Young's inequality, we obtain
\[
\begin{aligned}
    I^2_k&\leq \|f_\varepsilon\|_{L^{p,a, \eps}(B^*_{\rho_k})}\Big( \int_{B^*_{\rho_k}}|\eta_k^2 v_{\varepsilon,k+1}|^\tau \omega^a_\varepsilon dz\Big)^\frac{1}{\tau}\Big( \int_{B^*_{\rho_k}} \chi _{\{v_{\varepsilon,k+1}>0\}} \omega^a_\varepsilon dz \Big)^{1-\frac{1}{\tau}-\frac{1}{p}}\\
    &\leq t \int_{B^*_{\rho_k}}|\nabla(\eta_kv_{\varepsilon,k+1})|^2\omega^a_\varepsilon dz+\frac{c}{t}\Big( \int_{B^*_{\rho_k}} \chi _{\{v_{\varepsilon,k+1}>0\}} \omega^a_\varepsilon dz \Big)^{2-\frac{2}{\tau}-\frac{2}{p}}\,,
\end{aligned}
\]
for some $t >0$ to be specified later. Since
\[
\{v_{\varepsilon,k+1}>0\}=\{u_\varepsilon-b_{k+1}>0\}=\{u_\varepsilon-b_k>b_{k+1}-b_k\}=\{v_{\varepsilon,k}>2^{-(k+1)}\}=\{v_{\varepsilon,k}^2>2^{-2(k+1)}\}\,,
\]
it follows that
\begin{equation}\label{eq:1functionest}
\int_{B^*_{\rho_k}} \chi _{\{v_{\varepsilon,k+1}>0\}} \omega^a_\varepsilon dz=\int_{B^*_{\rho_k}} \chi_{\{v_{\varepsilon,k}^2>2^{-2(k+1)}\}} \omega^a_\varepsilon dz\leq 2^{2(k+1)}E_k\,.
\end{equation}
Therefore, there exists a constant $C_1>1$, depending only on $a, p,\tau$, such that 
\[
I_k^2\leq t\int_{B^*_{\rho_k}}|\nabla(\eta_kv_{\varepsilon,k+1})|^2\omega^a_\varepsilon dz+\frac{C_1^{k+1}}{t}E_k^{2-\frac{2}{\tau}-\frac{2}{p}}\,.
\]
Finally, we use H\"older's inequality, \eqref{eq:spikehp} and \eqref{eq:1functionest} to get
\[
    I^3_k \leq \|F_\varepsilon\|^2_{L^{q, a, \eps}(B^*_{\rho_k})}\Big(\int_{B^*_{\rho_k}} \chi_{\{v_{\eps, k+1}>0\}}\omega_\varepsilon^adz\Big)^{1-\frac{2}{q}}\leq C_2^{k+1}E_k^{1-\frac{2}{q}}\,.
\]
Combining these estimates and choosing $t>0$ sufficiently small, we obtain that there exists a constant $C_0>1$ (depending only on $a, p,q,\tau$) such that 
\[
\int_{B^*_{\rho_k}}|\nabla(\eta_k v_{\varepsilon,k+1})|^2  \omega_\varepsilon^a dz \leq C_0^{k+1}(E_k+E_k^{2-\frac{2}{\tau}-\frac{2}{p}}+E_k^{1-\frac{2}{q}})\,.
\]
Furthermore, using again H\"older's inequality, Lemma \ref{L:Sobolev}, and \eqref{eq:1functionest}, we get
\[
\begin{aligned}
    E_{k+1}&\leq \Big(\int_{D_{k+1}}|v_{\eps, k+1}|^\tau \omega^a_\varepsilon dz\Big)^\frac{2}{\tau}\Big( \int_{D_{k+1}} \chi_{\{v_{\eps, k+1}>0\}}\omega^a_\varepsilon dz\Big)^{1-\frac{2}{\tau}}\\
        &\leq \Big(\int_{B^*_{\rho_k}}|\eta_k v_{\eps, k+1}|^\tau \omega^a_\varepsilon dz\Big)^\frac{2}{\tau}C_3^{k+1}E_k^{1-\frac{2}{\tau}}\\
        &\leq \tilde{C}^{k+1}E_k(E_k^{1-\frac{2}{\tau}}+E_k^{2-\frac{4}{\tau}-\frac{2}{p}}+E_k^{1-\frac{2}{\tau}-\frac{2}{q}})\,.
\end{aligned}
\]
Setting
\[
\gamma:=\min \Big\{1-\frac{2}{\tau}, 2-\frac{4}{\tau}-\frac{2}{p},1-\frac{2}{\tau}-\frac{2}{q}\Big\}>0
\]
and using that $E_k\leq E_0=\delta<1$, we find
\[
 E_{k+1}\leq \tilde{C}^{k+1}E_k^{1+\gamma}\,.
\]
Iterating this inequality gives
\[
E_{k}\leq (\tilde{C}^{\sum_{i=0}^k\frac{i}{(1+\gamma)^i}}\delta)^{(1+\gamma)^k}\,.
\]
As the series $S=\sum_{i=0}^\infty\frac{i}{(1+\gamma)^i}$ converges, we may choose $\delta\in (0,1)$ such that  $\tilde{C}^S\delta<1$,
which ensures that $E_k\rightarrow 0$ as $k\rightarrow \infty$. By the dominated convergence theorem, it follows that
\[
\lim_{k\rightarrow\infty}\int_{B^*_1}\omega_\varepsilon^a(u_\varepsilon-b_k)^2_+\chi_{D_k} dz=\int_{B^*_{1/2}}\omega_\varepsilon^a(u_\varepsilon-1)_+^2  dz=0
\]
which implies $u_\varepsilon\leq 1$ almost everywhere in $B^*_{1/2}$, proving the \emph{no-spike lemma}.

To conclude the proof, define
\[
v_\varepsilon:=\theta u_\varepsilon \quad \text{where }\quad \theta=\frac{\sqrt{\delta}}{\|u_\varepsilon\|_{L^{2, a, \eps}(B^*_1)}+ \|f_\varepsilon\|_{L^{p, a, \eps}(B^*_1)}+\|F_\varepsilon\|_{L^{q, a, \eps}(B^*_1)}}\,,
\]
with the same constant $\delta$ as in the no-spike lemma. Then $v_\eps$ satisfies the same equation as $u_\eps$ with data scaled by $\theta$, and hence meets the assumptions of the no-spike lemma. It follows that $|v_\varepsilon|\leq 1$ in $B^*_{1/2}$, and rescaling back gives
\[
\|u_\varepsilon\|_{L^\infty(B^*_{1/2})}\leq \frac{1}{\sqrt{\delta}}\left(\|u_\varepsilon\|_{L^{2, a, \eps}(B^*_1)}+ \|f_\varepsilon\|_{L^{p, a, \eps}(B^*_1)}+\|F_\varepsilon\|_{L^{q, a, \eps}(B^*_1)}\right).
\]
This completes the proof.
\end{proof}

\subsection{Stable H\"older estimates}

We now prove $\eps$-stable $C^{0,\alpha}$-regularity estimates for solutions to \eqref{eqeps} with $\eps_i >0$. Recall that the regularity of such solutions follows from Theorem \ref{thm1} with $a=0$ (see Section \ref{S:unif:ell:reg}); the aim of the following theorem is to show that the associated constant remains uniformly bounded as $\varepsilon \to 0$. For notational simplicity, and since this suffices for our purposes, we restrict to the case $\varepsilon_i = \varepsilon  > 0$ for all $i = 1, \ldots, n$.

\begin{teo}\label{thmC0stable}
Let $a\in (-1,\infty)^n$, $\eps \in (0,1]$ and define $\omega_\varepsilon^a:=\prod_{i=1}^{n}(\varepsilon^2+y_i^2)^{\frac{a_i}{2}}
$. Let $u_\varepsilon$ be a family of solutions to \eqref{eqeps}. Suppose that $p>\frac{d+\langle a^+\rangle}{2}$, $q>d+\langle a^+\rangle$ and $ \alpha\in (0,2-\frac{d+\langle a^+\rangle}{p}]\cap(0,1-\frac{d+\langle a^+\rangle}{q}]$.
Assume that $A$ satisfies \eqref{Ass:matrix}, that $A\in C^{0,\sigma(\cdot)}(B_1^*)$ for some modulus of continuity $\sigma$ and that $\|A\|_{C^{0,\sigma(\cdot)}(B_1^*)}\leq  L $. 
Then, there exists a constant $C$ which depends on $d,\alpha,p,q, L ,a,\lambda,\Lambda$ such that
\begin{equation}\label{eqC0alphaeps}
\|u_\varepsilon\|_{C^{0,\alpha}(B^*_{1/2})}\leq C\left(\|u_\varepsilon\|_{L^{2,a,\eps}(B^*_1)}+ \|f_\varepsilon\|_{L^{p,a,\eps}(B^*_1)}+\|F_\varepsilon\|_{L^{q,a,\eps}(B^*_1)}\right)\,.
\end{equation}
\end{teo}
\begin{proof}
  By Theorem \ref{L:schuader:unifell}, applied with $\omega^a_\eps A$ in place of $A$, inequality \eqref{eqC0alphaeps} holds with a constant that may, a priori, blow up as $\eps\to 0$. To show that this does not occur, we argue by contradiction. Suppose that there exists a sequence $\varepsilon_k \to 0$ such that 
\begin{equation}\label{eq:C0:contr}
    \|u_{\varepsilon_k}\|_{C^{0,\alpha}(B^*_{1/2})}\geq k\left(\|u_{\varepsilon_k}\|_{L^{2, a, \eps_k}(B^*_1)}+ \|f_{\varepsilon_k}\|_{L^{p, a, \eps_k}(B^*_1)}+\|F_{\varepsilon_k}\|_{L^{q, a, \eps_k}(B^*_1)}\right)=:k I_k 
\end{equation}
For simplicity, set $u_k:=u_{\varepsilon_k}$, $f_k:= f_{\eps_k}$, $F_k:= F_{\eps_k}$ and $\omega^a_k:=\omega^a_{\varepsilon_k}$. By Lemma \ref{thmdegiorgi}, we have 
\[
\|u_k\|_{L^\infty(B^*_{{3/4}})}\leq cI_k\,,
\]
and therefore, for $k$ sufficiently large,
\[
[u_k]_{C^{0,\alpha}(B^*_{1/2})}\geq \frac{k}{2}I_k\,.
\]
Next, let us consider consider a cutoff function $\eta\in C^\infty_c(B_{3/4})$ such that $\eta\equiv 1$ in $B_{1/2}$. Then, $\eta u_k\in C^{0,\alpha}(B^*_1)\cap \hat H^{1, a, \eps_k}(B^*_1)$ and it holds
\begin{equation}\label{eq:IkMk}
    M_k:=[\eta u_k]_{C^{0,\alpha}(B^*_1)}\geq[\eta u_k]_{C^{0,\alpha}(B^*_{1/2})}= [u_k]_{C^{0,\alpha}(B^*_{1/2})}\geq \frac{k}{2}I_k.
\end{equation}
By definition, there exist points $z_k,\zeta_k\in B^*_1$, $z_k\neq \zeta_k$, such that 
\[
    \frac{|\eta u_k(z_k)-\eta u_k(\zeta_k)|}{|z_k-\zeta_k|^\alpha}\geq \frac{M_k}{2}\,.
\]
Without loss of generality, we may further assume that $z_k, \zeta_k \in {B^*_{3/4}}$.

Set $r_k:=|z_k-\zeta_k|$. It holds
\[
    \|u_k\|_{L^\infty(B^*_{3/4})}\leq cI_k\leq c \frac{M_k}{k}\leq c \frac{|\eta u_k(z_k)-\eta u_k(\zeta_k)|}{k|z_k-\zeta_k|^\alpha}\leq  c \frac{\|u_k\|_{L^\infty(B^*_{3/4})}}{k r_k^\alpha}\,.
\]
Therefore,
\[
r_k\leq \frac{c}{k^{1/\alpha}}
\to 0, \quad \text{as }k\to\infty\,.
\]
Write $z_k=(x_k,y_k)\in \R^{d-n}\times\R^{n}$, and define 
\[
H:=\Big\{i\in \{1,\dots,n\} \ \mid \ \frac{(y_{k})_i}{r_k} \,\,\text{is bounded as } k \to \infty\Big\}\,.
\]
Set 
\begin{equation}\label{eq:hatzk}
  \hat{z}_k:=(x_k,  \hat{y}_k)\in \R^{d-n}\times \R^{n}\,, \quad \text{where}\quad  (\hat{y}_{k})_i:=
    \begin{cases}
        (y_{k})_i \quad &\text{if }\quad i \not\in H\,,\\
        0 \quad &\text{if }\quad  i \in H\,.
    \end{cases} 
\end{equation}
We then define the rescaled domains $\Omega_k$ and the limit blow-up domain $\Omega_\infty$ as
\[
\Omega_k:=\frac{B^*_1-\hat z_k}{r_k}\,, \qquad \Omega_\infty = \lim_{k\rightarrow \infty} \Omega_k\,.
\]
Since $r_k\rightarrow 0$, it follows that 
\begin{equation*}
    \Omega_\infty=
        \bigcap_{i\in H}\{y_i>0\}\,.
\end{equation*}
Finally, on each domain $\Omega_k$ we define two sequences of rescaled functions 
\[
    v_k(z) :=\frac{\eta u_k(\hat z_k+r_k z)-\eta u_k(\hat z_k)}{M_k r_k^\alpha}\,,\qquad w_k(z):=\eta(\hat z_k)\frac{u_k(\hat z_k+r_kz)-u_k(\hat z_k)}{M_kr_k^\alpha}\,.
\]

Let $z, \zeta \in \Omega_k$. By the definition of $M_k$, we have
\[
    |v_k(z)-v_k(\zeta)|=\frac{|\eta u_k(\hat z_k+r_k z)-\eta u_k(\hat z_k+r_k \zeta)|}{M_k r_k^\alpha}\leq |z-\zeta|^\alpha\,.
\]
As an immediate consequence, 
\[
|v_k(z)|=|v_k(z)-v_k(0)|\leq |z|^\alpha\,.
\]
Therefore, for every compact set $K\subset \R^d$ it holds $\|v_k\|_{C^{0,\alpha}(K\cap \Omega_\infty)}\leq C(K)$. By the Arzel\'a-Ascoli Theorem, it follows that $v_k$ converges uniformly on compact sets to a function $w\in C^{0,\alpha}_{\loc}(\Omega_\infty)$ satisfying the sublinear growth condition
\begin{equation}\label{eq:wgrowth}
    |w(z)|\leq |z|^\alpha\qquad \text{ for all } z\in \Omega_\infty.
\end{equation}
Moreover, 
\begin{equation}\label{eqvkn0}
    \Big|v_k\left(\frac{z_k-\hat z_k}{r_k}\right)-v_k\left(\frac{\zeta_k-\hat z_k}{r_k}\right)\Big|=\frac{|\eta u_k(z_k)-\eta u_k( \zeta_k)|}{M_k r_k^\alpha}\geq \frac{1}{2}.
\end{equation}
By the definition of $\hat z_k$, the sequences $\frac{z_k-\hat z_k}{r_k}$ and $\frac{\zeta_k-\hat z_k}{r_k}$ are bounded, and hence converge (up to a subsequence) to some points $\xi_1, \xi_2$ with $\xi_1\neq \xi_2$. Passing to the limit in \eqref{eqvkn0}, we deduce
\begin{equation*}
|w(\xi_1)-w(\xi_2)|\geq \frac{1}{2}\,,
\end{equation*}
that is, $w$ is not-constant. 

 In what follows, we show that $w$ must in fact be constant by invoking a Liouville-type theorem, thereby reaching a contradiction. First, for every $z\in K\subset \Omega_\infty$ we have
\begin{equation*}
\begin{aligned}
|v_k(z)-w_k(z)|&=\frac{|u_k(\hat z_k+r_kz)||\eta(\hat z_k+r_k z)-\eta(\hat z_k)|}{M_k r_k^\alpha}\\
&\leq c\frac{\|u_k\|_{L^\infty(B_{7/8})}}{M_k}r_k^{1-\alpha}\leq c\frac{ I_k }{M_k} r_k^{1-\alpha}\leq c\frac{r_k^{1-\alpha}}{k}\rightarrow 0\,.
\end{aligned}
\end{equation*}
Therefore, $v_k$ and $w_k$ converge to the same limit function $w$. As we shall see, the limit function $w$ satisfies a new (possibly weighted) boundary value problem. To state this problem precisely, we first introduce the relevant notation and the limiting weight. Define  
\[
\theta_{i,k}:=|(\varepsilon_k,r_k,(\hat y_k)_i)|\,,
\]
and, for $\beta \in \R^n$, set
\[
\Theta_k^\beta := \prod_{i=1}^{n}\theta_{i,k}^{\beta_i}\,.
\]
Moreover, let us call 
\[
\overline \eps = \lim_{k \to \infty}\frac{\eps_k}{|(\eps_k, r_k)|}\,, \qquad \overline r = \lim_{k \to \infty}\frac{r_k}{|(\eps_k, r_k)|}\,.
\]
With this notation, we define the rescaled weight
\[
\overline{\omega}_k^a(z):=\Theta_k^{-a}\omega_k^a(\hat z_k+r_k z)\,,
\]
and the limit weight 
\[
\overline{\omega}^a(z):=\prod_{i \in H}(\overline{\varepsilon}^2+\overline{r}^2y_i^2)^\frac{a_i}{2}\,,
\]
noticing that $\overline{\omega}_k^a\to \overline{\omega}^a$ a.e. in $\Omega_\infty$  and that $\overline \omega^a \in L^1_\loc(\Omega_\infty)$.

Fix $R>0$ and $\phi\in C^\infty_c(B_R)$. A standard computation shows that, for $k$ sufficiently large, 
\begin{equation}\label{eq:wkeq}
\int_{\Omega_k}\overline{\omega}_k^a (z)A(\hat z_k+r_k z)\nabla w_k(z)\cdot \nabla \phi(z)  \,dz=  J^1_k + J^2_k\,,
\end{equation}
where
\[
\begin{aligned}
&J^1_k = \frac{r_k^{2-\alpha}\eta (\hat z_k)}{M_k}\int_{\Omega_k}\overline \omega^a_k(z)f_k(\hat z_k+r_kz)\,\phi(z)\, \, dz\,,\\
&J^2_k = \frac{r_k^{1-\alpha}\eta(\hat z_k)}{M_k}\int_{\Omega_k}\overline \omega_k^a(z)F_k(\hat z_k+r_k z)\cdot \nabla \phi(z)\, dz\,.
\end{aligned}
\]
Our next goal is to show that $J^1_k,J^2_k\rightarrow 0$ as $k \to \infty$.

Using that $\overline{\omega}_k^a \in L^1_\loc(\R^d)$, by H\"older's inequality we have
\[
\begin{aligned}
\Big|\int_{\Omega_k}\overline \omega^a_k(z)f_k(\hat z_k+r_kz)\phi(z)\,dz\Big|\leq
&\, c\, \|\phi\|_{L^\infty(\Omega_k)}\Big(\int_{\Omega_k}\overline \omega^a_k(z)|f_k(\hat z_k+r_kz)|^p  dz\Big)^{\frac{1}{p}}\\
=&\, c\, \|\phi\|_{L^\infty(\Omega_k)}r_k^{-\frac{d}{p}}\Theta_k^{-\frac{a}{p}}\Big(\int_{B^*_1}\omega^a_k|f_k|^p \,dz\Big)^{\frac{1}{p}}\leq  \, c\, \|\phi\|_{L^\infty(\Omega_k)}r_k^{-\frac{d}{p}} \Theta_k^{-\frac{a}{p}}I_k\,,
\end{aligned}
\]
where $I_k$ is as in \eqref{eq:C0:contr}. As a consequence, using also \eqref{eq:IkMk} we get
\[
|J^1_k|\leq \, c\,\|\phi\|_{L^\infty(\Omega_k)}\frac{I_k}{M_k} r_k^{2-\alpha-\frac{d}{p}} \Theta_k^{-\frac{a}{p}}\leq \frac{c}{k} \|\phi\|_{L^\infty(\Omega_k)}r_k^{2-\alpha-\frac{d+\langle a^+ \rangle}{p}}\prod_{i=1}^{n}\Bigg(\frac{r_k^{a_i^+}}{\theta_{i,k}^{a_i}}\Bigg)^\frac{1}{p}\,.
\]
Since $0<r_k\leq \theta_{i,k}\leq c$ for some $c>0$, and $\alpha\leq 2-\frac{d+\langle a^+ \rangle}{p}$, it follows that
\[
|J^1_k|\leq \delta_k \|\phi\|_{L^\infty(\Omega_k)}\,,
\]
for some sequence $\delta_k\to 0$ as $k\to \infty$.

We now turn to the term $J^2_k$. By similar computations as above we get
\begin{align*}
\left|\int_{\Omega_k}\overline \omega_k^a(z)F_k(\hat z_k+r_k z)\cdot \nabla \phi(z)\,  dz\right|&\leq \|\nabla \phi \|_{L^2(\Omega_k, \overline\omega_k^a(z) dz) }\left( \int_{\Omega_k \cap B_R}\overline\omega_k^a(z)|F_k(\hat z_k+r_k z)|^2 \, dz \right)^\frac{1}{2}\\
& \leq \, c\, \|\nabla \phi \|_{L^2(\Omega_k, \overline\omega_k^a(z)dz) }\left( \int_{\Omega_k}\overline\omega_k^a(z)|F_k(\hat z_k+r_k z)|^q \,dz \right)^\frac{1}{q}\\
&= \, c\, \|\nabla \phi \|_{L^2(\Omega_k, \overline\omega_k^a(z) dz) } r_k^{-\frac{d}{q}}\Theta_k^{-\frac{a}{q}}\left(\int_{B^1_*}\omega_k|F_k|^q \,dz \right)^\frac{1}{q}\\
&\leq \, c\, \|\nabla \phi \|_{L^2(\Omega_k, \overline\omega_k^a(z) dz) } r_k^{-\frac{d}{q}}\Theta_k^{-\frac{a}{q}}I_k\,.
\end{align*}
Therefore,
\[
|J^2_k|\leq \frac{c}{k} \|\nabla \phi\|_{L^2(\Omega_k,\overline{\omega}^a_k(z)dz)}r_k^{1-\alpha-\frac{d+\langle a^+ \rangle}{q}}\prod_{i=1}^{n}\left(\frac{r_k^{a_i^+}}{\theta_{i,k}^{a_i}}\right)^\frac{1}{q}\,,
\]
and hence, also in this case,
\[
|J^2_k|\leq \delta'_k \|\nabla \phi\|_{L^2(\Omega_k,\overline{\omega}^a_k(z)dz)}\,,
\]
for some sequence $\delta'_k\to0$ as $k\to\infty$. It is worth emphasizing that, by the preceding discussion, one has that, for every $\phi \in C^\infty_c(B_R)$, it holds
\begin{equation}\label{eq3.31}
\Big|\int_{\Omega_k}\overline{\omega}_k^a (z) A(\hat z_k+r_k z)\nabla w_k(z)\cdot \nabla \phi(z) \, dz\Big|\leq o(1)(\|\phi\|_{L^\infty(\Omega_k)}+\|\nabla\phi\|_{L^2(\Omega_k,\overline{\omega}^a_k(z)dz)})\,.
\end{equation}
Next, we aim to show that 
\begin{equation}\label{eq3.38}
\int_{\Omega_k}\overline{\omega}_k^a (z)A(\hat z_k+r_k z)\nabla w_k(z)\cdot \nabla \phi(z) \, dz\rightarrow \int_{\Omega_\infty}\overline \omega^a (z){A}_\infty\nabla w(z) \cdot \nabla \phi (z) \, dz\,,
\end{equation}
where ${A}_\infty := \lim_{k \to \infty}A(\hat z_k + r_k z )$. Since $A$ admits a modulus of continuity, the Arzel\`a-Ascoli theorem guarantees that the convergence $A(\hat z_k+r_k z) \to A_\infty$ is locally uniformly, and that ${A}_\infty$ is a constant matrix.

To prove \eqref{eq3.38}, first we establish a local uniform bound for the weighted $L^2$-norm of $\nabla w_k$. Fix $0<r<R$, and let $\eta\in C^\infty_c(B_R)$ be a cutoff function such that $\eta\equiv1$ on $B_r$.  Taking $\phi=\eta^2 w_k$ in \eqref{eq3.31} and performing standard computations, we obtain 
\begin{equation}\label{eq3.55}
\|\nabla w_k\|_{L^2(B_r\cap \Omega_k,\overline{\omega}^a_k(z)dz)}\leq C_{r,a,p,q},
\end{equation}
where we have also used that $\|w_k\|_{{L^\infty(B_R \cap \Omega_k)}} \leq c$, since $w_k \to w$, and $w$ is locally bounded. In addition, by standard Rellich-Kondrakov we deduce that $w_k$ weakly converges to $w$ in $H^1_{\loc}(\Omega_\infty)$ and strongly in $L^2_{\loc}(\Omega_\infty)$.

Fix a compact set $K\subset \Omega_\infty$, and take $\eta \in C^{\infty}_c(K)$. Recall that $\overline \omega^a\in L^1_\loc(\Omega_\infty)$. Then we can test \eqref{eq:wkeq} with $\eta^2(w_k-w) \in H^1(K, \overline \omega^a_k (z)dz)$ and, after standard computations, infer that $\nabla w_k \to \nabla w$ strongly in $L^2_{\loc}(\Omega_\infty)$. In particular, $\nabla w_k\rightarrow \nabla w$ a.e. in $\Omega_\infty$. Therefore, by \eqref{eq3.55} and Fatou's lemma, we obtain 
\[
\|\nabla w\|_{L^2(B_r\cap \Omega_\infty,\overline{\omega}^a(z)dz)}< \infty\,.
\]
Since $w \in H^1_{\loc}(\Omega_\infty) \subset W^{1,1}_{\loc}(\Omega_\infty)$, Lemma \ref{lemH=W} and Lemma \ref{thmH=W} imply that $w \in H^1_\loc(\Omega_\infty, \overline \omega^a(z) dz)$. Finally, recalling that $\overline{\omega}^a_k$ is uniformly integrable in $L^1(K)$ for every compact set $K \subset \overline\Omega_\infty$, by Vitali's convergence theorem we get \eqref{eq3.38}.

Putting together \eqref{eq3.31} and \eqref{eq3.38} we conclude that for all $R>0$ and all $\phi\in C^\infty_c(B_R)$ it holds
\[
\int_{\Omega_\infty} \overline \omega^a A_\infty\nabla w \cdot \nabla \phi  \,dz=0\,,
\]
that is, $w$ is an entire solution to 
\begin{equation*}
\begin{cases}
-\dive(\overline \omega^a A_\infty\D w) = 0, & \text{ in }\Omega_\infty \\
\overline\omega^a A_\infty\D w\cdot e_{y_i} = 0,  & \text{ on } \Sigma_i\,, \ \text{ for } \ i\in H\,.
\end{cases}
\end{equation*}
Let $n_* = \#H$. If $n_*=0$, it follows that $\Omega_\infty=\R^d$ and $\overline\omega^a = cost$. Then, by \eqref{eq:wgrowth} and the classical Liouville theorem in $\R^d$, we deduce that $w$ must be constant. If $n_* > 0$, note that $A_\infty = A(\hat z_\infty)$, where $\hat z_\infty = \lim_{k \to \infty} \hat z_k$. By definition of $\hat z_k$, it follows that $\hat z_\infty \in \cap_{i \in H}\Sigma_i$, that is, $\hat z_\infty$ lies at the intersection of $n_*$ hyperplanes.  Hence, by Assumption \eqref{Ass:matrix}, and after reordering the variables if necessary, we may write $A_\infty$ in the block form 
\[
A_\infty = 
\begin{pmatrix}
    P_\infty & Q_\infty\\
    Q_\infty^\top & S_\infty\,,
\end{pmatrix}
\]
where $P_\infty$ is a $(d-n_*) \times (d-n_*)$ matrix, $Q$ is $n_* \times (d-n_*)$, and $S$ is an $n_*\times n_*$ diagonal matrix. That is, we obtain that $w$ is an entire solution to a problem of the type \eqref{eq:entire:liouville}. Hence, by Theorem \ref{T:Liouville}, $w$ must again be constant. Since we have already observed that $w$ is non-constant, this gives a contradiction and completes the proof.
\end{proof}

\subsection{Stable Schauder estimates}

We now establish the $\varepsilon$-stable $C^{1,\alpha}$ estimates. Maintaining the same notation as before, we have the following result.
\begin{teo}\label{thmc1stab}
Let $a\in (-1,\infty)^n$, $\eps \in (0,1]$ and define $\omega_\varepsilon^a:=\prod_{i=1}^{n}(\varepsilon^2+y_i^2)^{\frac{a_i}{2}}
$. Let $u_\varepsilon$ be a family of solutions to \eqref{eqeps}. Suppose that $p>d+\langle a^+\rangle$ and $ \alpha\in (0,1-\frac{d+\langle a^+\rangle}{p}]$.
Assume that $A$ satisfies \eqref{Ass:matrix}, that $A\in C^{0,\alpha}(B_1^*)$ and that $\|A\|_{C^{0,\alpha}(B_1^*)}\leq  L $. 
Then, there exists a constant $C$ which depends on $d,\alpha,p, L ,a,\lambda,\Lambda$ such that
\begin{equation}\label{eqC1alphaeps}
\|u_\varepsilon\|_{C^{1,\alpha}(B_{1/2}^*)}\leq C\left(\|u_\varepsilon\|_{L^{2,a, \eps}(B_1^*)}+ \|f_\varepsilon\|_{L^{p,a, \eps}(B_1^*)}+\|F_\varepsilon\|_{C^{0,\alpha}(B_1^*)}\right)\,.
\end{equation}
\end{teo}

\begin{proof}

By Theorem \ref{L:schuader:unifell}, applied with $\omega^a_\eps A$ in place of $A$, inequality \eqref{eqC1alphaeps} holds with a constant that may, a priori, blow-up as $\eps\to 0$. To show that this does not occur, we argue by contradiction. Suppose that there exists a sequence $\varepsilon_k \to 0$ such that 
\begin{equation*}
    \|u_{\varepsilon_k}\|_{C^{1,\alpha}(B^*_{1/2})}\geq k\left(\|u_{\varepsilon_k}\|_{L^{2,a, \eps_k}(B^*_1)}+ \|f_{\varepsilon_k}\|_{L^{p, a, \eps_k}(B^*_1)}+\|F_{\varepsilon_k}\|_{C^{0,\alpha}(B^*_1)}\right)=:k I_k \,.
\end{equation*}
For simplicity, set $u_k:=u_{\varepsilon_k}$, $f_k:= f_{\eps_k}$, $F_k:= F_{\eps_k}$ and $\omega^a_k:=\omega^a_{\varepsilon_k}$. By Theorem \ref{thmC0stable} we have 
\[
\|u_k\|_{C^{0,\alpha}(B^*_{1/2})}\leq c I_k \,.
\]
Hence, for $k$ sufficiently large,
\[
\|\nabla u_k\|_{L^\infty(B^*_{1/2})}+[\nabla u_k]_{C^{0,\alpha}(B^*_{1/2})}\geq \frac{k}{2}I_k\,.
\]
Moreover, by definition there exists $\xi_k\in B^*_{1/2}$ such that
\[
|\nabla u_k(\xi_k)|\geq \frac{\|\nabla u_k\|_{L^\infty(B^*_{1/2})}}{2},
\]
and 
\begin{align*}
    |\nabla u_k(\xi_k)|^2&\leq  \frac{1}{\int_{B^*_{1/2}}\omega_k^a(z)\,dz}\int_{B^*_{1/2}}\omega^a_k(z)|\nabla u_k(\xi_k)|^2 dz \\
    &\leq c\Big(\int_{B^*_{1/2}} \omega_k^a(z)|\nabla u_k(\xi_k)-\nabla u_k(z)|^2  dz +\int_{B^*_{1/2}} \omega_k^a|\nabla u_k|^2 dz\Big)\leq c( [\nabla u_k]^2_{C^{0,\alpha}(B^*_{1/2})}+I^2_k)\,,
\end{align*}
where in the last inequality we used the Caccioppoli inequality \eqref{eqwecacc}. Collecting the previous estimates, we conclude that for $k$ large enough,
\[
[\nabla u_k]_{C^{0,\alpha}(B^*_{1/2})}\geq c k I_k.
\]
Let us consider a cutoff function $\eta\in C^\infty_c(B_{3/4})$ such that $\eta\equiv 1$ in $B_{1/2}$. Then, $\eta u_k\in C^{0,\alpha}(B^*_1)\cap \hat H^1(B^*_1, \omega^a_k(z)dz)$ and 
\begin{equation*}
    M_k:=[\nabla (\eta u_k)]_{C^{0,\alpha}(B^*_1)} \geq [\nabla (\eta u_k)]_{C^{0,\alpha}(B^*_{1/2})}= [\nabla u_k]_{C^{0,\alpha}(B^*_{1/2})}\geq ckI_k.
\end{equation*}
By definition, there exist points $z_k,\zeta_k\in B^*_1$ with $z_k\neq\zeta_k$ such that 
\begin{equation*}
    \frac{|\nabla (\eta u_k)(z_k)-\nabla (\eta u_k)(\zeta_k)|}{|z_k-\zeta_k|^\alpha}\geq \frac{M_k}{2}\,,
\end{equation*}
and without loss of generality we may assume that $\zeta_k, z_k\in B^*_{3/4}$. 

Set $r_k :=|z_k- \zeta_k|$, $\hat z_k$ as in \eqref{eq:hatzk}, $\Omega_k:=\frac{B^*_1-\hat{z}_k}{r_k}$ and $\Omega_\infty=\lim_{k\to\infty}\Omega_k$. For $z\in\Omega_k$ let us define the functions
\begin{align*}
    v_k(z):=\frac{\eta(\hat{z}_k+r_k z)[u_k(\hat{z}_k+r_k z)-u_k(\hat{z}_k)]-r_k\eta(\hat{z}_k)\nabla u_k(\hat{z}_k)\cdot z}{M_k r_k^{1+\alpha}}\,,\\
    w_k(z):=\frac{\eta(\hat{z}_k)[u_k(\hat{z}_k+r_k z)-u_k(\hat{z}_k)]-r_k\eta(\hat{z}_k)\nabla u_k(\hat{z}_k)\cdot z}{M_k r_k^{1+\alpha}}\,.
\end{align*} 
Note that, in contrast with the proof of Theorem \ref{thmC0stable}, we cannot ensure at this stage that $r_k \to 0$.

Let us observe that, for every $z, \zeta \in \Omega_k$,
\[
   \begin{aligned}
    |\nabla v_k(z)-\nabla v_k(\zeta)|\leq&\frac{1}{M_k r_k^\alpha}|\nabla (\eta u_k)(\hat z_k+r_k z)-\nabla (\eta u_k)(\hat{ z}_k+r_k\zeta)|\\
    +&\frac{1}{M_k r_k^\alpha}|u_k(\hat{z}_k)||\nabla \eta (\hat z_k+r_k z)-\nabla\eta (\hat z_k+r_k \zeta)|\\
    \leq &|z-\zeta|^\alpha+\frac{I_k}{M_k}c |z-\zeta|^\alpha\leq |z-\zeta|^\alpha(1+\frac{c}{k})\leq 2|z-\zeta|^\alpha\,.
    \end{aligned}
\]
Since $v_k(0)=0$ and $\nabla v_k (0)=0$, it follows that 
\begin{equation}
\label{eq:vk:1alpha:estimate}
\begin{aligned}
&|\nabla v_k(z)|=|\nabla v_k(z) - \nabla v_k(0)| \leq 2|z|^\alpha\,,\\ 
    &|v_k(z)|=\Big|\int_0^1\nabla v_k(tz)\cdot z \,dt \Big|\leq \int_0^1 |\nabla v_k (tz)-\nabla v_k(0)||z|dt\leq \frac{2}{1+\alpha}|z|^{1+\alpha}\,.
\end{aligned}
\end{equation}
As a consequence, for every compact set $K\subset \R^d$, we have $\|v_k\|_{C^{1,\alpha}(K\cap \Omega_\infty)}\leq c(K)$. Hence, by the Arzel\'a-Ascoli theorem, $
v_k\rightarrow v$ in $C_{\loc}^{1}(\Omega_\infty)$. Moreover, passing to the limit in \eqref{eq:vk:1alpha:estimate}, we obtain the subquadratic growth condition
\begin{equation}\label{eq:subquadgrowth}
|v(z)| \leq \frac{2}{1+\alpha}|z|^{1+\alpha}\,.
\end{equation}
Next, we prove that $v$ has non constant gradient in $\Omega_\infty$. Let us observe that
\begin{align*}
\Big| \nabla v_k\left(\frac{z_k - \hat z_k}{r_k}\right)-\nabla v_k\left(\frac{\zeta_k -\hat z_k}{r_k}\right)\Big|=&\big| \nabla (\eta u_k)(z_k)-\nabla (\eta u_k)(\zeta_k)-u_k(\hat z_k)(\nabla \eta (z_k)-\nabla \eta (\zeta_k))\big|\frac{1}{M_k r_k ^\alpha}\\
\geq & \big| |\nabla (\eta u_k)(z_k)-\nabla (\eta u_k)(\zeta_k)|-|u_k(\hat z_k)(\nabla \eta (z_k)-\nabla \eta (\zeta_k))|\big|\frac{1}{M_k r_k ^\alpha}\\
\geq &\frac{1}{2}-r_k^{1-\alpha}\frac{1}{k}\geq \frac{1}{4}\,.
\end{align*}
Passing to the limit in the previous estimate we get that there exists two points $\xi_1\not=\xi_2$ such that $|\D v(\xi_1)-\D v(\xi_2)|\ge 1/4$, hence, $\D v$ is not constant. Summing up, the limit function $v$ has non-constant gradient and satisfies a sub-quadratic growth condition. We will apply a suitable Liouville-type theorem to obtain a contradiction.

We now show that $r_k\rightarrow 0$. Define 
\[
\ell_k := \frac{\eta(\hat z_k)\nabla u_k(\hat z_k)}{M_k r_k^{\alpha}}\,.
\]
Since $\eta\equiv 0$ in $B_1\setminus B_{3/4}$, 
\[
\|\nabla (\eta u_k)\|_{L^\infty(B^*_1)}\leq M_k\,.
\]
Consequently
\begin{align}\label{eq3.81}
    |\eta (\hat z_k)\nabla u_k(\hat z_k)|=|\nabla (\eta u_k)(\hat z_k)-u_k(\hat z_k)\nabla \eta (\hat z_k)|\leq M_k+c\|u_k\|_{L^\infty(B^*_{3/4})}\leq M_k+\frac{c}{k}M_k\leq  2 M_k\,.
\end{align}
Hence, $\ell_k \leq cr_k^{-\alpha}$. 

Assume now by contradiction that $r_k\rightarrow \bar r >0$. Then, up to subsequences, $\ell_k\to \ell$ for some $\ell \in \R^d$. Let $\hat z = \lim_{k\to \infty}\hat z_k$; then $\Omega_\infty=\frac{B^*_1-\hat z}{\bar r}$. Note that, since $r_k \not \to 0$, it follows that $\hat z_k = (x_k, 0)$ (see \eqref{eq:hatzk}). Therefore, for every $b \in (0,1)$ and all sufficiently large $k$,
\[
\frac{B^*_{b}-\hat z}{\bar r}\subset \Omega_k\cap \Omega_\infty\,.
\]
If $z\in \frac{B^*_{b}-\hat z}{\bar r}$, then
\[
|\hat z_k +r_k z|\leq |\hat z_k-\hat z+(r_k-\bar r)z|+|\hat z+\bar rz|\leq o(1)+b\leq  \frac{1+b}{2}\,.
\]
Thus, for all such $z$,
\[
\begin{aligned}
\left| v_k(z)+\ell_k \cdot z \right|&=\frac{|\eta(\hat z_k+r_k z)||u_k(\hat z_k+r_k z)-u_k(\hat z_k)|}{M_k r_k^{1+\alpha}}\leq c\frac{\|u_k\|_{L^\infty(B^*_{\max\{(1+b)/2, 3/4\}})}}{I_kk}\leq c\frac{1}{k}\rightarrow 0\,,
\end{aligned}
\]
where the last inequality follows from Lemma \ref{thmdegiorgi}. Thus, we infer that $v(z)=\ell\cdot z$ in $\frac{B^*_{b}-\hat z}{\bar r}$. Since $b \in (0,1)$ is arbitrary, $v$ is an affine function in $\Omega_\infty$, contradicting the fact that $v$ has non-constant gradient in $\Omega_\infty$. Therefore $r_k\rightarrow 0$ and, using the same notations as in Theorem \ref{thmC0stable}, we have  
\begin{equation*}
    \Omega_\infty=
        \bigcap_{i\in H}\{y_i>0\} \,.
\end{equation*}

Now fix a compact set $K\subset \Omega_\infty$. For $z\in K$ we have
\begin{align*}
|v_k(z)-w_k(z)|=\frac{|\eta(\hat z_k+r_k z)-\eta (\hat z_k)||u_k(\hat z_k+r_k z)-u_k(\hat{z}_k)|}{M_k r_k^{1+\alpha}}\leq c\frac{r_k^{\beta-\alpha}}{k}\rightarrow 0
\end{align*}
where we used the $C^{0,\beta}$ regularity of $u_k$ for every $\beta \in (0,1)$ and the smoothness of $\eta $. Therefore, $v_k$ and $w_k$ converge to the same limit function $w$. As we shall see, the limit function $w$ satisfies a new (possibly weighted) boundary value problem.

To establish this, we proceed as in the proof of Theorem \ref{thmC0stable}, to which we refer for the relevant notation. Fix $R>0$ and $\phi\in C^\infty_c(B_R)$. For $k$ sufficiently large we have  
\begin{align}
\begin{split}\label{eq3.86}
    \int_{\Omega_k}\overline\omega^a_k(z)A(\hat z_k+r_kz)\nabla w_k(z)\cdot\nabla \phi (z) dz= T_k^1-T_k^2-T_k^3-T_k^4\,,
\end{split}
\end{align}
where
\[
\begin{aligned}
&T_k^1:=\frac{r_k^{1-\alpha} \eta (\hat z _k)}{M_k}\int_{\Omega_k}\overline\omega^a_k(z)f_k(\hat{z}_k+r_k z)\phi(z) \,dz\,,\\
&T_k^2:=\frac{r_k^{-\alpha} \eta (\hat z _k)}{M_k} \int_{\Omega_k} \overline\omega^a_k(z)(F_k(\hat z_k+r_k z)-F_k(\hat z_k)) \cdot \nabla \phi(z)\,  dz\,,\\
& T_k^3:=\frac{r_k^{-\alpha} \eta (\hat z _k)}{M_k} \int_{\Omega_k} \overline\omega^a_k(z)(A(\hat z_k+r_k z)-A(\hat z_k))\nabla  u_k(\hat z_k)\cdot\nabla \phi (z) \, dz\,,\\
& T_k^4:=\frac{r_k^{-\alpha}\eta (\hat z _k)}{M_k} \int_{\Omega_k}\overline\omega^a_k(z)(A(\hat{z}_k)\nabla u_k(\hat z_k)+F_k(\hat z_k))\cdot \nabla \phi(z)  \,dz\,.
\end{aligned}
\]
Arguing as in the proof of Theorem \ref{thmC0stable} and using the assumption on $p$, we obtain
\[
|T_k^1|\leq \delta_k \|\phi\|_{L^\infty(\Omega_k)},\qquad \text{with }\delta_k\rightarrow 0.
\]
Next, since $F_k \in C^{0,\alpha}(B_1^*)$, we have
\[
\begin{aligned}
\Big| \int_{\Omega_k} \overline\omega^a_k(z)(F_k(\hat z_k+r_k z)- F_k(\hat z_k))\cdot\nabla \phi(z)\, dz\Big|\leq c I_k r_k^\alpha \|\nabla \phi\|_{L^2(\Omega_k,\overline \omega^a_k(z)dz)}\,.
\end{aligned}
\]
Recalling that $I_k \leq cM_k/k$, it follows that
\[
|T_k^2|\leq \delta'_k \|\nabla \phi\|_{L^2(\Omega_k,\overline \omega^a_k(z)dz)},\qquad \text{with }\delta'_k\rightarrow 0.
\]

We now estimate $T_k^4$. For each $i \in H$, we have $(\hat y_k)_i = 0$, that is, $\hat z_k \in \Sigma_i$. Thus, by the Neumann boundary condition,
\[
\eta(\hat z_k)(A\nabla u_k+F_k)(\hat z_k)\cdot e_{y_i} = 0 \quad \text{ for all }i \in H\,.
\]
Then, using that $\Omega_k \to \Omega_\infty$ and the characterization of $\Omega_\infty$, we get that, for $k$ sufficiently large, 
\[
\phi\,\overline\omega^a_k(z)\eta(\hat z_k)(A\nabla u_k+F_k)(\hat z_k) \cdot\nu = 0 \qquad \text{ on }\partial \Omega_k\,,
\]
where $\nu$ is the outward unit normal to $\partial \Omega_k$. Consequently, integration by parts gives
\begin{align}\label{eq4.92}
\begin{split}
    \int_{\Omega_k}\overline\omega^a_k(z)\eta(\hat z_k)&(A\nabla u_k+F_k)(\hat z_k) \cdot \nabla \phi(z) \, dz\\
    &=\sum_{i\not \in H}\int_{\Omega_k}     \phi\,\eta(\hat z_k)(A\nabla u_k+F_k)(\hat z_k)\cdot e_{y_i}\, \partial_{y_i}\overline\omega^a_k(z) \, dz\,.
\end{split}
\end{align}
Fix now $i\notin H$, recall that $(\hat y_k)_i = (y_k)_i$ in this case, and define $\xi_k^i = \hat z_k - (y_k)_i e_{y_i}$ the unique projection of $\hat z_k$ onto $\Sigma_i = \{y_i = 0\}$. Then by the Neumann boundary condition 
\[
\eta(\hat z_k)(A\nabla u_k+F_k)(\hat z_k)\cdot e_{y_i} = \eta(\hat z_k)(A\nabla u_k+F_k)(\hat z_k)\cdot e_{y_i} - \eta(\xi^i_k)(A\nabla u_k+F_k)(\xi^i_k)\cdot e_{y_i}.
\]
Thus, using that $A \in C^{0,\alpha}(B^*_1)$ and Theorem \ref{thmC0stable} in $B^*_{3/4}$ we get 
\begin{equation}\label{eq:6.66}
\begin{split}
|\eta(\hat z_k)(A\nabla u_k+F_k)(\hat z_k)\cdot e_{y_i}|&\\
\leq  |A\nabla (\eta u_k)(\hat z_k)-A\nabla&(\eta u_k)(\xi^i_k)|+|u_k A \nabla \eta(\hat{z}_k)-u_kA\nabla \eta(\xi^i_k)|+ |\eta F_k(\hat z_k)-\eta F_k(\xi^i_k)|\\
 \leq c(M_k(y_k)_i^\alpha& + \|u_k\|_{C^{0,\alpha}(B^*_{3/4})}(y_k)_i + I_k(y_k)_i^\alpha) \leq cM_k(y_k)_i^\alpha\,.
\end{split}
\end{equation}
Next, we notice that, for $i \not \in H$ and $z \in B_R$, 
\[
    |\partial_{y_i}\overline \omega^a_k| = \overline\omega^a_k\frac{|(y_k)_i + r_k y_i|}{\eps_k^2 + ((y_k)_i + r_k y_i)^2}r_k\leq c \overline \omega^a_k \frac{r_k}{(y_k)_i}\,.
\]
Thus
\begin{align}\label{eq4.97}
 \int_{\Omega_k}  |\phi| |\partial_{y_i}\overline \omega^a_k|\,dz\leq  \|\phi\|_{L^\infty(\Omega_k)}\int_{\Omega_k\cap B_R}  |\partial_{y_i}\overline \omega^a_k|\,dz\leq c \|\phi\|_{L^\infty(\Omega_k)} \frac{r_k}{(y_k)_i}.
\end{align}
By \eqref{eq4.92}, \eqref{eq:6.66} and \eqref{eq4.97} we find that
\[
    \Big|\int_{\Omega_k}\overline\omega^a_k(z)\eta(\hat z_k)(A\nabla u_k+F_k)(\hat z_k) \cdot \nabla \phi(z) \, dz \Big|\leq cM_k\|\phi\|_{L^\infty(\Omega_k)}\frac{r_k}{((y_k)_i)^{1-\alpha}}\,,
\]
and thus,
\[
    |T_k^4|\leq c\|\phi\|_{L^\infty(\Omega_k)}\sum_{i\notin H}\left(\frac{r_k}{(y_k)_i}\right)^{1-\alpha}=c\delta''_k\|\phi\|_{L^\infty(\Omega_k)}\,, \qquad \text{with }\delta''_k \to 0\,.
\]
It remains to estimate $T_k^3$. We claim that 
\begin{equation}\label{eq:c1claim}
    |T_k^3|\leq c\delta'''_k\|\nabla \phi\|_{L^2(\Omega_k,\overline \omega^a_k dz)}\,, \quad \text{ for some }\delta'''_k \to 0\,.
\end{equation}
We postpone the proof of this claim and first complete the argument. Combining all estimates in \eqref{eq3.86}, we obtain
\[
    \Big|\int_{\Omega_k}\overline\omega^a_k(z)A(\hat z_k+r_kz)\nabla w_k(z)\cdot \nabla \phi (z) dz \Big|\leq \tilde\delta_k(\|\phi\|_{L^\infty(\Omega_k)}+\|\nabla \phi\|_{L^2(\Omega_k,\overline \omega_k^a dz)}).
\]
for some $\tilde\delta_k \to 0$. Proceeding exactly as in the final part of the proof of Theorem \ref{thmC0stable} we infer that
\[
\int_{\Omega_\infty} \overline \omega^a A_\infty\nabla w \cdot \nabla \phi  \,dz=0\,, \qquad \text{ for all } R>0 \text{ and }\phi\in C^\infty_c(B_R)\,,
\]
that is, $w$ is an entire solution to 
\begin{equation*}
\begin{cases}
-\dive(\overline \omega^a A_\infty\D w) = 0, & \text{ in }\Omega_\infty \\
\overline\omega^a A_\infty\D w\cdot e_{y_i} = 0,  & \text{ on } \partial \Omega_\infty \cap \Sigma_i\,, \ \text{ for } \ i\in H\,.
\end{cases}
\end{equation*}
Here, $A_\infty$ is a uniformly elliptic constant matrix. Let now set $n_* = \#H$. 
If $n_*=0$, then $\Omega_\infty=\R^d$ and $\overline\omega^a = cost$. By \eqref{eq:subquadgrowth} and the classical Liouville theorem in $\R^d$, it follows that $w$ must be affine. If $n_* > 0$ the same conclusion follows by applying Theorem \ref{T:Liouville}, since $A_\infty$ satisfies its assumptions (see the proof of Theorem \ref{thmC0stable}). In both cases, $w$ is affine. However, we have already observed that $\nabla w$ is non-constant, which leads to a contradiction and thus completes the proof.

Finally, we prove the claim \eqref{eq:c1claim}. First, suppose that the matrix $A$ is slightly more regular, that is, assume $A \in C^{0,\alpha'}(B^*_1)$ for some $\alpha<\alpha'<1$. In this case, we easily obtain (see also \eqref{eq3.81})
\[
    |T_k^3|\leq c\|\nabla \phi\|_{L^2(\Omega_k,\overline \omega^a_kdz)}r_k^{\alpha'-\alpha}\rightarrow 0\,,
\]
which proves the claim under this stronger assumption. Consequently, Theorem \eqref{thmc1stab} holds when $A$ satisfies the suboptimal regularity condition $A \in C^{0,\alpha'}(B^*_1)$. Let us now return to the original assumption $A \in C^{0,\alpha}(B^*_1)$. We can apply the suboptimal regularity estimate just established to deduce that, for all $\beta \in (0,\alpha)$,
\[
\|u_k\|_{C^{1,\beta}(B_{3/4}^*)}\leq C\left(\|u_k\|_{L^{2,a, \eps}(B_1^*)}+ \|f_k\|_{L^{p,a, \eps}(B_1^*)}+\|F_\varepsilon\|_{C^{0,\alpha}(B_1^*)}\right)\,.
\]
As a consequence, 
\[
\|\nabla u_k\|_{L^\infty(B_{3/4}^*)} \leq c I_k \leq c \frac{M_k}{k}\,.
\]
Therefore,
\[
    |T_k^3|\leq c\|\nabla \phi\|_{L^2(\Omega_k,\overline \omega^a_k dz)}\frac{1}{k} \to 0\,,
\]
and the claim follows also in this case, completing the proof of the theorem.
\end{proof}

\subsection{H\"older and Schauder estimates when \texorpdfstring{$\varepsilon=0$}{e=0}}

Finally, by combining the results established in the previous sections, we obtain the main theorem of the work.

\begin{proof}[Proof of Theorem \ref{thm1}]
We prove statement (ii); the proof of the other statement is analogous. Throughout the proof, the constants involved depend only on the parameters specified in the statement. Let $u$ be a weak solution to \eqref{eq:degenerate:equation:B*} in $B_1^*$. By Lemma \ref{lemA.2*}, there exists a sequence of solutions $u_{\e_k}$ to 
\[
\begin{cases}
-\div( \omega_{\varepsilon_k}^a A\nabla u_{\varepsilon_k})=\omega^a_{\varepsilon_k} f_{k}+\div(\omega_{\varepsilon_k}^aF_{k} )\qquad &\text{in }B_{3/4}^*\\
\omega^a_{\varepsilon_k}(A\D u_{\varepsilon_k} +F_{k})\cdot e_{y_i} = 0,  & \text{on }\partial B^*_{3/4} \cap \Sigma_i, \text{ for } i=1,\dots,n,
\end{cases}
\]
where $\e_k\to0$, $\omega_{\varepsilon_k}^a:=\prod_{i=1}^{n}({\varepsilon_k}^2+y_i^2)^{\frac{a_i}{2}}
$ and
\[
\|u_k\|_{L^{2,a, \eps_k}(B_{3/4}^*)}  + \|f_k\|_{L^{p,a, \eps_k}(B_{3/4}^*)}+
\|F_k\|_{C^{0,\alpha}(B_{3/4}^*)} 
\le C 
\big(\|u\|_{L^{2,a}(B_1^*)}  +
\|f\|_{L^{p,a}(B_1^*)}+
\|F\|_{C^{0,\alpha}(B_1^*)}\big).
\]
Moreover, $u_k\to u$ in $H^{1}(B_{3/4}^*)$.

By combining Theorem \ref{thmc1stab} with the estimate above, we obtain
\[
\|u_{\e_k}\|_{C^{1,\alpha}(B_{1/2}^*)} \le C \big(\|u\|_{L^{2,a}(B_1^*)}  +
\|f\|_{L^{p,a}(B_1^*)}+
\|F\|_{C^{0,\alpha}(B_1^*)}\big). 
\]
By the Arzel\'a-Ascoli theorem, we have that $u_{\varepsilon_k} \to v$ in $C^1(B_{1/2}^*)$. Since  $u_{\varepsilon_k} \to u$ in $H^1(B_{3/4}^*)$, we have $v = u$, and \eqref{eq:1alpha} follows by passing to the limit in the above estimate.

Furthermore, the convergence $u_{\varepsilon_k} \to u$ in $C^1(B_{1/2}^*)$ implies that $u$ satisfies the boundary condition \eqref{eq:BC:thm1} and the proof is complete.
\end{proof}

\begin{remark}\label{R:drift}
In view of the higher-order regularity results of the next section, it is worth noting that the conclusions of Theorem \ref{thm1} remain valid even if the right-hand side of equation \eqref{eq:degenerate:equation:B*} contains an additional drift term of the form $\omega^a b \cdot \nabla u$, where $b$ is a vector filed satisfying the (suboptimal) assumption $b\in L^p(B_1, \omega^a dz)$ for some $p>d+\langle a^+\rangle$. The modifications required in the proofs are straightforward and follow standard arguments, so we omit them.
\end{remark}

\section{Higher-order regularity in some simple cases}\label{S:CKN}

\subsection{An application to CKN inequalities with monomial weights}

We now apply the previous results to establish higher-order regularity for a class of model equations associated with Caffarelli--Kohn--Nirenberg (CKN) inequalities. In particular, we 
consider the boundary value problem
\begin{equation}\label{eq7.2}
    \begin{cases}
-\dive(\omega^a h\D u) = \omega^a f & \text{in }B_1^* \\
\omega^a \partial_{y_i}  u  = 0  & \text{on }\partial B^*_1 \cap \Sigma_i, \text{ for } i=1,\dots,n,
\end{cases}
\end{equation}
where $h,f$ are scalar functions satisfying, for some $\alpha \in (0,1)$ and for every $i=1,\dots,n$
\begin{equation}\label{eq7.3}
\begin{cases}
h\in C^{2,\alpha}({B_1^*})\,, \\
h\ge c>0\,, \\
\partial_{y_i}h=0 \quad \text{on }\partial B^*_1 \cap \Sigma_i, \\ 
\end{cases}
\quad 
\begin{cases}
f\in C^{1,\alpha}({B_1^*})\,, \\
\partial_{y_i}f=0 \quad \text{on }\partial B^*_1 \cap \Sigma_i\,,\\ 
\displaystyle\frac{\partial_{y_i}f}{y_i} \in L^{p_i}(B_1^*,\omega^{a+2e_i}), \ \text{where } p_i=\frac{d+\langle (a+2e_i)^+\rangle}{1-\alpha}
\end{cases}
\end{equation}

\begin{remark}\label{R:supsing}
In the proof of the next theorem, we may encounter weights having one \emph{supersingular} component.
We provide here a brief overview of the results needed to handle such weights. 

Let us write $z = (x, y', \tau) \in \R^{d-n}\times\R^{n-1}\times \R$. Let $\hat a \in (-1, \infty)^{n-1} \times\{0\}$, $a_n \in \R$, and call $a = \hat a + a_ne_n$. Moreover, let $\eps \in [0,1]$. We define the norm
\[
\|u\|^2_{H^{1,a, \eps}(B_R)} = \int_{B_R}\omega^{ \hat a}(\eps^2 + \tau^2)^{a_n/2} ( u^2  + |\nabla u|^2) dz\,,
\]
and introduce the weighted Sobolev space 
\[
\tilde H^{1, a,\eps}(B_R) = \text{ the completion of }C^\infty(\overline{B_R}\setminus \Sigma_n) \text{ with respect to }\|\cdot\|_{H^{1,a, \eps}(B_R)}\,.
\] 
We recall some key properties of these spaces.

\noindent $-$ \emph{Inclusions and density:} 
Define
\[
H^{1, a, \eps}(B_R) = \text{ the completion of }\big\{u \in C^\infty(\overline{B_R}) \mid \|u\|_{H^{1,a, \eps}(B_R)}< \infty\big\} \text{ with respect to }\|\cdot\|_{H^{1,a, \eps}(B_R)}\,.
\] 
When either $a_n \in (-1, \infty)$ or $\eps >0$, this is exactly the Sobolev space we introduced in Section \ref{S:ff}, and one has
\[
\tilde H^{1, a, \eps}(B_R) \subsetneq H^{1, a, \eps}(B_R)\,.
\]
On the other hand, if $a_n \in (-\infty, -1]$ and $\eps = 0$, equality holds
\[
\tilde H^{1, a}(B_R) = H^{1, a}(B_R)\,.
\]
Moreover, if $u \in C^0(B_R) \cap C^1(B_R\setminus \Sigma_n)$ satisfies $\|u\|_{H^{1,a,\eps}(B_R)}< \infty$ and $u = 0$ in $\Sigma_n$, then $ u \in \tilde H^{1, a, \eps}(B_R)$. The proof follows by adapting the argument of \cite[Proposition 2.5]{CorFioVit25}.

\noindent $-$ \emph{Isometric transformation:} Let $a_n < 1$ and $\eps = 0$.  If $ u \in \tilde H^{1, a}(B_R)$, then $v:=|\tau|^{a_n/2}u \in \tilde H^{1, \hat a}(B_R)$. In particular, $\tilde H^{1, a}(B_R)$ is isometric to $\tilde H^{1, \hat a}(B_R)$ endowed with the (equivalent) norm
\[
\int_{B_R}\omega^{\hat a} (v^2 + |\nabla v|^2) dz +\frac{a_n(a_n-2)}{4} \int_{B_R}\omega^{\hat a}|\tau|^{-2}v^2  dz - \frac{a_n}{2}\int_{\partial B_1}\omega^{\hat a}v^2\, d \mathcal{H}^{n-1}\,. 
\]
To verify the equivalence of this norm, one can argue as in \cite[Appendix B]{SirTerVit21b}, using also the results from Section \ref{S:ff} (see in particular Remark \ref{Hardyrmrk}).

We say that $u$ is a weak solution to 
\[
\begin{cases}
- \dive(\omega^{\hat a}|\tau|^{a_n}A \nabla u)  = \omega^{\hat a}|\tau|^{a_n}f + \div(\omega^{\hat a}|\tau|^{a_n}F) & \text{ in }B_R\\
u = 0 & \text{ on }B_R \cap \Sigma_n\,,
\end{cases}
\]
if and only if $ u \in \tilde H^{1, a}(B_R)$ satisfies
\[
\int_{B_R}\omega^{\hat a}|\tau|^{a_n}\nabla u \cdot \nabla \phi\, dz = \int_{B_R} \omega^{\hat a}|\tau|^{a_n}f\phi \,dz- \int_{B_R} \omega^{\hat a}|\tau|^{a_n}F \cdot \nabla\phi \,dz\,, \, \text{ for every }\phi \in C^\infty_c(B_R \setminus \Sigma_n)\,. 
\]
Finally, we observe that Lemma \ref{lemA.1} remains valid for weights with one supersingular component. The proof can be adapted with minor changes: the only difference is that, in this setting, the $H = W$ property is not known to hold.
To overcome this issue, one may argue as in the proof of \cite[Lemma 2.12]{SirTerVit21a}, making use of the isometry described above.
\end{remark}

We are now in a position to prove the higher regularity result for solutions to \eqref{eq7.2}.

\begin{Theorem}\label{thm7.1}
Let $a\in (-1,\infty)^n$, $u$ be a weak solution to \eqref{eq7.2} and assume that \eqref{eq7.3} holds for some $\alpha\in(0,1)$. Then, $u\in C^{3,\alpha}(B_r^*)$ for every $r\in(0,1)$.
\end{Theorem}
\begin{proof}

We begin by observing that, in view of \eqref{eq7.3}, if $u \in H^{1,a}(B_1^*)$ is a weak solution to \eqref{eq7.2}, then
\[
\bar u (x, y) = u(x, |y_1|, \ldots, |y_n|)\,
\]
belongs to $H^{1,a}_{\rm e}(B_1)$ (see Lemma \ref{L:Hevenid}) and weakly satisfies
\begin{equation}\label{eq:balleq}
-\dive(\omega^a \bar h\D \bar u) = \omega^a \bar f \qquad \text{in }B_1\,
\end{equation}
where 
\begin{equation}\label{eq:hofh}
\begin{cases}
\bar h\in C^{2,\alpha}(B_1)\,, \\ 
\bar h\ge c>0\,, \\
\bar h (x, y) = h(x, |y_1|, \ldots, |y_n|)\,,\\ 
\end{cases}
\quad 
\begin{cases}
\bar f\in C^{1,\alpha}(B_1)\,, \\
\bar f (x, y) = f(x, |y_1|, \ldots, |y_n|)\,.\\ 
\displaystyle\frac{\partial_{y_i}\bar f}{y_i} \in L^{p_i}(B_1,\omega^{a+2e_i}), \ \text{where } p_i=\frac{d+\langle (a+2e_i)^+\rangle}{1-\alpha}
\end{cases}
\end{equation}
Conversely, if $ u \in H^{1,a}_{\rm e}(B_1)$ is a weak solution to \eqref{eq:balleq}, then its restriction $\bar u = u_{|B_1^*} \in H^{1,a}(B_1^*)$ is a solution to \eqref{eq7.2} with $h = \bar h_{|B_1^*}$ and $f = \bar f_{|B_1^*}$. 

Hence, it suffices to prove the following equivalent statement: every solution $u$ that is even in each variable $y_i$ to 
\[
-\dive(\omega^a h\D u) = \omega^a f \qquad \text{in }B_1
\]
with $h, f$ satisfying \eqref{eq:hofh}, is such that  $u\in C^{3,\alpha}(B_r)$ for every $r\in(0,1)$.

Let $k \in \N$. Throughout the proof, we set $d = k + n$, where $k$ denotes the number of unweighted $x$-variables. The proof is divided into several steps and argue by induction on $n$. When $n = 0$, the conclusion follows directly for all $k \in \N$ from standard elliptic regularity.
Assume now that the theorem holds for $n-1$ and for all $k$, and let $\bar u = u_{|B^*_1}$. By applying Theorem \ref{thm1} with $A = h {I}_d$, we obtain that $\bar u\in C^{1,\beta}(B_r^*)$ for all $\beta\in (0,1)$ and all $r\in(0,1)$. Moreover, $\bar u$ satisfies pointwise the boundary condition $\partial_{y_i} \bar u =0$ in $\partial B_r^* \cap \Sigma_i$ for all $i=1,\dots,n$. Thanks to this, we infer that
$u\in C^{1,\beta}(B_r)$, for all $\beta\in (0,1)$ and for all $r\in(0,1)$.

Now, fix one variable $y_i$ for $i=1,\dots,n$; without loss of generality, let us take $y_n$, and for simplicity set $y_n = \tau$. Call $\hat a = a - a_n e_n \in (-1, \infty)^{n-1}\times \{0\}$, so that $\omega^a = \omega^{\hat a }|\tau|^{a_n}$. Fix $\eps \in (0,1]$ and define
\begin{align*}
    \omega^a_\varepsilon= \omega^{\hat a} \rho_\varepsilon^{a_n} =  \omega^{\hat a} (\eps^2 + \tau^2)^{a_n/2}\,, \qquad \partial_\tau^{\varepsilon, a_n}u:=\rho^{a_n}_\varepsilon \partial_\tau u, \qquad \partial_\tau^{a_n}u:=|\tau|^{a_n}\partial _\tau u, \qquad\mathcal{G}u:=\tau^{-1}\partial_\tau u\,.
\end{align*}
Fix $0<r'<r<1$. By Lemma \ref{lemA.2}, there exists a sequence $u_\varepsilon$ of weak solutions to 
\begin{equation}\label{eq7.2eps}
    -\div(\omega_\varepsilon^a h \nabla u_\varepsilon)=\omega^a_\varepsilon f \quad\text{in }B_r
\end{equation}
Observe that we took $f_\varepsilon=f$, using that $f\in C^{1,\alpha}(B_1)$. Moreover, since $u$, $h$, and $f$, are all even in the $y$-variables, the same holds for $u_\eps$ (to see this, simply take $\xi$ to be radial in \eqref{eq:cutoffpropapprox}, and recall that $u_\eps$ solves \eqref{eq6.8}). Furthermore, by the inductive hypothesis on $n$, we have that $u_\varepsilon\in C^{3,\alpha}_{\loc}(B_r)$. As a consequence, we infer that $\partial^{\eps, a_n}_\tau u_\eps \in \tilde H^{1,a-a_ne_n,\eps}(B_r)$ (see Remark \ref{R:supsing}).

Next, fix $\phi\in C^\infty_c(B_r)$ and set 
\[
\psi:=\rho_\varepsilon^{-a_n}\partial_\tau \phi\,.
\]
Integrating by parts, we obtain
\begin{align*}
    &\int_{B_r} \omega^a_\varepsilon h\nabla u_\varepsilon\cdot \nabla \psi \, dz =\int_{B_r} \omega^{\hat a} h \nabla u_\varepsilon\cdot \nabla (\partial_\tau \phi)\,dz-a\int_{B_r}  \omega^{\hat a} \rho_\varepsilon^{-2} h \partial_\tau u_\varepsilon \partial_\tau \phi\, \tau\, dz\\
    &=-\int_{B_r}  \omega^{\hat a} \partial_\tau h \nabla u_\varepsilon\cdot \nabla \phi\, dz-\int_{B_r} \omega^{\hat a} h \nabla (\partial_\tau u_\varepsilon)\cdot \nabla \phi\,dz -a\int_{B_r}\omega^{\hat a} \rho_\varepsilon^{-2} h \partial_\tau u_\varepsilon \partial_\tau \phi\,\tau\,dz\\
    &=-\int_{B_r} \omega^{\hat a} \rho_\varepsilon^{-a_n}\partial_\tau^{\varepsilon,a_n}h \nabla u_\varepsilon \cdot \nabla \phi\, dz-\int_{B_r} \omega^{\hat a} \rho_\varepsilon^{-a_n} h\nabla(\partial_\tau^{\varepsilon,a_n} u_\varepsilon)\cdot \nabla \phi\, dz\,,
\end{align*}
and 
\begin{align*}
    \int_{B_r}\omega^a _\varepsilon f \psi\, dz=\int_{B_r}  \omega^{\hat a} f \partial_\tau \phi\, dz=-\int_{B_r}  \omega^{\hat a} \rho_\varepsilon^{-a_n}  \partial_\tau^{\varepsilon,a_n}  f\phi\, dz\,.
\end{align*}
Therefore, testing \eqref{eq7.2eps} with $\psi$ we deduce that $\partial_\tau^{\varepsilon,a_n} u_\varepsilon$ weakly satisfies
\begin{equation*}
\begin{cases}
    -\div( \omega^{\hat a} \rho_\varepsilon^{-a_n} h \nabla(\partial_\tau^{\varepsilon,a_n} u_\varepsilon))=\omega^{\hat a} \rho_\varepsilon^{-a_n}\partial_\tau^{\varepsilon,a_n}  f  + \div( \omega^{\hat a} \rho_\varepsilon^{-a_n}\partial_\tau^{\varepsilon,a_n}h \nabla u_\varepsilon) & \text{in }B_r\,,\\
    \partial_\tau^{\varepsilon,a_n} u_\varepsilon = 0 & \text{on }B_r \cap \Sigma_n\,.
    \end{cases}
\end{equation*}
We have the following estimates
\begin{align*}
\begin{split}
    &\int_{B_r}\omega^{\hat a}\rho^{-a_n}_\varepsilon (\partial_\tau^{\varepsilon,a_n}  f)^2\, dz=\int_{B_r} \omega^a_\varepsilon (\partial_\tau f)^2\, dz\leq c\,,\\
    &\int_{B_r}\omega^{\hat a}\rho^{-a_n}_\varepsilon(\partial_\tau^{\varepsilon,a_n}h )^2\, dz=\int_{B_r} \omega^a_\varepsilon (\partial _\tau h)^2\, dz\leq c\,,\\
    &\int _{B_r}\omega^{\hat a}\rho^{-a_n}_\varepsilon(\partial_\tau^{\varepsilon,a_n} u_\varepsilon)^2\, dz= \int_{B_r}\omega^a_\eps (\partial_\tau u_\varepsilon)^2\, dz\leq   \int_{B_r}\omega^a_\eps |\nabla u_\varepsilon|^2\,dz\leq c\,,
\end{split}
\end{align*}
where the last inequality is true by Lemma \ref{lemA.2}. Using these bounds together with the Caccioppoli inequality \eqref{eqwecacc} (see also Remark \eqref{R:cacsupsing}), we obtain
\[
\int _{B_{r'}}\omega^{\hat a}\rho^{-a_n}_\varepsilon((\partial_\tau^{\varepsilon,a_n} u_\varepsilon)^2 + |\nabla \partial_\tau^{\varepsilon,a_n} u_\varepsilon|^2)\, dz \leq c\,.
\]
We can then apply Lemma \ref{lemA.1}, in the version stated in Remark \ref{R:supsing}, and pass to the limit as $\eps \to 0$. We deduce that $\partial^{a_n}_\tau u\in \tilde H^{1, \hat a - a_ne_n}(B_{r'})$, is odd with respect to $\tau$, and satisfies
\begin{equation*}
\begin{cases}
    -\div( \omega^{\hat a} |\tau|^{-a_n} h \nabla(\partial_\tau^{a_n} u))=\omega^{\hat a} |\tau|^{-a_n}\partial_\tau^{a_n}  f + \div(\omega^{\hat a} |\tau|^{-a_n}(\partial_\tau^{a_n}h) \nabla u)  & \text{ in }B_{r'}\,,\\
    \partial_\tau^{a_n} u = 0 & \text{ on }B_{r'}\cap \Sigma_n\,.
\end{cases}
\end{equation*}
We now introduce the characteristic odd solution with respect to the variable $\tau$. Define
\begin{equation*}
    \Phi_{a_n}:=(1+a_n)\int_0^\tau|s|^{a_n}(h( \hat{z},s))^{-1}ds\,,
\end{equation*}
where $\hat{z}=(x,y_1,\dots,y_{n-1})$. It is straightforward to verify that $\Phi_{a_n}\in C^{0}(B_1) \cap C^1(B_1 \setminus \Sigma_n)$, that it is odd with respect to $\tau$, and that $\|\Phi_{a_n}\|_{H^{1,\hat a -a_ne_n}(B_1)}< \infty$. Hence $\Phi_{a_n} \in \tilde H^{1,\hat a-a_ne_n}(B_1)$. Furthermore, $\Phi_{a_n}$ satisfies
\begin{equation*}
\begin{cases}
    -\div( \omega^{\hat a} |\tau|^{-a_n} h \nabla \Phi_{a_n})=\div(\omega^{\hat a} |\tau|^{-a_n}h\nabla_{\hat z}\,\Phi_{a_n}) & \text{ in }B_1\,,\\
    \Phi_{a_n} = 0 & \text{ on }B_1\cap \Sigma_n\,,
    \end{cases}
\end{equation*}
where $\nabla_{\hat z}\,\Phi_{a_n}:=\nabla\Phi_{a_n}-(\partial_\tau \Phi_{a_n})e_\tau $. Next, define the quotient 
\[
w:=\Phi_{a_n}^{-1}\partial_\tau^{a_n} u\,.
\]
Let $\phi_j \in C^\infty_c(B_{r'}\setminus\Sigma_n)$ be such that $\phi_j \to \partial_\tau^{a_n} u$ in $\tilde H^{1,\hat a -a_ne_n}(B_{r'})$. 
From \eqref{eq7.3} we have the estimates
\begin{equation}\label{eq:Phibound}
0 < c|\tau|^{a_n+1} \leq |\Phi_{a_n}| \leq C|\tau|^{a_n + 1}\,, \qquad  |\nabla \Phi_{a_n}| \leq c |\tau|^{a_n}\,.
\end{equation}
Thus we have
\[
\begin{aligned}
\|w- \Phi_{a_n}^{-1}\phi_j\|_{H^{1,\hat a + (2 + a_n)e_n}(B_{r'})}^2 \qquad \qquad \qquad& \\
\leq c\int_{B_{r'}}\omega^{\hat a}|\tau|^{-a_n}\Big((\partial_\tau^{a_n} u& - \phi_j)^2 + |\nabla(\partial_\tau^{a_n} u - \phi_j)|^2 +|\tau|^{-2}(\partial_\tau^{a_n} u - \phi_j)^2\Big)\,dz \,.
\end{aligned}
\]
From Remark \ref{Hardyrmrk} it follows that 
\[
\|w - \Phi_{a_n}^{-1}\phi_j\|_{H^{1,\hat a + (2 + a_n)e_n}(B_{r'})} \leq c\|\partial_\tau^{a_n} u - \phi_j\|_{ H^{1,\hat a -a_ne_n}(B_{r'})}\to 0\,.
\]
Thus, up to a standard mollification argument, we conclude that $w \in \tilde H^{1,\hat a + (2+  a_n)e_n}(B_{r'})$.

Define 
\[
\varphi:=(\tau|\tau|^{a_n})^{-1}\Phi_{a_n}\,.
\]
Then $w\in \tilde H^{1, \hat a + (2 + a_n)e_n}(B_{r'})$ is an even weak solution to 
\begin{align}\label{eq7.12}
\begin{split}
    -\div(&\omega^{\hat a} |\tau|^{2+a_n} \varphi^2 h \nabla w)=\div(\omega^{\hat a} |\tau|^{2+a_n}[\varphi \,(\mathcal{G}h)\nabla u-h (\mathcal{G}u)(\tau|\tau|^{a_n})^{-1}\nabla_{\hat z}\,\Phi_{a_n}]\\
    +&\omega^{\hat a} |\tau|^{2+a_n}[\varphi \mathcal{G}f-(\mathcal{G}h)(\tau|\tau|^{a_n})^{-1}\nabla u \cdot \nabla \Phi_{a_n}+h(\tau|\tau|^{a_n})^{-2} \nabla_{\hat z}\,\Phi_{a_n}\cdot \nabla_{\hat z}(\partial_\tau^{a_n}u)]\qquad \text{in }B_{r'}.
\end{split}
\end{align}
To verify this, argue as in \cite[Lemma 2.4]{SirTerVit21b}. In particular, since the weight $|\tau|^{2 + a_n}$ is superdegenerate, Lemma \ref{lemH=W} ensures that it suffices to test  \eqref{eq7.12} against functions $\phi\in C^\infty_c(B_{r'}\setminus \Sigma_n)$.

Using the elementary identities
\[
\begin{aligned}
&(\tau|\tau|^{a_n})^{-1}\nabla_{\hat z}\Phi_{a_n} = \nabla_{\hat z}\varphi\,,\\
&(\tau|\tau|^{a_n})^{-1} \nabla_{\hat z} (\partial_\tau^{a_n}u) = \nabla_{\hat z} (w \varphi) = \varphi \nabla_{\hat z}w + w \nabla_{\hat z}\varphi\,,\\
&(\tau|\tau|^{a_n})^{-1}\nabla u \cdot \nabla \Phi_{a_n}=\nabla_{\hat z} u  \cdot \nabla_{\hat z}\,\varphi+h^{-1}\mathcal{G}u\,,
\end{aligned}
\]
we can rewrite \eqref{eq7.12} in the more compact form
\begin{align}\label{eq7.15}
    \begin{split}
 -\div( \omega^{\hat a} &|\tau|^{2+a_n} \varphi^2 h \nabla w)=\div(\omega^{\hat a} |\tau|^{2+a_n} \tilde F)+   \omega^{\hat a} |\tau|^{2+a_n} \tilde f+ \omega^{\hat a} |\tau|^{2+a_n} b \cdot \nabla w \quad \text{in }B_{r'}\,,
    \end{split}
\end{align}
where
\[
\begin{aligned}
&\tilde f = \varphi \mathcal{G}f-(\mathcal{G}h)\nabla_{\hat z} u  \cdot \nabla_{\hat z}\,\varphi -h^{-1}(\mathcal{G}u)(\mathcal{G}h) +hw|\nabla_{\hat z}\varphi|^2\,, \\
&\tilde F =  \varphi \,(\mathcal{G}h)\nabla u-h (\mathcal{G}u)\nabla_{\hat z}\,\varphi\,,\\
&b = h\varphi\nabla_{\hat z}\varphi\,.
\end{aligned}
\]
By \cite[Lemma 2.4]{TerTorVit24}, we have $\varphi\in C^{2,\alpha}(B_1)$. Moreover, in view of \eqref{eq:Phibound}, there exists constants $0 < c \leq C$ such that $c\leq \varphi\leq C$ in $B_1$.
Since $h$ is even in $\tau$, \cite[Lemma 2.3]{TerTorVit24} implies that $\mathcal{G}h\in C^{0,\alpha}(B_1)$. Furthermore, as $u$ belong to $C^{1,\beta}_\loc(B_1)$ for all $\beta\in(0,1)$ and is even in $\tau$, it follows that
\[
|\mathcal{G}u|\leq c|\tau|^{\beta-1} \qquad \text{for all } \beta\in (0,1).
\]
In fact, since $w = \varphi^{-1}\mathcal{G} u$, the same estimate also holds for $w$. Therefore, also in view of hypotheses \eqref{eq:hofh} on $f$, 
\[
\mathcal{G}f, \mathcal{G}u, w\in L^{p_n}(B_{r'},\omega^{\hat a} |\tau|^{2+a_n}dz), \qquad \text{with }\,\,p_n=\frac{d+\langle (a+2e_n)^+\rangle}{1-\alpha}
\]
Finally, as $\varphi \in C^{2,\alpha}(B_1)$, we have $\nabla_{\hat z}\varphi\in C^{1,\alpha}(B_1)$, and hence
\[
\tilde f \in L^{p_n}(B_{r'},\omega^{\hat a} |\tau|^{2+a_n}dz), \qquad \tilde F\in (L^{p_n}(B_{r'},\omega^{\hat a} |\tau|^{2+a_n}dz) )^d, \qquad b \in (C^{1,\alpha}(B_1))^d\,.
\]
Then, applying Theorem \ref{thm1} (see also Remark \ref{R:drift}) to equation \eqref{eq7.15}, we obtain $w\in C^{0,\alpha}_{\loc}(B_{r'})$, and therefore $\mathcal{G}u=w\varphi\in  C^{0,\alpha}_{\loc}(B_{r'})$. As a consequence, the datum $\tilde F$ in \eqref{eq7.15} inherits the same regularity, i.e., $\tilde F \in C^{0,\alpha}_\loc(B_{r'})$. A second application of Theorem \ref{thm1} then gives $w,\mathcal{G}u\in C^{1,\alpha}_\loc(B_{r'})$.

Repeating this argument for each coordinate $y_i$, we get that $\mathcal{G}_iu:={\partial_{y_i}u}/{y_i}\in C^{1,\alpha}_{\loc}(B_{r'})$ for all $i=1,\dots,n$. Now, for $\eps \in (0,1]$ define
\[
\omega_\eps^a = \prod_{i=1}^n (\eps^2 + y_i^2)^{a_i/2}.
\]
Let us fix $\phi \in C^\infty_c(B_{r'})$, and test equation \eqref{eq:balleq} with $\omega^{-a}_\eps\phi \in C^\infty_c(B_{r'})$. We obtain
\[
\int_{B_{r'}} (\omega^a\omega_\eps^{-a}) h \nabla u\cdot \nabla\phi \,dz = \int_{B_{r'}} (\omega^a\omega_\eps^{-a}) f \phi \,dz + \int_{B_{r'}} (\omega^a\omega_\eps^{-a}) h \big(\sum_{i=1}^na_i \frac{y_i^2}{\eps^2 + y_i^2}\mathcal{G}_i u\big)\phi \,dz\,.
\]
Since all functions involved are sufficiently smooth, and 
\[
\omega^a\omega_\eps^{-a} \leq c \prod_{-1<a_i<0}|y_i|^{a_i} \in L^1(B_1)\,, 
\]
we can pass to the limit as $\eps \to 0$. It follows that $u \in C^{1,\alpha}_\loc(B_1) \subset H^{1}_\loc(B_1)$ is a weak solution to 
\begin{equation*}
   -\div(h\nabla u)=f + h\sum_{i=1}^n a_i\mathcal{G}_i u \in C^{1,\alpha}_\loc(B_{r'}).
\end{equation*}
By the classical Schauder estimates for elliptic equations, we then conclude that $u\in C^{3,\alpha}_\loc(B_{r'})$. Since $r'>0$, the thesis follows.
\end{proof}

\begin{Corollary}[Application to monomial CKN]\label{corapplCKN}
Let $a \in [0, \infty)^n$, $0\le b-q<1$ and $ p = \frac{2(d+|a|)}{d+|a|-2+2(b-q)}$. Let $u$ be a optimizer of 
\eqref{eqmonCKN}. Then $u\in C^{3, \alpha}_\loc(\overline{\R^d_*}\setminus\{0\})$ for all $\alpha \in (0,1)\cap (0, p-2]$. 
\end{Corollary}
\begin{proof}
We refer to \cite{Paglia25} for the precise definition of an optimizer for \eqref{eqmonCKN}. For our purposes, it suffices to recall that, under the assumptions on $a, b, p, q$ stated above, such an optimizer $u$ exists, belongs to $H^{1,a}_\loc(\R^d_* \setminus \{0\})$, and solves the Euler-Lagrange equation (up to a multiplicative constant in the right hand side)
\[
\begin{cases}
-\div(\omega^a |z|^{-2q} \nabla u) = \omega^a |z|^{-bp}|u|^{p-2}u & \text{in }\R^d_* \setminus \{0\}\,,\\
\omega^a \partial_{y_i} u = 0  & \text{on }(\partial\R^d_*\cap \Sigma_i) \setminus \{0\}, \text{ for }i= 1, \ldots, n\,.
\end{cases}
\]
Fix $z_0 \in \overline{\R^d_* }\setminus  \{0\}$ and choose $R = R(z_0) >0$ such that $0 \not \in \overline{B_R(z_0)}$. Then $u$ is a solution to 
\[
\begin{cases}
-\div(\omega^a h \nabla u) = \omega^a g|u|^{p-2}u & \text{in }B^*_R(z_0)\,,\\
\omega^a \partial_{y_i} u = 0  & \text{on }\partial B^*_R(z_0)\cap \Sigma_i, \text{ for }i= 1, \ldots, n\,.
\end{cases}
\]
where $h = (|z|^{-2q})_{|B_R(z_0)}$ and $g = (|z|^{-bp})_{|B_R(z_0)}$ satisfy \eqref{eq7.3}.

Fix $0<R_3<R_2<R_1<R$. First, using the Sobolev inequality in Lemma \ref{L:Sobolev}, one can perform a standard Moser-type iteration (see Lemma \ref{thmdegiorgi}) and get that $u \in L^\infty(B^*_{R_1}(z_0))$.

Now we apply Theorem \ref{thm1} with $f = g|u|^{p-2}u \in L^\infty (B^*_{R_1}(z_0)) $, obtaining $u \in C^{1,\beta}(B^*_{R_2}(z_0))$ for all $\beta \in (0,1)$, and the pointwise Neumann condition $\partial_{y_{i}} u = 0$ on $ \partial B^*_{R_2}(z_0)\cap \Sigma_i $ for all $i = 1, \ldots, n$.

As a consequence,  this implies that $|u|^{p-2}u \in C^{1, \alpha}(B^*_{R_2}(z_0)) $ for $\alpha\in(0,1)\cap(0,p-2]$. Moreover, it is easy to check that $f = g|u|^{p-2}u$ satisfies the assumption \eqref{eq7.3}, which allows us to apply Theorem \ref{thm7.1}. Then, $u\in C^{3,\alpha}(B_{R_3}^*(z_0))$ and the conclusion follows since $R_3>0$ is arbitrary. 
\end{proof}

\subsection{Smoothness of solutions to the isotropic homogeneous equation}

We conclude with a simple result showing that, in the most basic setting, solutions are indeed smooth.
\begin{Theorem}[Smoothness of \emph{degenerate harmonic} functions]\label{TheoSmooth}
    Let $a\in (-1,\infty)^n$, and let $u\in H^{1,a}(B_1^*)$ be a weak solution to \eqref{eq7.2} with $f=0$ and $h\equiv 1$. Then $u\in C^\infty(B_r^*)$, for every $r\in(0,1)$.
\end{Theorem}
\begin{proof}
Following the same argument as in the proof of Theorem \ref{thm7.1}, we find that, for each $i = 1,\ldots, n$, the function $w_i = {\partial_{y_i}u}/{y_i}$ is a weak solution to
\begin{equation*}
    -\dive( \omega^a |y_i|^{2}\nabla w_i)=0\,.
\end{equation*}
Hence, by Theorem \ref{thm1}, we have $w_i\in C^{1,\alpha}_{\loc}(B_1)$ for all $\alpha \in (0,1)$. From \eqref{eq7.2} it then follows that 
\begin{equation*}
    -\Delta u=\sum_{i=1}^n a_i w_i\in C^{1,\alpha}_{\loc}(B_1)\,.
\end{equation*}
Therefore, classical Schauder's theory implies $u\in C^{3,\alpha}_{\loc}(B_1)$. In fact, we get that every solution to problems in the form \eqref{eq7.2} belongs to $C^{3,\alpha}_{\loc}(B_1)$. In particular, $w_i \in C^{3,\alpha}_{\loc}(B_1)$. We can then iterate this bootstrapping procedure, obtaining $u \in C^\infty_\loc(B_1)$, which concludes the proof.
\end{proof}

\section*{Acknowledgements}
The authors are research fellows of Istituto Nazionale di Alta Matematica INDAM group GNAMPA. 
G.C., G.F. and S.V. are supported by the GNAMPA project E5324001950001 \emph{PDE ellittiche che degenerano su variet\`a di dimensione bassa e frontiere libere molto sottili}. G.C. is supported by the PRIN project 20227HX33Z \emph{Pattern formation in nonlinear phenomena}. S.V. is supported by the PRIN project 2022R537CS \emph{$NO^3$ - Nodal Optimization, NOnlinear elliptic equations, NOnlocal geometric problems, with a focus on regularity}.
F.P. is supported by the European Union's Horizon Europe research and innovation programme under the Marie Sklodowska-Curie grant agreement No 101126554.
\section*{Disclaimer}
Co-Funded by the European Union. Views and opinions expressed are however those of the author only and do not necessarily reflect those of the European Union. Neither the European Union nor the granting authority can be held responsible for them.

\end{document}